\documentclass[a4paper,twoside,11pt]{book}

\usepackage[english]{babel}
\usepackage[T1]{fontenc}
\usepackage[ansinew]{inputenc}
\usepackage{geometry}
\usepackage{color}
 \usepackage{amsmath}
 \numberwithin{equation}{section}
\usepackage{lmodern} 

\usepackage[automark,headsepline,ilines,komastyle]{scrpage2} 
\ifoot{}\ofoot{}
\usepackage{graphicx}
\usepackage{amsmath}
\usepackage{amsthm}
\usepackage{amsfonts}
\usepackage{amssymb}
\usepackage{graphics}
\usepackage{color}
\usepackage{enumerate}
      \newtheorem{theorem}{Theorem}[section]
\theoremstyle{break} \newtheorem{satz}{Theorem}

      \newtheorem{lemma}[theorem]{Lemma}
      \newtheorem{corollary}[theorem]{Corollary}

      \newtheorem{definition}[theorem]{Definition}
        \newtheorem{proposition}[theorem]{Proposition}  \newtheorem{conj}[theorem]{Conjecture}        \newtheorem{question}[theorem]{Question}
\newtheorem{remark}[theorem]{Remark}
      \theoremstyle{remark}

      \newtheorem{example}[theorem]{Example}

\definecolor{orange}{rgb}{1,0.5,0}

\newcommand\proofsymbol{\frame{\rule[0pt]{0pt}{8pt}\rule[0pt]{8pt}{0pt}}}

\newcommand\with{\ \vrule\ }  
\newcommand\C{\mathbb C}         
\newcommand\R{\mathbb R}         
\newcommand\N{\mathbb N}         
\newcommand\Ha{\mathbb H}        
\newcommand\D{\mathbb D}         
\newcommand\Do{\mathcal{D}_{n}}         
\renewcommand\i{\operatorname{i}}        
\renewcommand\Im{\operatorname{Im}}        
\renewcommand\Re{\operatorname{Re}}        
\newcommand\aut{\operatorname{Aut}}
\newcommand\id{\operatorname{\bf id}}

      \newcommand{\eps}{\varepsilon}
      \newcommand{\lip}{\operatorname{Lip_{left}}\left(\frac{1}{2}\right)}
   \newcommand{\Lip}{\operatorname{Lip}\left( \frac{1}{2}\right)}
\newcommand{\LandauO}{\mathcal{O}} 
\newcommand{\Landauo}{{\scriptstyle\mathcal{O}}}
\newcommand{\diam}{\operatorname{diam}}
\newcommand{\dist}{\operatorname{dist}}
\newcommand{\cara}{\overset{Cara}{\longrightarrow}}
\newcommand{\supp}{\operatorname{supp}}
\newcommand{\hcap}{\operatorname{hcap}}

\newcommand{\B}{\mathbb B}
\newcommand{\I}{\mathcal{I}} \newcommand{\M}{\mathcal{M}}
\newcommand{\Po}{\mathcal{P}} \newcommand{\Ho}{\mathcal{H}}

\usepackage{hyperref}
\hypersetup{pdfpagemode={UseOutlines},bookmarksopen=true,
   bookmarksopenlevel=0,bookmarksnumbered=true,hypertexnames=false,
   colorlinks,linkcolor={blue},citecolor={blue},urlcolor={red},
   pdfstartview={FitV},breaklinks=true}

\usepackage{float}

\begin{document} 
\parindent 0pt 

\setcounter{section}{0}

\begin{titlepage}


\begin{center}
\medskip
{{\LARGE {\bf Embedding Problems in Loewner Theory}}}\\

\medskip
\end{center}

\vspace*{2cm}
\begin{center}
Dissertationsschrift zur Erlangung des
naturwissenschaftlichen\\[1ex]
Doktorgrades der Julius-Maximilians-Universit\"at W\"urzburg

\end{center}
\vspace*{1cm}
\begin{figure}[h]
    \centering
   \includegraphics[width=60mm]{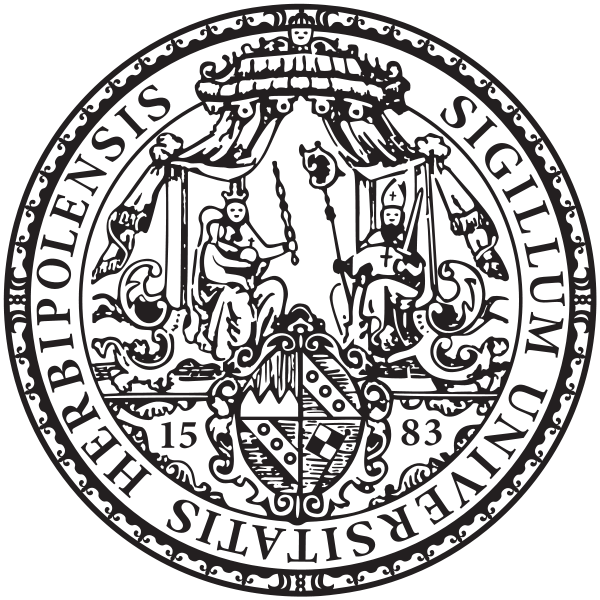}
 \end{figure}

\vspace*{1cm}

\begin{center}
vorgelegt von\\[4ex]
{ {\large Sebastian Schlei\ss inger}}\\[4ex]
aus\\[4ex]
W\"urzburg, Deutschland
\end{center}

\vspace*{2cm}

\begin{center}
{W\"urzburg 2013}
\end{center}

\end{titlepage}

\newgeometry{left=24mm,right=24mm, top=-5mm, bottom=25mm} 
\tableofcontents
 \newgeometry{left=24mm,right=24mm, top=30mm, bottom=35mm} 
\newpage
\section*{Notation}
 
\begin{tabular}{ll}
$\N$ & $:=\{1,2,3,...\}$, the set of positive integers\\[1mm]
$\N_0$ & $:=\N\cup\{0\}$, the set of non-negative integers\\[1mm]
$\R$ & the set of real numbers\\[1mm]
$\C$ & the set of complex numbers\\[1mm]
$\hat{\C}$ & := $\C\cup \{\infty\}$, the extended complex plane\\[1mm]
$\mathcal{B}(z,r)$ & :=$\{w\in\C\;|\; |w-z|<r\}$, the open Euclidean disc \\[1mm]
${}$ & with radius $r>0$ and center $z\in\C$\\[1mm]
$\D$  & := $\mathcal{B}(0,1)$, the (Euclidean) unit disc\\[1mm]
$\B_n$ & :=$\{(z_1,...,z_n)\in \C^n \;|\; \sum_{k=1}^n |z_k|^2<1\}$, the Euclidean unit ball\\[1mm]
$\mathcal{D}_n$ & either $\B_n$ or the polydisc $\D^n=\underbrace{\D\times...\times\D}_{\text{$n$ times}}$\\[1mm]
$\Ho(D,E)$ & the set of analytic functions defined in a domain $D\subset\C^n$ and mapping \\[1mm]
${}$ &   into the domain $E\subset \C^m$\\[1mm]
$U(\mathcal{D}_n)$ & :=$\{f\in \Ho(\mathcal{D}_n,\C^n) \;|\; f \; \text{is univalent}\}$\\[1mm]
$S(\mathcal{D}_n)$ & :=$\{f\in \Ho(\mathcal{D}_n,\C^n) \;|\; f \; \text{is univalent and}\; f(0)=0, Df(0)=I_n\}$\\[1mm]
$S^*(\mathcal{D}_n)$ & := the set of all $f\in S(\mathcal{D}_n)$ such that $f(\mathcal{D}_n)$ is starlike with respect to $0$\\[1mm]
$\aut(\C^n)$ & :=$\{f\in \Ho(\C^n,\C^n) \;|\; f \; \text{is biholomorphic}\}$\\[1mm]
$\textbf{id}_{D}$ & := the identity $\textbf{id}_{D}:D\to D,$ $\textbf{id}_D(z)=z$ for all $z\in D$, where $D$ is an arbitrary domain\\[1mm]
$\mathcal{I}(\mathcal{D}_n)$ & := the set of infinitesimal generators of semigroups in $\mathcal{D}_n$\\[1mm]
$E(\mathcal{D}_n)$ & := the set of all $f\in U(\mathcal{D}_n)$ that can be embedded into a Loewner chain with range $\C^n$\\[1mm]
$F(\mathcal{D}_n)$ & := the set of all $f\in U(\mathcal{D}_n)$ such that $f(\mathcal{D}_n)$ is a Runge domain\\[1mm]
$\widetilde{E}(\mathcal{D}_n)$ & := the set of all $f\in U(\mathcal{D}_n)$ that can be embedded into a Loewner chain whose range\\[1mm]
${}$ & is biholomorphic to $\C^n$\\[1mm]
$\dot{f}_t$ & :=$\frac{\partial f_t}{\partial t},$ the partial derivative with respect to the variable $t$\\[1mm]
$D{f}_t$ & := the partial derivative with respect to the variable $z\in \C^n$\\[1mm]
$S^0(\mathcal{D}_n)$ & the set of all $f\in S(\mathcal{D}_n)$ having parametric representation\\[1mm]
$\mathbb{H}$ & $:=\{x+iy\in\C \;|\; y>0\}$, the upper half-plane\\[1mm]
$\mathcal{H}_\infty$ & the set of all holomorphic mappings $f:\Ha\to\Ha$ with hydrodynamic normalization\\[1mm]
$\mathcal{H}_u$ & the set of all univalent $f\in\mathcal{H}_\infty$\\[1mm]
$\mathcal{H}_b$ & the set of all $f\in\mathcal{H}_u$ such that $\Ha\setminus f(\Ha)$ is bounded\\[1mm]
$g_A$ & the unique conformal mapping from $\Ha\setminus A$ onto $\Ha$ with hydrodynamic normalization,\\[1mm]
${}$ & where $A$ is a hull\\[1mm]
$\hcap(A)$ & the half-plane capacity of a hull $A$\\[1mm]
$C([a,b],\R)$ & the set of all continuous functions $f:[a,b]\to\R,$ $a,b\in\R$ with $a<b$\\[1mm]
$\diam(A)$ & :=$\displaystyle \sup_{z,w\in A}|z-w|,$ the diameter of a subset $A\subset \C$\\[1mm]

\end{tabular}

\newpage

\chapter{Introduction to Loewner theory}

Loewner theory arose from an article by C. Loewner published in 1923 (see \cite{Loewner:1923}) and has become a powerful tool in complex analysis. It consists of the study of biholomorphic mappings by certain differential equations. These equations make it possible to use various techniques, e.g. from control theory, in order to attack (extremal) problems concerning biholomorphic functions, e.g. the famous Bieberbach conjecture.\\
Conversely, the Loewner differential equations can be used to provide relatively simple models for the description of planar growth processes, such as Hele-Shaw flows. O. Schramm realized that, in particular, those models are extremely useful for certain stochastic growth models of curves in the plane. He invented the stochastic Loewner evolution (or Schramm-Loewner evolution, SLE, see \cite{MR1776084}), which turned out to have striking applications, e.g. in statistical physics. O. Schramm, G. Lawler and W. Werner investigated SLE in a series of articles. Among other things, they used SLE to prove the Mandelbrot conjecture\footnote{The Mandelbrot conjecture states that the boundary of a two-dimensional Brownian motion has fractal dimension $4/3.$ }, see \cite{MR1849257}. Recently, two Fields medals were awarded to W. Werner (2006) and S. Smirnov (2010) for their contributions to SLE and related fields.\\

From a geometric or function theoretic point of view, Loewner's differential equations arise from a family $\{D_t\}_{t\geq 0}$ of increasing (or decreasing) subdomains of $\C$ (or more generally: from a complex manifold like $\C^n$), i.e. $D_t \subseteq D_s$ for $t\leq s$. Suppose that all such domains are biholomorphically equivalent to $\D$, and, for every $t\geq0,$ denote by $f_t:\D\to D_t$ a conformal map. Under some conditions for the family $\{f_t\}_{t\geq0}$, the map $t\to f_t$ satisfies a differential equation, which allows an effective analytic description of the growth process of the domains. \\
In fact, this construction still works for families of domains which are not biholomorphically equivalent to one fixed domain. The case of an increasing family of multiply connected domains was considered by Y. Komatu in 1943, see \cite{KomatuZweifach} and \cite{Komatu}, and leads to so-called Komatu-Loewner equations; see also \cite{MR2357634}, \cite{annulus1}, \cite{annulus2} for some recent references.\\

In this work, we will only deal with the simpler case where all domains are biholomorphically equivalent to one domain $D$. In Chapter 2 we will look at the several variable cases where $D$ is either the Euclidean unit ball or the polydisc and in Chapter 3 we will consider the case $D=\D.$\\
In order to give some motivation for the problems we will talk about, we now give a short (and certainly not comprehensive) historical overview over the classical Loewner equations as they have been considered by Loewner, Pommerenke, Kufarev, et al. \\
Except for Example \ref{lion}, we will not discuss SLE at all, as our work is only related to the deterministic setting and because this topic with all its applications is too extensive to be explained in a few lines.

\section{The radial Loewner equation and the Bieberbach conjecture}
One of the most famous and useful results of complex analysis is the Riemann Mapping Theorem: every simply connected domain $\Omega\subsetneq \C$ is conformally equivalent to the unit disc $\D:=\{z\in\D\with |z|<1\},$ i.e. there exists a conformal (holomorphic and bijective) map $f:\D \to\Omega.$ This mapping is unique if we require the normalization $f(0)=z_0,$ $f'(0)>0,$ for some $z_0\in\Omega.$ Thus the set of all simply connected, proper subdomains of $\C$ can be described (modulo translation and scaling) by the class $S$ of normalized univalent functions:
$$S:=\{f\in\Ho(\D,\C)\with f(0)=0,\; f'(0)=1,\; f\; \text{univalent}\}.$$
Now it is of interest to find connections between analytic properties of a function $f\in S$ and geometric properties of its image domain $f(\D).$ Many results of this kind have been established since the beginning of the 20th century. Exemplarily, we mention the characterization of normalized starlike mappings: Denote by $S^*\subset S$ the set of all $f\in S$ that map $\D$ onto a domain that is starlike with respect to 0\footnote{A domain $D\subset \C$ with $0\in D$ is called \emph{starlike w.r.t. 0} if $z\in D$ implies $rz\in D$ for all $r\in[0,1].$}. Now the following characterization of $S^*$ is known: Let $f\in\Ho(\D,\C)$ with $f(0)=0$, $f'(0)=1.$ Then $f\in S^*$ if and only if 
\begin{equation}\label{star} \Re \left(\frac{f(z)}{zf'(z)}\right)>0 \qquad \text{for all} \; z\in\D. \end{equation} Every $f\in S$ can be described by its power series expansion in $0:$ If we write $f(z)=z+a_2z+a_3z^2+...$, then all properties of $f$ are encoded in the coefficients $a_2,a_3,...$  \\
In 1909, P. Koebe proved that $S$ is compact with respect to the topology induced by locally uniform convergence. Thus, for every $n\geq 2,$ there exists a universal bound for $|a_n|$ w.r.t. the class $S.$ In 1916, L. Bieberbach proved that $|a_n|\leq 2$ and he conjectured  that
$$|a_n|\leq n \qquad \text{for all} \; n\geq 2.$$
C. Loewner found a proof for the case $n=3$ in 1923, see \cite{Loewner:1923}. His article represents an important step towards a proof of the Bieberbach conjecture and a remarkable contribution to complex analysis in general, as he introduced a new method that has been extended to what is now called \emph{Loewner theory}. Loewner was inspired by the work of S. Lie and his idea was to describe conformal mappings as ``a concatenation of infinitesimal conformal mappings''. More precisely, this means to represent a conformal mapping as (the first element of) the solution of a non-autonomous version of the differential equation for one real-parameter semigroups of holomorphic self-mappings of $\D$ (or any other domain).\\  In the following we briefly describe the setting of Loewner's method, in which two differential equations appear, an ordinary and a partial differential equation for conformal mappings that fix the origin. Because of this normalization, these equations are nowadays called radial Loewner ODE and radial Loewner PDE. Major contributions to the theory following Loewner's original setting were made by C. Pommerenke.\\
A mapping $f:\D \times [0,\infty) \to \C$ is called \emph{normalized Loewner chain} if $f(\cdot,t)$ is univalent with $f(0,t)=0$, $f'(0,t)\footnote{By $f'$ we denote the partial derivative $\partial f/\partial z.$ }=e^t$ for all $t\geq0$ and $f(\D,s)\subset f(\D,t),$ whenever $0\leq s\leq t < \infty.$ \\
 In particular, the function $f(\cdot,s)$ is subordinated to $f(\cdot,t)$ when $s\leq t$, and we can write $$f(\cdot,s)=f(\varphi_{s,t}(\cdot),t)$$ where $z\mapsto \varphi_{s,t}(z)$ is a univalent mapping defined in $\D$ and bounded by $1$.  The ``transition mappings'' $\varphi_{s,t}$ have the normalization $\varphi_{s,t}(z)=e^{s-t}z+...$ and from the definition it follows that
$$ \varphi_{s,\tau} = \varphi_{t,\tau}\circ \varphi_{s,t} \qquad \text{for all}\; 0\leq s\leq t\leq \tau<\infty. $$
This property shows that the family $\{\varphi_{s,t}\}_{0\leq s\leq t}$, nowadays called \emph{evolution family}, is connected to semigroups of analytic self-mappings of $\D.$ Now, the following differential equations hold 
  \begin{eqnarray}
 \dot{f}(z,t)=&f'(z,t)\cdot zp(z,t) \qquad &\text{(Loewner PDE)}, \label{PDE}\\
\dot{\varphi}_{s,t}(z)=& -\varphi_{s,t}(z)\cdot p(\varphi_{s,t}(z),t) \qquad &\text{(Loewner ODE)}, \label{ODE}
  \end{eqnarray}
where $p(z,\cdot)$ is measurable for every $z\in\D$ and the function $p(\cdot,t)$ is analytic and maps $\D$ into the right half-plane with $p(0,t)=1$ for every $t\geq 0.$ Because of the normalizations of $f(z,t)$ and $\varphi_{s,t},$ these equations are called \emph{radial} Loewner equations.  

\begin{example}\label{infinf}
Let $f\in S^*.$ Then $f(z,t):=e^t\cdot f(z)$ is a normalized Loewner chain and it satisfies (\ref{PDE}) with $p(z,t)=p(z):= \frac{f(z)}{z\cdot f'(z)}.$ Note that $\Re(p(z))>0$ for all $z\in \D $ by relation (\ref{star}).\\
Conversely, one can show that if $p(z,t)$ does not depend on $t,$ then a normalized Loewner chain satisfying (\ref{PDE}) has the property $f(\cdot,t)=e^t\cdot f(\cdot,0)$. From this it follows immediately that $f(\cdot,0)\in S^*.$ \hfill $\bigstar$ 
\end{example}

One of the main problems in Loewner theory is to find connections between the function $p(z,t)$ and the corresponding evolution families and Loewner chains that satisfy the differential equations. Moreover, Loewner's description of conformal mappings makes it possible to attack various problems by the powerful methods of control theory.\\

Pommerenke showed that every $f\in S$ can be embedded into normalized Loewner chains, i.e. for every $f\in S$ there exists a normalized Loewner chain $f(z,t)$ with $f=f(\cdot,0).$ Loewner showed that this is true for all slit mappings and that the Loewner chain is uniquely determined in this case, see \cite{Loewner:1923}.
\begin{satz}\label{Michelle}
Let $f\in S,$ such that $f(\D)=\C\setminus \Gamma,$ where $\Gamma$ is a Jordan arc. Then there exists a uniquely determined normalized Loewner chain $f(z,t)$ with $f(\cdot ,0)=f$ and  
\begin{equation}\label{Obama} p(z,t) = \frac{1+\kappa(t)z}{1-\kappa(t)z}, \end{equation}
where $\kappa:[0,\infty)\to \partial \D$ is a continuous function, called \emph{driving function}.
\end{satz}

\begin{remark}
Clearly, the Loewner chain is uniquely determined as there is only one way to erase the slit $\Gamma$. Let us have a look at another formulation of  Theorem \ref{Michelle}:\\[0.2cm]
There exists exactly one parameterization $\gamma:[0,\infty)\to \Gamma$ of the slit $\Gamma$ such that the normalized conformal mappings $f_t:\D\to \C\setminus \gamma[t,\infty)$ satisfy the Loewner PDE (\ref{PDE}) where $p(z,t)$ has the form (\ref{Obama}). \\[0.2cm]
Thus, even though there is no further assumption on the regularity of the slit $\Gamma$, there always exists a parameterization such that the family $\{f_t\}_{t\geq 0}$ of conformal mappings is differentiable with respect to $t$. In some sense, this problem is related to Hilbert's fifth problem of finding differentiable structures for continuous groups, see \cite{Goodman}.
\end{remark}

Note that the set of all $f\in S$ that map $\D$ onto $\C$ minus a slit is \emph{dense} in $S.$ In view of this, Theorem \ref{Michelle} is quite useful, as certain extremal problems can be reduced to the study of slit mappings, which in turn can be described by those Loewner equations where $p(z,t)$ has the simple form (\ref{Obama}). \\

The Bieberbach conjecture was finally proven by L. de Branges in 1984. He used ideas related to Loewner's and methods from functional analysis. FitzGerald and Pommerenke gave a simpler, purely function theoretic proof in \cite{MR792819}. 

\section{The chordal case}

Besides Pommerenke, several Soviet mathematicians, especially P. Kufarev, made important contributions to Loewner theory, see, e.g., \cite{MR0013800,  MR0023907, MR0257336}. In particular, they considered a differential equation for certain univalent mappings $f:\Ha\to\Ha,$ where we denote by $\Ha$  the upper half-plane $\Ha=\{z\in\C\with \Im(z)>0\}$. Instead of fixing an interior point, as in the class $S$, they considered mappings with \emph{hydrodynamic normalization}, i.e. $f(z)=z-\frac{c}{z}+\Landauo(\frac1{z})$ for a $c\geq 0$ in an angular sense. In other words, the following limit should exist and be finite:
$$\angle\lim_{z\to\infty}z\left(f(z)-z\right)=c\geq 0.$$

If $\Gamma \subset \overline{\Ha}$ is a slit, i.e. a simple curve such that $\Ha\setminus \Gamma$ is simply connected, then there exists exactly one conformal mapping $f_\Gamma:\Ha\to \Ha\setminus \Gamma$ with hydrodynamic normalization $f_\Gamma(z)=z-\frac{c}{z}+\Landauo(\frac1{z})$ and $c=:\hcap(\Gamma)$ is called the half-plane capacity of $\Gamma$. Kufarev et al. proved the following counterpart to Theorem \ref{Michelle}, see \cite{MR0257336}.
\begin{satz}\label{Putin} For every slit $\Gamma$ with $\hcap(\Gamma)=2T$ there exists a uniquely determined, continuous function $U:[0,T]\to\R$ such that the solution to 
\begin{equation}\label{Moskau} \dot{g}_t = \frac{2}{g_t-U(t)}, \qquad g_0 = \id_\Ha,\end{equation}
satisfies $g_{T}=f_{\Gamma}^{-1}.$
\end{satz}

\begin{remark}
Again, the function $U$ is called driving function of $\Gamma$ and the differential equation for $g_t$ is called \emph{chordal} Loewner equation\footnote{The name \emph{chordal} is derived from the picture we get, when $T\to\infty:$ The equation will produce a slit that connects a point on the real axis to $\infty$, which is another boundary point of $\Ha$ on the Riemann sphere. In this case, $\Gamma$ looks like a \emph{chord} within $\Ha$. (Note that $\Ha$ looks like a disc on the Riemann sphere.)}.  Thus, Theorem \ref{Michelle} as well as Theorem \ref{Putin} can both be stated roughly as $$\text{\textit{``For every slit there exists a unique, continuous driving function.''}}$$
Of course, the meaning of ``slit'' and ``driving function'' depends on the domain and the Loewner equation under consideration. 
\end{remark}

Conversely, if one takes a continuous driving function and solves the Loewner equation  (\ref{Moskau}), then the solution is not necessarily describing the growth of a slit. The first example for such a driving function was given by Kufarev (for the radial case) in \cite{MR0023907}.\\

 The most famous example for a driving function $U$ is certainly given by the following example.
\begin{example}\label{lion}
 The case $U(t)=B(\kappa t)$, where $B$ is a standard Brownian motion and $\kappa \geq 0$, is called  \emph{chordal SLE} \footnote{Similarly, \emph{radial SLE} is obtained by putting $\kappa(t)=e^{iB(\kappa t)}$ in (\ref{Obama}).}.
Here, the Loewner equation generates a random curve, which need not be simple. More precisely, there are three different cases dependent on the choice of $\kappa,$ see Figure \ref{sle}: 
\begin{itemize}
 \item $0\leq \kappa\leq 4:$ The random curve is a slit with probability 1.
\item  $4<\kappa<8:$ The random curve hits itself and $\R$ infinitely many often with probability 1. 
\item $8\leq \kappa:$\hspace{0.7cm} The random curve is spacefilling with probability 1.
\end{itemize}\hfill $\bigstar$
\begin{figure}[h]
    \centering
   \includegraphics[width=140mm]{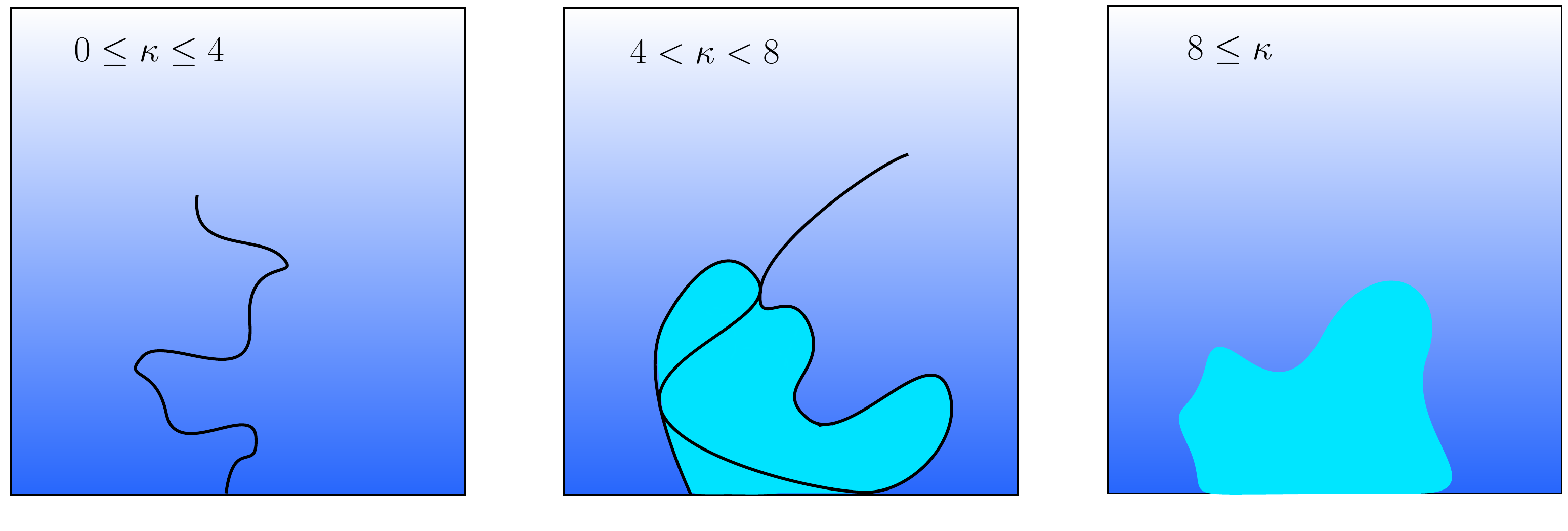}
 \caption{The three cases of chordal SLE.}\label{sle}
 \end{figure}
\end{example}

\section{Generalization of Loewner theory}

The invention of SLE by Schramm led to an explosion in the study of random planar curves as well as of deterministic Loewner theory. In particular, F. Bracci, M. Contreras and S. D\'{\i}az-Madrigal et al. generalized the basic notions of Loewner theory and derived Loewner equations that contain the radial and chordal equations as special cases, see \cite{MR2995431, MR2789373}. They gave more general definitions for Loewner chains and evolution families and showed that those are connected to so called Herglotz vector fields by the Loewner differential equations.

\begin{figure}[h] 
    \centering
   \includegraphics[width=105mm]{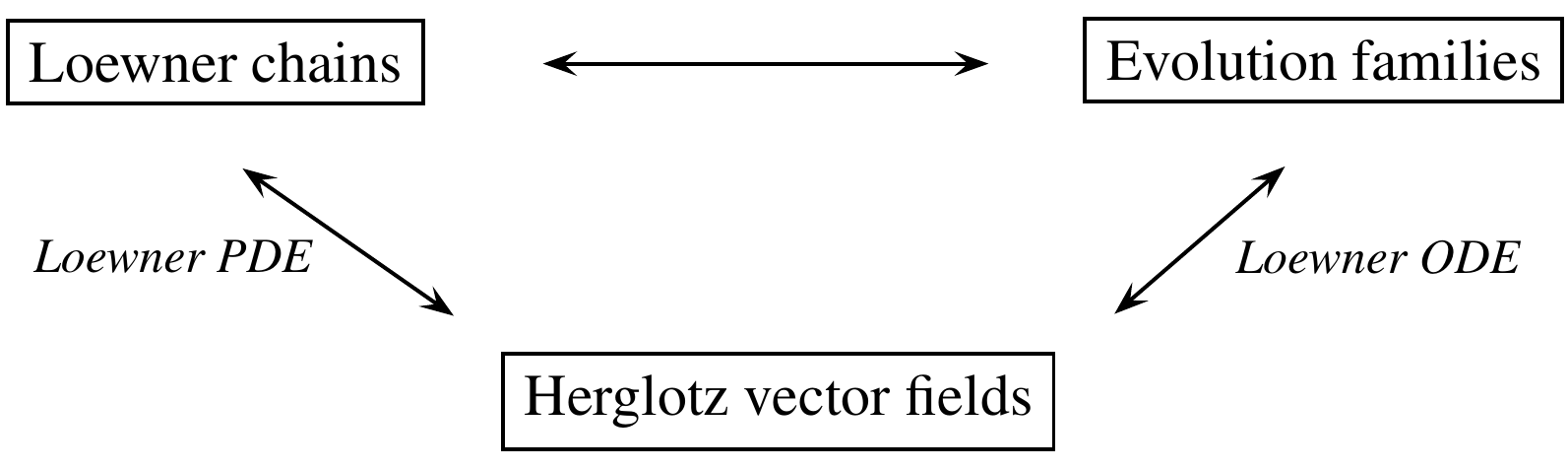}
 \caption{The three basic notions of Loewner theory.}
 \end{figure}

 Moreover, they generalized and transferred these concepts from the case of a simply connected, proper subdomain of $\C$ to (Kobayashi) hyperbolic complex manifolds, see \cite{MR2507634}. As in the case of almost every notion from one dimensional complex analysis that is transferred to the higher-dimensional theory, also Loewner theory for several complex variables bares new  phenomena and problems. Special Loewner equations for complex variables were considered before by J. Pfaltzgraff, T. Poreda, I. Graham, G. Kohr, M. Kohr and others. 
\newpage
\section{Outline of the thesis}
In this thesis, we study several embedding problems for Loewner chains and evolution families, respectively. On the one hand, for a given set $A$ of univalent mappings, embedding problems ask for the characterization of the set of all Loewner chains whose first element is a function from the set $A$. (Or: Describe the set of all evolution families $\{\varphi_{s,t}\}$ such that $\varphi_{0,1}\in A$.)\\
Conversely, for a given a set $L$ of Loewner chains, we might ask for a characterization of the set of all univalent functions that can be the first element $f_0 $ of a Loewner chain $\{f_t\}_{t\geq0}\in L$. (Or: Given a set $E$ of evolution families, how can the set of all elements of the form $\varphi_{0,1}$ be characterized?)\\

The work at hand is divided into two parts. In Chapter 2 we describe the Loewner equations in the modern sense for the two most studied subdomains of $\C^n$: the Euclidean unit ball $\B_n:=\{(z_1,...,z_n)\in \C^n \with |z_1|^2+...+|z_n|^2<1\}$ and the polydisc $\D^n=\{z\in \C^n \with |z_1|,...,|
z_n|<1\}$, and we prove a conjecture made in \cite{MR2943779} about support points of the higher-dimensional analog $S^0(\B_n)$ of the class $S.$ Before we actually prove the statement, we discuss the Runge property of the image domains of all elements of $S^0(\B_n)$ and some important implications. The connections of Loewner chains to this property follows from a theorem by Docquier and Grauert from the 1960's and was established by Arosio, Bracci and Wold in \cite{ABWold}. \\  

Chapter 3 focuses on two problems that come along with Theorem \ref{Putin}. Firstly, we would like to generalize Theorem \ref{Putin} to several slits. Suppose we are given $n$ disjoint slits. We should expect to be able to derive a similar equation for the growth of $n$ slits such that we can assign $n$ unique continuous driving functions to these slits. This will lead us to the problem of finding ``constant coefficients''. Roughly speaking, this means that we have to find a Loewner equation which describes the growth of slits in such a way that each single slit grows with a constant speed.\\
The second problem asks for the converse of Theorem \ref{Putin}. The theorem induces a mapping $\mathcal{DR}$ from the set $$\{\Gamma\with \Gamma \; \text{is a slit with}\,\hcap(\Gamma)=2T\}\quad \text{into the set} \quad \{f:[0,T]\to\R \with f \;\text{is continuous} \}.$$ The mapping $\mathcal{DR}$ is one-to-one but \emph{not surjective}, as some simple examples show. We will have a look at the problem how to characterize the image of $\mathcal{DR}$ and we will also consider this problem for the case of several slits. 

\newpage

\chapter{On Loewner theory in the unit ball and polydisc}

\section{Introduction}

There are various ways to find ``higher dimensional'' analogs of classical complex analysis, and, in particular, of univalent functions: The classical case of a holomorphic (univalent) function $f:D\to \C,$ where $D\subset \C$ is a domain, can be extended, e.g., in the following ways:
\begin{enumerate}
 \item[$\blacktriangleright$] Quaternionic analysis: Replace $\C$ by the skew field $Q$ of all quaternions and consider differentiable (and injective) functions $f:D\to Q,$ $D\subseteq Q.$ 
 \item[$\blacktriangleright$] Quasiregular maps: Replace $\C$ by $\R^n,$ $n\geq 2$, and consider quasiregular functions $f: D\to \R^n,$ $D\subseteq \R^n.$ Injective quasiregular maps are called quasiconformal.
 \item[$\blacktriangleright$] Complex analysis in several variables: Consider holomorphic functions $f:D\to \C^n$, $D\subseteq \C^n,$ i.e. $f$ depends holomorphically on each variable separately. The map $f$ is called univalent if it is holomorphic and injective.
\end{enumerate}
Loewner theory has been generalized successfully to the third (and most studied) case, complex analysis in several variables. In this chapter we will throw a glance at this theory and some problems and differences that appear when $n\geq 2.$\\
First we introduce some notations: For $z,w\in\C^n,$ we denote by $\left<z,w\right>$ the Euclidean inner product: $\left<z,w\right>=\sum_{j=1}^n z_j\overline{w_j}.$ Let $\|z\|=\sqrt{\left<z,z\right>}$ be the Euclidean norm of $z$ and denote by $\|z\|_{\infty}=\max_{j=1,...,n}|z_j|$ its maximum norm. Furthermore, if $D,E\subset \C^n$ are  domains, we denote by $\Ho(D,E)$ the set of all holomorphic functions from $D$ into $E.$\\ Even though many phenomena and results from Loewner theory can be stated for complex manifolds with quite general assumptions, we will treat only two cases:   In the following, we let the standard domain $\Do,$ $n\geq 1,$ be either
\begin{enumerate}
\item[\textbullet] the  \textit{Euclidean unit ball}:\hspace{0.3cm} $\Do= \B_n:=\{z\in \C^n \with \|z\|<1\}$ or
\item[\textbullet] the \textit{polydisc}: \hspace{2.15cm}$\Do= \D^n=\{z\in \C^n \with \|z\|_{\infty}<1\}$.\footnote{More precisely, either $\Do=\B_n$ in the whole Chapter 2 or $\Do=\D^n$ in the whole Chapter 2. For example: the set $\Ho(\Do,\Do)$ refers to $\Ho(\B_n,\B_n)$ or $\Ho(\D^n,\D^n)$.}
\end{enumerate}
Looking for a Riemann Mapping Theorem in higher dimensions means trying to classify the biholomorphism class $$[D]:=\{E \subset \C^n \with \text{There exists a biholomorphic map}\; f:D\to E\},$$ where $D\subseteq\C^n$ is a given domain. This task seems to be much harder compared to the case $n=1$, as a remarkable result by H. Poincar{\'e} shows: he proved that $[\B_n]\not= [\D^n]$ for $n\geq 2,$ see \cite{Poi}. (This result is sometimes referred to as the ``nonexistence of a Riemann mapping theorem'' in higher dimensions, which, however, is only half the story. There are also positive mapping results, see \cite{MR1608852}, p. 30, or \cite{greene2011geometry}.) In fact, many useful or beautiful theorems from classical complex analysis don't generalize to higher dimensions, which led to the introduction of many new notions and tools in order to handle problems that only appear when $n\geq 2,$ like the famous \textit{Levi problem}.\\

However, there are many classical results concerning geometric and analytic properties of univalent functions which have been transferred to the higher dimensional case. Let $$U(\Do):=\{f\in\Ho(\Do,\C^n)\with f \;\, \text{is univalent}\}.$$ 
Furthermore, we define the following natural generalization of the class $S$:
$$ S(\Do):=\{f\in\Ho(\Do,\C^n) \with f\; \text{is univalent and}\; f(0)=0, Df(0)=I_n\}.$$
  
A collection of properties of univalent functions on $\Do$, like conditions for starlikeness, convexity, etc., and many references can be found in  Chapter 6 and 7 of \cite{graham2003geometric}. Compared to the case $n=1,$ the normalization $f(0)=0, Df(0)=I_n,$ is of much less use when $n\geq 2.$ This is due to the fact that the automorphism group 
$$ \aut(\C^n) := \{f\in \Ho(\C^n,\C^n)\with f \;\, \text{is biholomorphic}\}, $$ is much more complicated in this case. So the normalization $f(0)=0, Df(0)=I_n$, ``factors out'' $\aut(\C^n)$ if and only if $n=1,$ see \cite{graham2003geometric}, p. 210. 

\begin{proposition}
 $S(\Do)$ is closed, but not compact for $n\geq 2.$
\end{proposition}
 \begin{proof}
The closedness of $S(\Do)$ follows from Hurwitz's theorem and the fact that the functionals $f\mapsto f(0),$ and $f\mapsto Df(0)$ are continuous.\\
Let $h:\C\to\C$ be an entire function with $h(0)=h'(0)=0.$ Then, the function $(z_1,z_2,...,z_n)\mapsto (z_1,z_2,...,z_n+h(z_1))$ is a normalized univalent function on $\C^n.$ Its restriction to $\Do$ clearly belongs to $S(\Do),$ but the set of all these functions is not a normal family. 
 \end{proof}

When $n\geq 2,$ the class $S(\Do)$ fails to share another important property with the class $S$: \\
An important fact from Loewner theory in one dimension is that every $f\in S$ can be embedded in at least one classical radial Loewner chain. Geometrically, this means that any simply connected, proper subdomain of $\C$ can be spread biholomorphically to $\C.$ In Section \ref{Teerapat}, we will see that this is not longer true for the class $S(\Do)$ in higher dimensions. \\

One can use an analog to the radial Loewner equation to define the subclass $S^0(\Do)$ of all elements of $S(\Do)$ that have \textit{parametric representation}, see Section \ref{Julia}. In particular, every element of $S^0(\Do)$ can be embedded in a radial Loewner chain. This was done by T. Poreda for $\Do=\D^n$ and by G. Kohr for the case $\Do=\B_n.$ This class shares many properties with the class $S$, especially compactness. However, there are still many open questions concerning $S^0(\Do)$ -- questions that arise from drawing comparisons to the one-dimensional case, e.g. estimates on the coefficients of the power series expansion in $0$ -- even though one can apply tools from Loewner theory by the very \textit{definition} of $S^0(\Do)$.\\

In Sections \ref{Oana} and \ref{Ikkei}, we answer two of those questions, namely two conjectures on support points made in \cite{MR2943779}. A crucial ingredient for the solution is the fact that each $f\in S^0(\Do)$ maps $\Do$ onto a Runge domain. In fact, this is true for a much larger class of univalent functions, which, in turn, has quite interesting consequences for the general Loewner theory in higher dimensions, see Section \ref{Teerapat}. In Sections \ref{Herglotz} and \ref{Uta} we recall the basic notions and results concerning modern Loewner theory, as it was developed by F. Bracci, M. Contreras and S. D\'{i}az-Madrigal et al.

\section{Herglotz vector fields, evolution families and Loewner chains}\label{Herglotz}

In this section, we introduce the three basic notions from modern Loewner theory: Herglotz vector fields, evolution families and Loewner chains, and we briefly describe their connections. The  definitions and theorems of this section are taken from \cite{MR2507634} and \cite{ABWold}.\\

The fundamental concept of Loewner theory is the description of infinitesimal changes of biholomorphic mappings $f:\Do \to\C^n$ by \textit{semigroups} of holomorphic self-mappings of $\Do$.

\begin{definition}
 A (\textit{continuous one-real-parameter}) \textit{semigroup} in $\Do$ is a mapping $\Phi_t:[0,\infty)\to \mathcal{H}(\Do,\Do)$ with
 \begin{itemize}
\item[(1)] $\Phi_0=\textbf{id}_{\Do},$ 
\item[(2)] $\Phi_{t+s}=\Phi_t\circ \Phi_s,$
\item[(3)] $\Phi_t$ tends to the identity $\textbf{id}_{\Do}$ on compacta of $\Do$ for $t\to0.$ \end{itemize}
\end{definition}

Given a semigroup $\Phi_t$ and a point $z\in \Do,$ then the limit $$G(z):=\lim_{t\to 0}\frac{\Phi_t(z)-z}{t}$$ exists and the vector field $G:\Do\to\C^n$, called the \textit{infinitesimal generator} of $\Phi_t,$ is a holomorphic function (see, e.g., \cite{MR1174816}). We denote by $\I(\Do)$ the set of all infinitesimal generators of semigroups in $\Do$. For any $z\in \Do,$ $\Phi_t(z)$ is the solution of the initial value problem \begin{equation}\label{cauchy}\dot{w}_t=G(w_t),\quad w_0=z.\end{equation}

Now, the \textit{ordinary Loewner equation} is the following generalized version of (\ref{cauchy}):
\begin{equation}\label{ordLoewner}\dot{w}_t=G(w_t,t),\quad w_0(z)=z\in\Do \quad\text{with} \quad G(\cdot,t)\in \I(\Do) \quad \text{for almost all }\; t\geq 0;\end{equation}
i.e. for almost each $t\geq 0$, the right side of the differential equation is an infinitesimal generator, but it may change with $t$.\\

In modern Loewner theory the vector field $(z,t)\mapsto G(z,t)$ is required to be a so called Herglotz vector field. This notion was  introduced in \cite{MR2507634} for general complete hyperbolic manifolds.

\begin{definition}
 A \textit{Herglotz vector field} of order $d\in[1,+\infty]$ on $\Do$ is a mapping $G:\Do\times [0,\infty) \to \C^n$ with the following properties:
\begin{itemize}
 \item[(1)] The mapping $G(z,\cdot)$ is measurable on $[0,\infty)$ for all $z\in \Do.$
\item[(2)] The mapping $G(\cdot,t)$ is holomorphic for all $t\in[0,\infty).$
\item[(3)] $G(\cdot, t)\in \I(\Do)$ for almost all $t\in[0,\infty).$
\item[(4)] For any compact set $K\subset \Do$ and $T>0,$ there exists a function $C_{T,K}\in L^d([0,T],\R^+_0)$ such that $\|G(z,t)\|\leq C_{T,K}(t)$ for all $z\in K$ and for almost all $t\in[0,T].$ 
\end{itemize}
\end{definition}
\begin{remark}
In \cite{MR2507634} and in some other early works, the definition of a Herglotz vector field involves a condition on the function $G(\cdot,t)$ (for almost all $t$) using the Kobayashi metric of $\Do$ instead of property (3). In \cite{MR2887104} (Theorem 1.1), however, it was shown that this condition is equivalent to $G(\cdot, t)\in \I(\Do)$.
\end{remark}

 One usually introduces a second time parameter $s$ for the solution to the ordinary Loewner equation in order to vary the initial condition, which leads to the definition of evolution families.

\begin{definition} Let $d\in[1,+\infty].$ A family $\{\varphi_{s,t}\}_{0\leq s\leq t<\infty}$ of univalent self-mappings of $\Do$ is called \emph{evolution family of order $d$} if the following three conditions are satisfied: 
\begin{itemize}
\item[(1)] $\varphi_{s,s}=\textbf{id}_{\Do}$ for all $s\geq 0.$
\item[(2)] $\varphi_{s,t}=\varphi_{u,t}\circ \varphi_{s,u}$ for all \;$0\leq s\leq u\leq t< \infty.$
\item[(3)] For any compact subset $K\subset \Do$ and for any $T>0$ there exists a non-negative function $k_{K,T}\in L^d([0,T],\R^+_0)$ such that for all $0\leq s\leq u\leq t\leq T$ and for all $z\in K,$
$$ \|\varphi_{s,u}(z)-\varphi_{s,t}(z)\| \leq \int_u^t k_{K,T}(\tau)\, d\tau. $$ 
\end{itemize}
\end{definition}

For Herglotz vector fields, we always get a unique solution of (\ref{ordLoewner}), which establishes a one-to-one correspondence between Herglotz vector fields and evolution families, see Propositions 3.1, 4.1 and 5.1 in \cite{MR2507634}.

\begin{theorem}\textbf{(Herglotz vector fields and evolution families)}\label{Eva}\\[1mm]
 For any Herglotz vector field $G(z,t)$ of order $d$ on $\Do$  there exists a unique evolution family  $\{\varphi_{s,t}\}_{0\leq s\leq t<\infty}$ of order $d$ such that for all $z\in \Do$ 
\begin{equation}\label{cauchy2}\frac{\partial{\varphi}_{s,t}(z)}{\partial t}=G({\varphi}_{s,t}(z),t) \quad \text{for almost all} \;\; t\in [s,\infty)\quad \text{and}\quad \varphi_{s,s}(z)=z.\end{equation}
Conversely, for any evolution family $\varphi_{s,t}$ of order $d$, there exists a Herglotz vector field $G(z,t)$ of order $d$ such that $\varphi_{s,t}$ satisfies (\ref{cauchy2}).
\end{theorem}
 
\begin{definition}\label{Mats}
Let $d\in[1,+\infty].$ A family $\{f_t:\Do\to\C^n\}_{t\in[0,\infty)}$ of univalent mappings is called \emph{Loewner chain of order $d$} if the following two conditions are satisfied:
\begin{itemize}
 \item[(1)] $\{f_t\}_{t\geq 0}$ has growing images: $f_s(\Do)\subseteq f_t(\Do)$ for all $0\leq s\leq t.$
\item[(2)] For any compact set $K\subset \Do$ and $T>0,$ there exists a function $c_{K,T}\in L^d([0,T],\R^+_0)$ such that $$\|f_s(z) - f_t(z)\| \leq \int_s^t c_{K,T}(\tau)\, d\tau\quad \text{for all} \quad z\in K\;\; \text{and}\;\;  0\leq s\leq  t\leq T.$$
\end{itemize}
A Loewner chain $f_t$ is called normalized if $e^{-t}f_t \in S(\Do)$ for all $t\geq 0.$
\end{definition}

Thus, a Loewner chain spreads the initial domain $f_0(\Do)$ ``biholomorphically'' to the larger domain $$R=\bigcup_{t\geq 0} f_t(\Do).$$
$R$ is called the \emph{Loewner range} of the Loewner chain.\\
Loewner chains are described by the partial Loewner equation, which has the following form:
\begin{equation}\label{PDE2}\dot{f}_t(z)=-Df_t(z)G(z,t) \quad \text{for almost all} \; t\in [0,\infty),\end{equation}
where $G(z,t)$ is a Herglotz vector field. The important difference between the ordinary equation is that we don't know the possible initial values $f_0,$ i.e. we cannot prescribe an arbitrary univalent function $f:\Do\to\C^n.$ However, Herglotz vector fields also guarantee the existence of solutions to the Loewner PDE in $\Do$ (and more generally in complete hyperbolic starlike domains), see \cite{ABWold}:

\begin{theorem}\textbf{(Herglotz vector fields and Loewner chains)} \\[1mm]
 For any  Herglotz vector field $G(z,t)$ of order $d$ on $\Do,$ there exists a Loewner chain $\{f_t:\Do\to\C^n\}_{t\geq 0}$ of order $d$ such that for all $z\in \Do$ equation (\ref{PDE2}) is satisfied.\\
Furthermore, if $R=\cup_{t\geq 0}f_t(\Do)$ and $\{g_t\}_{t\geq0}$ is another solution to (\ref{PDE2}), then it has the form $\{g_t=\Phi\circ f_t\}_{t\geq 0}$ for a holomorphic map $\Phi: R\to \C^n.$ $\{g_t\}_{t\geq 0}$ is a family of univalent mappings if and only if $\Phi$ is univalent.\\
Conversely, if $\{f_t:\Do\to\C^n\}_{t\geq0}$ is a Loewner chain of order $d$, then there exists a Herglotz vector field $G(z,t)$ of order $d$ on $\Do$ such that $\{f_t\}_{t\geq0}$ satisfies (\ref{PDE2}).
\end{theorem}

Finally, the following relation holds between evolution families and Loewner chains, see, e.g., \cite{ABWold}.

\begin{theorem}\textbf{(Evolution families and Loewner chains)}\label{Thomas}\\[1mm]
Let $\{\varphi_{s,t}\}_{0\leq s\leq t<\infty}$ be an evolution family of order $d$ satisfying (\ref{cauchy2}) with $G(z,t).$ Any family $\{f_t:\Do\to\C^n\}_{t\geq0}$ of univalent mappings with
$$ f_s = f_t \circ \varphi_{s,t}, \qquad 0\leq s\leq t, $$ 
is a Loewner chain of order $d$ satisfying (\ref{PDE2}) with $G(z,t).$\\
Conversely, if $\{f_t:\Do\to\C^n\}_{t\geq0}$ is a Loewner chain of order $d$ satisfying (\ref{PDE2}) with $G(z,t),$ then $\varphi_{s,t}:=f_t^{-1}\circ f_s$ is an evolution family of order $d$ satisfying (\ref{cauchy2}) with $G(z,t).$
\end{theorem}

 \begin{remark}\label{Chiffre} Sometimes it is convenient to consider the inverse mappings of solutions to (\ref{cauchy2}) and (\ref{PDE2}), or ``decreasing Loewner chains'', which lead to slightly modified differential equations, again called Loewner equations in the literature. We will encounter such a modified version in Chapter 3. Therefore,
we mention two further Loewner equations at this point. Let $G$ be a Herglotz vector field, then we distinguish between the following four basic types.

{\center
\renewcommand{\arraystretch}{1.5}
\begin{tabular}{|c|c|c|} \hline
Type & Name & Differential Equation \\[1mm] \hline
$(I)$ & ordinary Loewner equation & $\dot{w}_t(z)=G(w_t(z),t)$ \\[2mm] \hline 
$(II)$ & partial  Loewner equation & $\dot{f}_t(z)=-Df_t(z)G(z,t) $  \\[2mm] \hline
$(III)$ & reversed ordinary Loewner equation & $\dot{w}_t(z)=-G(w_t(z),t)$ \\[2mm] \hline
$(IV)$ & reversed partial  Loewner equation & $\dot{f}_t(z)=Df_t(z)G(z,t) $\\\hline
\end{tabular}\\[0.5cm]}

We will not discuss the general reversed equations here. Instead, we refer to \cite{localdu}, where the authors study relations between those equations and ``decreasing Loewner chains'' for the case $n=1.$ \end{remark}

Many problems in Loewner theory ask for connections between properties of the three basic notions we introduced in this section. For example, one may be interested in finding analytic properties of a Herglotz vector field $G(z,t)$ that imply a certain geometric property for the domains $\varphi_{s,t}(\Do),$ such as: $\varphi_{s,t}$ has quasiconformal extension to the boundary $\partial \Do,$ or, when $n=1:$ $\varphi_{s,t}(\D)$ is a slit disc, etc.\\
  In Chapter 3, we will look at some problems of this kind for a particular type of Herglotz vector fields in one dimension. Next we briefly describe the general situation for the case $n=1.$

\section{The one-dimensional case}

Fortunately, the set $\I(\D)$ of infinitesimal generators of the unit disc can be characterized completely by the Berkson-Porta representation formula (see \cite{MR0480965}): $$\I(\D)=\{z\mapsto (\tau-z)(1-\overline{\tau}z)p(z)\with \tau\in\overline{\D}, \; p\in\Ho(\D,\C) \;\text{with}\; \Re p(z)\geq0 \;\text{for all}\; z\in\D\}.$$

\begin{remark}
Let  $p\in\Ho(\D,\C) \;\text{with}\; \Re p(z)\geq0 \;\text{for all}\; z\in\D$. Note that either $p(z)\equiv ib,$ for some  $b\in \R$ or $\Re p(z)>0$ for all $z\in\D$ by the maximum principle. A useful description of the second case is given by the Riesz-Herglotz representation formula for the Carath\'{e}odory class $$\mathcal{P}:=\{p\in\Ho(\D,\C) \with \Re p(z)>0 \;\text{for all}\; z\in\D, \;p(0)=1\}.$$ For any $f\in \Ho(\D,\C)$ with $f/f(0)\in\mathcal{P}$ there exists a probability measure $\mu$ on $\partial \D$ such that 
\begin{equation}\label{Christian}     f(z) = f(0)\cdot \int_{\partial \D} \frac{u + z}{u -z} \, d\mu(\theta)\quad \text{for all}\; z\in\D.  \end{equation}
\end{remark}

\begin{remark}
 If $\Phi_t$ is a semigroup on $\D$ such that $\Phi_t$ is not an elliptic automorphism of $\D$ for all $t> 0,$ then the point $\tau$ of its infinitesimal generator is the Denjoy-Wolff point of the semigroup, i.e. $\lim_{t\to\infty}\Phi_t(z)=\tau$ for all $z\in\D$.
\end{remark}

Consequently, the ordinary Loewner equation for $\D$ reads \begin{equation}\label{loe}\dot{\varphi}_{s,t}=(\tau_t-\varphi_{s,t})(1-\overline{\tau_t}\cdot\varphi_{s,t})\cdot p(\varphi_{s,t},t), \quad \varphi_{s,s}(z)=z\in \D, \; t\geq s,\end{equation}  where the right side of the differential equation is a Herglotz vector field.\\

Assume that an evolution family $\varphi_{s,t}$ does not contain automorphisms of $\D$ whenever $s<t.$ It has been proven that in this case all non-identical elements of $\varphi_{s,t}$ share the same Denjoy-Wolff point $\tau$ if and only if $\tau_t\equiv \tau$ for almost all $t\geq 0,$ see \cite{MR2719792}, p. 4. Therefore, one considers the following two special cases: $$\tau_t \equiv \tau_0\in\D  \quad \text{and} \quad \tau_t \equiv \tau_0\in\partial\D.$$ Of course, we can assume without loss of generality that $\tau_0=0$ in the first and $\tau_0=1$ in the second case. In modern literature, e.g. in \cite{MR2719792}, the first case is called \emph{radial} and the second \emph{chordal}. In the following we will use these terms in a different sense, namely only for special cases where also the function $p$ has further normalizations.
\begin{enumerate}
 \item[\textbullet] A Herglotz vector field that satisfies $\tau_t\equiv 0,$ $\Re p(z,t)>0$ and $p(0,t)=1$ for all $z\in\D$ and almost all $t\geq 0$ is called \emph{radial} (in the classical sense). In other words: $z\mapsto -G(z,t)/z\in \mathcal{P}$ for almost all $t\geq0.$ 
Geometrically, a radial evolution family, i.e. an evolution family $\varphi_{s,t}$ satisfying the radial Loewner ODE, describes the growth of sets from $\partial \D$ to 0. It is normalized by $\varphi_{s,t}(0)=0$ and $\varphi_{s,t}'(0)=e^{t-s}.$ The radial Loewner PDE has the form
\begin{equation}\label{Patricia} \dot{f}_t(z)=zp(z,t)f_t'(z) \quad \text{for all} \; z\in\D \; \text{and for almost all} \; t\geq 0. \end{equation}
Many important, well-known statements concerning the radial Loewner PDE can be found in  \cite{Pom:1975}, Chapter 6. We summarize some of them in the following Theorem.
\begin{theorem}\label{Pommes} Every normalized Loewner chain $\{f_t\}_{t\geq0}$ on $\D$, i.e. $e^{-t}f_t\in S$ for all $t\geq 0,$ satisfies (\ref{Patricia}). Conversely, every univalent solution to (\ref{Patricia}) is a normalized Loewner chain. Furthermore, the following statements hold: \begin{enumerate}
\item If $\{f_t\}_{t\geq0}$ is a normalized Loewner chain, then $f=f_0$ has \emph{parametric representation}, which means that $f$ can be calculated by the associated evolution family $\varphi_{s,t}$ via $$ f = \lim_{t\to\infty}e^t\varphi_{0,t}, $$
see \cite{Pom:1975}, Theorem 6.3. Furthermore, the Loewner range of the chain is $\C.$
\item For every $f\in S,$ there exists a normalized Loewner chain $\{f_t\}_{t\geq0}$ such that $f=f_0,$
see \cite{Pom:1975}, Theorem 6.1.
\item For every $f\in S$ that maps $\D$ onto $\C$ minus a slit, there exists exactly one normalized Loewner chain $\{f_t\}_{t\geq0}$ with $f=f_0$ and $p(z,t)$ has the form
\begin{equation}\label{Pavel} p(z,t)= \frac{\kappa(t)+z}{\kappa(t)-z}, \end{equation}
where $\kappa:[0,\infty)\to\partial\D$ is a continuous function (the so-called driving function), see \cite{Loewner:1923}. 
\end{enumerate}

\end{theorem}
\item[\textbullet] In the (classical) \emph{chordal case}, we have $\tau_t\equiv 1$ and the function $p(z,t)$ satisfies further technical assumptions. This case is treated in more detail in Chapter 3. Here we only look at a special case that explains where the name ``chordal'' comes from:
Let the Herglotz vector field $G(z,t)$ have the form $$G(z,t)=(z-1)^2\cdot\frac1{\frac{1+z}{1-z}-iu(t)},$$
where $u(t)$ is a continuous, real-valued function. Every simple curve that grows from a point $p\in\partial \D\setminus{1}$ to $1$ within $\D$ -- a ``chord'' joining $p$ and 1 -- can be described by an evolution family that corresponds to a Herglotz vector field of this form, see Section \ref{Pascha}. This is comparable to the radial case of the form (\ref{Pavel}).
\end{enumerate}

We have seen that the Loewner range of a normalized Loewner chain is always the whole complex plane $\C.$ The following result handles the general case.

\begin{theorem}[Theorem 1.6 and 1.7 in \cite{MR2789373}]\label{Fipi}
 Let $G(z,t)$ be a Herglotz vector field on $\D$ with evolution family $\varphi_{s,t}$. Then there exists a unique Loewner chain $\{f_t\}_{t\geq0}$ with $f_0\in S$ such that its range $R$ is either an Euclidean disc or the complex plane. Furthermore,
\begin{enumerate}[(a)]
 \item $R=\C$ if and only if  for one $z\in \D$ (and thus for all $z\in \D$), 
$$ \beta(z):=\lim_{t\to\infty}\frac{|\varphi_{0,t}'(z)|}{1-|\varphi_{0,t}(z)|^2} =0.$$
\item If $R\not=\C,$ then the Euclidean disc $R$ has the form $R=\{z\in\C\with |z|<1/\beta(0)\}.$
\end{enumerate}
\end{theorem}

\section{The Loewner PDE and the range problem}\label{Uta}

Let $G(z,t)$ be a Herglotz vector field on $\Do.$ Recall that there exists a Loewner chain $\{f_t: \Do\to \C^n\}_{t\geq0}$ that satisfies the corresponding Loewner PDE (\ref{PDE2}) and that $R=\bigcup_{t\geq0} f_t(\Do)$ is the Loewner range of this chain.\\
If $\{g_t: \Do\to \C^n\}_{t\geq0}$ is another Loewner chain on $\Do$ satisfying the same PDE, then $g_t=\Phi \circ f_t$ with a univalent map $\Phi: R \to \C^n.$ Consequently, the biholomorphism class $[R]$ of the Loewner range is uniquely determined by the Herglotz vector field $G(z,t)$ and we call $[R]$ the \textit{Loewner range of} $G(z,t)$.
\begin{example}
 For $n=1,$ any Loewner range $R$ of a Loewner chain is a union of increasing, simply connected domains, and thus, $R$ is simply connected too. Hence $[R]=\C$ or $[R]=\D,$ where $[R]=\C$ if and only if $R=\C.$ \hfill $\bigstar$
\end{example}

This bivalence of the biholomorphism class $[R]$ doesn't hold in higher dimensions any longer: the constant Loewner chain $\{\textbf{id}_{\Do}\}_{t\geq 0}$ is an example for the case $[R]=[\Do]$ and the case $[R]=[\C^n]$ is given, e.g., by the Loewner chain $\{e^t\textbf{id}_{\Do}\}_{t\geq 0}.$  By combining these two cases, one easily obtains examples for the ranges $[R]=[\Do^{n-k}\times \C^k]$, $k=0,...,n\geq2.$ Thus the number of possible values for $[R]$ is greater than two for $n\geq 2.$ However, the following statement can be regarded as a replacement for this two--valuedness.
\begin{proposition} If the range $R$ of a Loewner chain is a (Kobayashi) hyperbolic domain, then it is biholomorphically equivalent to $\Do.$ \\
In particular, if $\Do=\B_n,$ then $[R]=[\D^n]$ is impossible for $n\geq 2$ and vice versa. 
\end{proposition}
\begin{proof} See Theorem 3 in \cite{MR0435441}.
\end{proof}
Consequently, either $[R]=[\Do]$ or $R$ is not a hyperbolic domain. For more information about  Loewner ranges, we refer to \cite{MR3043148}.\\

 The general problem of determining the Loewner range of a given Herglotz vector field  seems to be nontrivial.  An example for sufficient conditions on the Herglotz vector  field such that $[R]=[\C^n]$ can be found in \cite{Arosio:2012}, Theorem 1.4.\\
From Theorem \ref{Eva} it follows that the biholomorphism class $[R]$ of a Loewner range is also uniquely determined by the corresponding evolution family $\varphi_{s,t}$. In Section 4 of \cite{MR3043148} one can find certain conditions on $\varphi_{s,t}$ that determine $[R],$ see also Theorem \ref{Fipi} for the case $n=1$. However, also from this point of view, it seems to be difficult to find out whether $[R]=[\C^n]$ or not, as this problem is related to the so-called Bedford conjecture.

\begin{example}
 Let $\{\phi_t: \B_n \to \C^n\}_{t\geq0}$ be a Loewner chain in the unit ball such that $\phi_t\in \aut(\C^n)$ for all $t\geq 0$ and $\phi_0(z)=z.$  Thus, this Loewner chain extends the unit ball to a domain $\Omega=\bigcup_{t\geq 0}\phi_t(\B_n).$ \\
We assume the simplest case where  $t\mapsto \phi_t^{-1}$ is a semigroup of automorphisms. Furthermore, we assume that $\phi_t^{-1}$ converges in $\overline{\B_n}$ locally uniformly to one point $p\in \overline{\B_n}$ for $t\to\infty.$ As $p\in \Omega,$ we conclude that $\Omega$ is the basin of attraction of $\phi_1^{-1}$ with respect to $p,$ i.e. $$ \Omega = \{z\in\C^n\with \phi_1^{-k}(z)\to p \; \text{for} \; k\to\infty\}\supset \overline{\B_n}.$$
Under certain conditions, e.g. the eigenvalues of $(\phi_1^{-1})'(p)$ have modulus $< 1,$  it is known that $[\Omega] = [\C^n]$, see p. 84 in \cite{MR929658}. This should still hold when $\phi_t^{-1}$ is not necessarily a semigroup but the domains $\phi_t^{-1}(\B_n)$ are shrinking ''uniformly``.
\hfill $\bigstar$\end{example}

More precisely, the Bedford conjecture is the following statement about a more general basin of attraction, see \cite{MR2262775} and \cite{Arosio:2012}.
\begin{conj}\textbf{(Bedford's conjecture)}\\
Let $f_1, f_2,...$ be a sequence of automorphisms of $\C^n$ that fix the origin. Suppose that there exist $a,b\in\R$ with $0<a<b<1,$ such that for every $z\in\B_n$ and $k\in \N$ the following holds:
$$ a\|z\| \leq \|f_k(z)\| \leq b\|z\|. $$
Then the basin of attraction $\{z\in\C^n \with F_m(z):=(f_m\circ ... \circ f_1)(z)\to 0 \quad \text{for}\quad m\to\infty\}$ is biholomorphic to $\C^n.$
\end{conj}

An important question is, which elements of $U(\Do)$ can be embedded in (i.e. can be the first element of) a Loewner chain. However, if we have no further restrictions, this question is trivial, as the constant Loewner chain $f(\cdot,t)=c$ satisfies the Loewner PDE with $G(z,t)\equiv0$ for every $c\in U(\Do).$ Thinking of the classical case, where a simply connected domain is extended to the whole complex plane, we might ask which univalent functions can be embedded in a Loewner chain with Loewner range $R$ satisfying $R=\C^n$ or more generally $[R]=[\C^n].$ \\
So let us define the following two subsets of $U(\Do)$: $$E(\Do)=\{f\in U(\Do)\with \text{There is a Loewner chain} \; \{f_t\}_{t\geq0} \;\text{with}\; f_0=f \;\text{and} \;R=\C^n\},$$
$$\widetilde{E}(\Do)=\{f\in U(\Do)\with \text{There is a Loewner chain} \; \{f_t\}_{t\geq0} \;\text{with}\; f_0=f \;\text{and} \;[R]=[\C^n]\}.$$
When $n=1,$ these two sets are identical and by Theorem \ref{Pommes} (b) we have $$E(\D)=\widetilde{E}(\D)=U(\D).$$ In higher dimensions, the elements of $E(\Do)$ are connected to those $f\in U(\Do)$ which map $\Do$ onto Runge domains. We will discuss this connection in the next section.

\section{The Runge property}\label{Teerapat}

Let $D\subseteq \C^n$ be a domain. We let $\Po(D,\C^n)\subset\Ho(D,\C^n)$ be the set of all polynomials with the topology induced by locally uniform convergence in $D$.\\ For $n=1,$ a version of the polynomial Runge theorem says that $\Po(D,\C)$ is dense in $\Ho(D,\C)$ whenever $D$ is simply connected. In higher dimensions, this is no longer true and one calls $D$ a \textit{Runge domain} if $\Po(D,\C^n)$ is dense in $\Ho(D,\C^n)$. More generally, if $A, B\subset \C^n$ are domains with $A\subset B$, then $(A,B)$ is called a \textit{Runge pair} if $\Ho(B,\C^n)$ is dense in $\Ho(A,\C^n)$. Thus, $D$ is a Runge domain if and only if $(D,\C^n)$ is a Runge pair.\\
Both the unit ball $\B_n$ and the polydisc $\D^n$ are simple examples of Runge domains, because they are even domains of convergence (of a power series). However, the Runge property is not invariant with respect to biholomorphic mappings. The first example was found by J. Wermer, see \cite{MR0121500}. He constructed a non-Runge domain which is biholomorphically equivalent to the bidisc $\D^2.$ The fact that also $\C^n$ can be mapped biholomorphically onto a non-Runge domain for $n\geq 2$ readily implies that the same is true for $\Do$:

\begin{proposition}\label{nonrunge}
 For every $n\geq 2,$ there exists a non-Runge domain $D$ which is biholomorphic equivalent to $\Do.$ 
\end{proposition}
\begin{proof}
In \cite{MR2372737}, E. Wold constructs a Fatou-Bieberbach domain $E\subset \C^2$ which is not Runge. Let $F: \C^2 \to E $ be biholomorphic. For $n\geq 2$ we define $G:\C^n\to\C^n,$ $G(z_1,z_2,z_3,...,z_n):=(F(z_1,z_2),z_3,...,z_n).$ It is easy to see that also $G(\C^n)$ is not Runge. Now, write $$G(\C^n)=\bigcup_{k=1}^\infty G(k\cdot \Do).$$ If $G(k\cdot \Do)$ was Runge for every $k\geq 1,$ then $G(\C^n)$ would be Runge as a union of increasing Runge domains. Consequently, there exists $k_0\in \N$ such that $D:=G(k_0\cdot \Do)$ is not Runge. 
\end{proof}

Now we rephrase the property that $\Do$ is mapped biholomorphically by $f$ onto a Runge domain in terms of $f$.

 \begin{lemma}\label{ru}
  Let $f\in U(\Do).$ The following statements are equivalent:
\begin{itemize}
 \item[a)] $f(\Do)$ is a Runge domain.
\item[b)]  There exists a sequence $p_k$ of polynomials with $\displaystyle f^{-1}=\lim_{k\to\infty}p_k$ locally uniformly on $f(\Do).$
\item[c)] For every $g\in \Ho(\Do,\C^n)$ there exists a sequence $p_k$ of polynomials with $\displaystyle g=\lim_{k\to\infty}{(p_k\circ f)}$ locally uniformly on $\Do.$
\end{itemize}
When $n\geq 2,$ each of these statements is equivalent to:
\begin{itemize}
\item[d)] There exists a sequence $\varphi_k$ of automorphisms of $\C^n$ with $\displaystyle f=\lim_{k\to\infty}\varphi_k$ locally uniformly on $\Do.$
\end{itemize}  
 \end{lemma}
 \begin{proof} 
 ``$a)\Rightarrow b):$'' This follows from the definition of the Runge property.\\[3mm] ``$b)\Rightarrow c):$'' Let $h_k$ be a sequence of polynomials which approximate $ f^{-1}$ locally uniformly on $f(\Do)$ and let  $g\in \Ho(\Do,\C^n).$ As $\Do$ is Runge, there exists a sequence $q_k$ of polynomials which approximate $g$ locally uniformly on $\Do.$ Then $q_k\circ h_k\circ f$ approximates locally uniformly the map $g\circ f^{-1}\circ f=g.$ \\[3mm]
``$c)\Rightarrow a):$'' First, note that locally uniform convergence of a sequence $g_k(w)$ in $f(\Do)$ is equivalent to locally uniform convergence of the sequence $g_k(f(z))$ in $\Do$ as $f$ is biholomorphic. Now let $h\in \Ho(f(\Do),\C^n)$ and write $h=(h\circ f)\circ f^{-1}.$ Then $h\circ f\in \Ho(\Do,\C^n)$ and there exists a sequence $p_k$ of polynomials with $\displaystyle h\circ f=\lim_{k\to\infty}p_k(f)$ and thus $\displaystyle h=\lim_{k\to\infty}p_k$ locally uniformly in $f(\Do).$  \\[3mm]
Now let $n\geq2:$\\[3mm]
``$a)\Rightarrow d):$'' This follows from Theorem 2.1 in \cite{MR1185588}, which states that every biholomorphic map from a starlike domain onto a Runge domain can be approximated by automorphisms of $\C^n$, when $n\geq2$.\\[3mm]
``$d)\Rightarrow a):$'' This direction is already mentioned as a remark at the end of page 372 in \cite{MR1185588}, see also Proposition 1.2 (c) in \cite{MR1213106} for a complete proof.
 \end{proof}

 In \cite{ABWold}, a remarkable fact concerning solutions to the ordinary Loewner equation is shown.
\begin{theorem}[Proposition 5.1 in \cite{ABWold}]\label{runge}
\textcolor{white}{0}\\If $\varphi_{s,t}$ is the solution to the ordinary Loewner equation (\ref{cauchy2}), then $\varphi_{s,t}(\Do)$ is a Runge domain for all $0\leq s\leq t$.
\end{theorem}
This result follows directly from a much more general result of Docquier and Grauert in \cite{MR0148939}. We will apply this result to show that every $f\in E(\Do)$ maps onto a Runge domain. To show this, we have to recall the definition of ``semicontinuous holomorphic extendability'' which was introduced by Docquier and Grauert. We cite the English translation from \cite{ABWold}.
\begin{definition}\label{Gio}
 Let $N\subset \C^n$ be a domain and $M\subset N$ be open and nonempty. Then $M$ is \textit{semicontinuously holomorphically extendable to $N$} if there exists a family $\{M_t\}_{0\leq t \leq 1}$ of nonempty open subsets of $N$ with 
\begin{enumerate}[(1)]
 \item $M_t$ is a domain of holomorphy for all $t$ in a dense subset of $[0,1],$
\item $M_0=M$ and $\cup_{0\leq t\leq 1}M_t=N,$
\item $M_s\subseteq M_t$ for all $0 \leq s\leq t \leq 1,$
\item $\bigcup_{0\leq t< t_0}M_t$ is a union of connected components of $M_{t_0}$ for every $0< t_0\leq 1,$
\item $M_{t_0}$ is a union of connected components of the interior part of $\bigcap_{t_0< t\leq 1}M_t$ for every \\$0 \leq t_0 < 1.$
\end{enumerate}
 
\end{definition}
 Docquier and Grauert showed:

\begin{theorem}[Satz 19 in \cite{MR0148939}]\label{doc}
 Let $M,N\subset \C^n$ be domains. If $M$ is semicontinuously holomorphically extendable to $N,$ then $(M,N)$ is a Runge pair.
\end{theorem}

Now we can prove the following denseness result, which implies a striking difference between the cases $n=1$ and $n\geq 2$ for the class $E(\Do).$ It follows more or less directly from Theorem \ref{doc} and for sure, the authors of \cite{ABWold} were aware of this statement. As it is not mentioned directly, we also include a proof here.
\begin{theorem}\label{denseness}
 $E(\Do)$ is a dense subset of $F(\Do):=\{f\in U(\Do)\with f(\Do) \;\text{is a Runge domain}\}.$ Furthermore, $F(\Do)$ is a (relatively) closed subset of $U(\Do).$\\
In particular, for $n\geq 2$, this implies $$\overline{E(\Do)}\subsetneq U(\Do).$$
\end{theorem}

\begin{proof} For $n=1$ we have $F(\D)=U(\D).$ Let $f\in U(\D),$ then there exists a radial Loewner chain $\{f_t\}_{t\geq0}$ with $f_0=(f-f(0))/f'(0)\in S$ according to Theorem \ref{Pommes} (b).  Hence $\{g_t\}_{t\geq0}:=\{f'(0)f_t+f(0)\}_{t\geq0}$ is a Loewner chain with $g_0=f.$ We conclude $f\in E(\D)$ and  $$E(\D)=F(\D).$$
Now let $n\geq 2,$ $f\in E(\Do)$ and $\{f_t\}_{t\geq0}$ be a  Loewner chain with $f=f_0$ such that \begin{equation}\label{Smaug}\bigcup_{t\geq 0}f_t(\Do)=\C^n.\end{equation}
 First we show that $f(\Do)$ is a Runge domain:\\
For every $s,t$ with $0\leq s\leq t$ we have $f_s(\Do)\subset f_t(\Do)$ and Theorem \ref{doc} implies that
 $(f_s(\Do),f_t(\Do))$ is a Runge pair: The family $\{M_q\}_{0\leq q\leq 1}$ can be chosen as $M_q= f_{s+q\cdot(t-s)}(\Do).$ Clearly, $\{M_q\}$ satisfies the conditions (1), (2), (3) from Definition \ref{Gio}. Moreover, the continuity of Loewner chains, i.e. property (2) in Definition \ref{Mats}, implies that
$$  \bigcup_{0\leq t< t_0}M_t = M_{t_0}\quad \text{for every}\quad 0< t_0\leq 1 \quad \text{and} \quad \bigcap_{t_0< t\leq 1}M_t = M_{t_0} \quad \text{for every}\quad 0\leq t_0< 1. $$
Thus, also conditions (4) and (5) in Definition \ref{Gio} are satisfied.\\
Now we will show that $(f(\Do),\C^n)$ is also a Runge pair. This can be seen as follows:\\[1mm]
Let $g\in \Ho(f(\Do),\C^n)$ be arbitrary and let $\eps>0$ and $K\subset f(\Do)$ be compact. We have to find a function $h\in\Ho(\C^n,\C^n)$ with $\|g(z)-h(z)\|\leq \eps$ for all $z\in K.$ \\
First, as $(f_0(\Do),f_1(\Do))$ is a Runge pair, we can choose $$g_1\in \Ho(f_1(\Do),\C^n) \quad \text{with} \quad \|g(z)-g_1(z)\|<\frac{6\eps}{\pi^2} \quad \text{for all} \; z\in K.$$ 
Because of (\ref{Smaug}), we can find a strictly increasing sequence $(k_m)_{m\in\N_0}$ of non-negative integers with $k_0:=0$, such that  $$\overline{f_{k_{m-1}}(\Do)}\subset f_{k_m}(\Do)\quad \text{for every}\quad m\in\N.$$
Then we can choose $g_m$ for $m\in \N\setminus\{1\}$ inductively such that $$g_m\in\Ho(f_{k_m}(\Do),\C^n) \quad\text{and}\quad \|g_m(z)-g_{m-1}(z)\|<\frac{6\eps}{\pi^2 m^2}\quad \text{for all}\; z\in \overline{f_{k_{m-2}}(\Do)}.$$ Thus, we have for all $m\in\N$  and $z\in K:$ \begin{equation}\label{smeagol}\|g(z)-g_{m}(z)\|\leq \|g(z)-g_{1}(z)\|+\sum_{l=2}^\infty \|g_{l}(z)-g_{l-1}(z)\|< \frac{6\eps}{\pi^2}\sum_{l=1}^\infty \frac1{l^2}=\eps.\end{equation}
For a fixed $m_0\in\N,$ the sequence $\{g_m\}_{m\geq m_0}$ is defined in $f_{k_{m_0}}(\Do)$ and it is locally bounded. Montel's theorem gives us the existence of a subsequence that converges  uniformly on compacta to a function $h\in \Ho(\C^n,\C^n).$ (\ref{smeagol}) implies $$\|g(z)-h(z)\|\leq \eps \quad \text{for all} \; z\in K.$$
So $(f(\Do),\C^n)$ is a Runge pair, i.e. $f(\Do)=f_0(\Do)$ is a Runge domain.\\[2mm]
Next, let $\Phi\in \aut(\C^n)$ and denote by $\phi$ the restriction of $\Phi$ to $\Do.$ For $t\geq0,$ define the function $f_t: \Do \to \C^n$ by $f_t(z)=\Phi(e^tz).$ Then $\{f_t\}_{t\geq0}$ is a Loewner chain with $f_0=\phi$ and $\cup_{t\geq0}f_t(\Do)=\cup_{t\geq0}\Phi(e^t\cdot \Do)=\Phi(\C^n)=\C^n$ and consequently $\phi \in E(\Do).$ \\
Lemma \ref{ru} implies that any $f\in F(\Do)$ can be approximated by automorphisms of $\C^n$ when $n\geq 2.$ So $E(\Do)$ is dense in $F(\Do).$\\[2mm]
Finally, we have to show that $F(\Do)$ is (relatively) closed in $U(\Do).$ Suppose $f_k$ is a sequence of elements of $F(\Do)$ with limit $f\in U(\Do).$ Every $f_k$ can be approximated by automorphisms of $\C^n$. Hence, also $f$ can be approximated by automorphisms of $\C^n$ and consequently $f(\Do)$ is a Runge domain according to Lemma \ref{ru}.
\end{proof}

\begin{remark}
Note that we use only one simple Herglotz vector field in the proof of denseness of $E(\Do)$ with respect to $F(\Do)$ for $n\geq 2$, namely $G(z,t)=-z.$ Hence, the key argument is Lemma \ref{ru}  d), which has been one of the first results of what is now called Anders\'{e}n-Lempert theory. This theory treats general complex manifolds that have the so called \emph{density property} (and \emph{volume density property} respectively), see the  survey article \cite{MR2768636} for further reading.
\end{remark}

\begin{corollary}
 $\widetilde{E}(\Do)$ is a dense subset of $\widetilde{F}(\Do):=\{f\in U(\Do)\with (\Psi^{-1}\circ f)(\Do) \;\text{is a Runge}$ \\ $\text{domain for some univalent} \; \Psi:\C^n\to\C^n, \Psi(\C^n)\supseteq f(\Do)\}.$
\end{corollary}
\begin{proof}
 Let $\widetilde{f} \in \widetilde{E}(\Do)$ with a Loewner chain $\{f_t\}_{t\geq0},$ $f_0=f,$  having a range $R$ that is biholomorphically equivalent to $\C^n.$ There exists a biholomorphic map $\Phi:R\to \C^n$ and then $f:=\Phi\circ \widetilde{f} \in E(\Do).$ The previous Theorem implies that $((\Phi^{-1})^{-1}\circ \widetilde{f})(\Do)$ is a Runge domain. Furthermore, let $g\in \widetilde{F}(\Do)$ with a univalent $\Psi:\C^n\to\C^n$ such that $(\Psi^{-1}\circ g)$ maps $\Do$ onto a Runge domain. Then we can write $\Psi^{-1} \circ g=\lim_{k\to\infty}p_k$ for a sequence $p_k$ of elements of $E(\Do).$ Consequently $g=\lim_{k\to\infty}\Psi\circ p_k$ and $\Psi \circ p_k \in \widetilde{E}(\Do)$ for every $k\in \N$. 
\end{proof}

\begin{conj}\label{Isrunge}
 $E(\Do)=F(\Do)$ and $\widetilde{E}(\Do)=\widetilde{F}(\Do)$ for all $n\in \N.$
\end{conj}
\begin{remark}
 This first equality would imply in particular that every Runge domain biholomorphically equivalent to $\Do$ would be semicontinuously holomorphically extendable to $\C^n.$ A slightly weaker version of this statement is true indeed, see Satz 20 in \cite{MR0148939}.
\end{remark}

\begin{remark}\label{Fil}
 L. Arosio, F. Bracci and E. F. Wold could show that any $f\in F(\B_n),$ such that $f(\B_n)$ is a bounded strongly pseudoconvex $C^\infty$-smooth domain and $\overline{f(\B_n)}$ is polynomially convex, belongs to $E(\B_n),$ see Theorem 1.2 in \cite{filtering}.
\end{remark}

\section{The class \texorpdfstring{$S^0(\Do)$}{}}\label{Julia}

Recall that the class $S(\Do)$ is the set of all normalized univalent functions on $\Do$ and that it is not compact for all $n\geq 2.$\\
In \cite{MR1049182}, Poreda introduced the class $S^0(\D^n)$ as the set of all $f\in S(\D^n)$ that have ``parametric representation''. Later,  G. Kohr defined the corresponding subclass $S^0(\mathcal{\B}^n)$ for the unit ball, see \cite{MR1929522}, which has been extensively studied since its introduction.\\
 First, one considers a special subset $\M(\Do)\subset I(\Do)$ of normalized infinitesimal generators on $\Do.$ We define $$\mathcal{M}(\B_n):=\{h\in \Ho(\B_n,\C^n)\with h(0)=0, Dh(0)=-I_n, \Re\left<h(z),z \right><0 \; \text{for all}\; z\in \B_n\setminus\{0\}\}$$ 
and, in case of the polydisc,
$$ \mathcal{M}(\D^n):=\{h\in\Ho(\D^n,\C^n)\;|\; h(0)=0, Dh(0)=-I_n, \Re\left(\frac{h_j(z)}{z_j}\right)<0 \;\text{when}\;  \|z\|_\infty=|z_j|>0 \}. $$

For $n=1,$ we have $$\M(\D)=\{z\mapsto -zp(z)\with \Re(p(z))>0\; \text{for all}\; z\in\D \; \text{and}\; p(0)=1\},$$ which is the set of all infinitesimal generators that correspond to the classical radial case.\\

We will see soon that there is a one-to-one correspondence between $\M(\Do)$ and the set $S^*(\Do)\subset S(\Do)$ of all normalized \textit{starlike mappings} in $S(\Do),$ i.e. those mappings whose image domain is starlike with respect to the origin. The following  characterization of starlike mappings was proven by Matsuno in 1955 for $\Do=\B_n,$ see \cite{MR0080931}, and for $\Do=\D^n$ in 1970 by Suffridge, see \cite{MR0261040}.

\begin{theorem}\label{Juergen}
Let $f:\Do \to \C^n$ be locally biholomorphic, i.e. $Df(z)$ is invertible for every $z\in\Do,$ with $f(0)=0.$ Then the domain $f(\Do)$ is starlike with respect to $0$ if and only if the function $z\mapsto -(Df(z))^{-1}\cdot f(z)$ belongs to $\M(\Do).$
\end{theorem}

The following property shows a crucial difference between the class $\I(\Do)$ and its subset $\M(\Do)$.

\begin{theorem}[see Theorem 6.1.39 in \cite{graham2003geometric}]
$\M(\Do)$ is compact. 
\end{theorem}

Thus, by Montel's theorem, for every $r\in(0,1)$ there exists $C(r)\geq0$ such that $\|h(z)\|\leq C(r)$ for all $z\in r\cdot \overline{\Do}$ and $h\in \M(\Do).$\\
 Consequently, a mapping $G:\Do\times \R^+ \to \C^n$ with $G(\cdot,t)\in \mathcal{M}_n$ for  all $t\geq 0$ such that $G(z,\cdot)$ is measurable on $\R^+$ for all $z\in\Do$ is automatically a Herglotz vector field of order $\infty$. The corresponding evolution family $\varphi_{s,t}$ is normalized by $\varphi_{s,t}(0)=0,$ $D\varphi_{s,t}(0)=e^{s-t}I_n$. \\

For such Herglotz vector fields, one can compute a special solution to the Loewner PDE from the evolution family $\varphi_{s,t},$ which directly leads to the definition of $S^0(\Do),$ see \cite{graham2003geometric}, Chapter 8 for the case $\Do=\B_n$ and \cite{MR1049182} for $\Do=\D^n$. 

\begin{theorem}\label{Gabriela}
 Let $G(z,t)$ be a Herglotz vector field with $G(\cdot,t)\in \mathcal{M}_n$ for all $t\geq0.$ For every $s\geq 0$ and $z\in \Do,$ let $\varphi_{s,t}(z)$ be the solution of the initial value problem $$\frac{\partial{\varphi_{s,t}(z)}}{\partial t}=G(\varphi_{s,t}(z),t) \quad \text{for almost all}\quad t\geq s, \quad \varphi_{s,s}(z)=z.$$
Then the limit \begin{equation}\label{represent}\lim_{t\to\infty}e^t\varphi_{s,t}(z)=:f_s(z)\end{equation}
exists for all $s\geq 0$ locally uniformly on $\Do$ and defines a univalent function there. Moreover, $f_s(z)=f_t(\varphi_{s,t}(z))$ for all $z\in \Do$ and $0\leq s\leq t<\infty.$  $\{f_s\}_{s\geq 0}$ is a normalized Loewner chain with the property that $\{e^{-s}f_s\}_{s\geq0}$ is a normal family on $\Do.$ Finally,
$$\dot{f}_t(z)=-Df_t(z)G(z,t) \quad \text{for all}\quad z\in \Do\quad \text{and for almost all }\quad t\geq 0.$$
\end{theorem}
 
The first element $f_0\in S(\Do)$ of the Loewner chain in the theorem above is said to have \textit{parametric representation.} It can be easily verified that $e^{-s}f_s$ has parametric representation too for every $s\geq 0$. 
\begin{definition}
 $S^0(\Do):=\{f\in S(\Do)\with \, f \; \text{has parametric representation}\}.$
\end{definition}

The following characterization of the class $S^0(\Do)$ can be found in \cite{graham2003geometric}, Remark 8.1.7, in the case of the unit ball. It is easy to see that it holds true for the polydisc too.
\begin{proposition}$f\in S^0(\Do)$ if and only if there exists a normalized Loewner chain $\{f_t\}_{t\geq0}$ with $f=f_0$ such that $\{e^{-t}f_t\}_{t\geq 0}$ is a normal family on $\Do.$
\end{proposition}

\begin{remark}
 There is also a more geometric characterization of the domains $f(\Do),$ $f\in S^0(\Do),$ called \emph{asymptotic starlikeness}. This notion was introduced by Poreda in \cite{MR1049183}. He showed that this property is a necessary condition for a domain to be the image of a function $f\in S^0(\D^n).$ Under some further assumptions this condition is also sufficient. In \cite{MR2425737}, Theorem 3.1, Graham, Hamada, G. Kohr and M. Kohr proved that $f\in S^0(\B_n)$ if and only if $f\in S(\B_n)$ and $f(\B_n)$ is an asymptotically starlike domain. 
\end{remark}

If $f\in S(\Do),$ then $f\in S^*(\Do)$ if and only if $\{e^t f\}_{t\geq0}$ is a normalized Loewner chain. In particular \begin{equation*}S^*(\Do)\subset S(\Do).\end{equation*}
The following example demonstrates that there is a one-to-one correspondence between $\M(\Do)$ and $S^*(\Do).$ 

\begin{example} 
If $f\in S^*(\Do),$ then $-(Df)^{-1}\cdot f \in \M(\Do)$ by Theorem \ref{Juergen}. The converse relation is given by Theorem \ref{Gabriela}:\\
Let $h\in \M(\Do)$ and denote by $\varphi_{s,t}$ the evolution family that corresponds to the time-independent  Herglotz vector field $G(z,t):=h(z).$ It has the property $\varphi_{s,t}(z)=\varphi_{0,t-s}(z)$ for all $t\geq s$ and $z\in\Do.$ Thus, if $\{f_s\}_{s\geq0}$ is the Loewner chain from Theorem \ref{Gabriela}, then 
$$f_s(z)=\lim_{t\to\infty}e^t \varphi_{s,t}(z)=e^s\lim_{t\to\infty}e^{t-s}\varphi_{0,t-s}(z)=e^s\cdot f_0(z).$$
Hence, $f:=f_0\in S^*(\Do)$ and $$-(Df(z))^{-1}f(z)=-(e^tDf(z))^{-1}e^tf(z)=-(Df_t(z))^{-1}\dot{f}_t(z)=h(z).$$
\hfill $\bigstar$\end{example}\vspace{0.4cm}

Elements of the class $S^0(\Do)$ enjoy the following inequalities, which are known as the Koebe distortion theorem when $n=1.$

\begin{theorem}[Corollary 8.3.9 in \cite{graham2003geometric}]\label{koebe1}
 If $f\in S^0(\B_n)$, then $$\frac{\|z\|}{(1+\|z\|)^2} \leq \|f(z)\|\leq \frac{\|z\|}{(1-\|z\|)^2} \quad \text{for all} \quad z\in \B_n. $$
In particular, $\frac1{4}\B_n \subseteq f(\B_n).$ (Koebe quarter theorem for the class $S^0(\B_n).$)
\end{theorem}

\begin{theorem}[Theorem 1 and Theorem 2 in \cite{MR1049182}]\label{koebe2}
 If $f\in S^0(\D^n)$, then $$\frac{\|z\|_\infty}{(1+\|z\|_\infty)^2} \leq \|f(z)\|_\infty\leq \frac{\|z\|_\infty}{(1-\|z\|_\infty)^2} \quad \text{for all} \quad z\in \D^n. $$
In particular, $\frac1{4}\D^n \subseteq f(\D^n).$ (Koebe quarter theorem for the class $S^0(\D^n).$)
\end{theorem}

We summarize some consequences of Theorem \ref{koebe1} and \ref{koebe2} in combination with Theorem \ref{denseness}.

\begin{corollary}\label{range}
Let $n\geq 2$, $f\in S^0(\Do)$ and let $\{f_t\}_{t\geq0}$ be a normalized Loewner chain with $f=f_0$ such that $\{e^{-t}f_t\}_{t\geq0}$ is a normal family. Then
\begin{itemize}
 \item[a)] $\bigcup_{t\geq0}f_t(\Do)=\C^n$, $f(\Do)$ is a Runge domain and $f$ can be approximated locally uniformly by automorphisms of $\aut(\C^n).$
 \item[b)] $S^0(\Do)$ is a proper subset of $S(\Do)\cap F(\Do).$
\item[c)] $ \{\textbf{id}_{\Do}\} \subsetneq S^0(\Do)\cap \operatorname{Aut}(\C^n)\subsetneq S(\Do)\cap \aut(\C^n).$ 
\end{itemize}
 
\end{corollary}
\begin{proof}
Theorem \ref{koebe1} and \ref{koebe2} imply $$\bigcup_{t\geq 0}f_t(\Do)\supseteq \bigcup_{t\geq 0}\left(\frac{e^t}{4}\cdot \Do\right)=\C^n.$$ Consequently, $f\in E(\Do)$ and a) follows from Theorem \ref{denseness}.\\
$b):$ $S^0(\Do)$ is compact, but the set $S(\Do)\cap F(\Do)$ is not a normal family:
 Let $g:\C\to \C$ with $g(0)=g'(0)=0$ and let $F_g(z_1,...,z_n):= (z_1,...,z_{n-1}, z_n+g(z_1)).$ Then $F_g\in S(\Do)\cap\aut(\C^n)\subset S(\Do)\cap F(\Do),$ but the set of all $F_g$'s is not a normal family.
\\
$c):$ $S^0(\Do)\cap \aut(\C^n)$ is a normal family but we have seen that $S(\Do)\cap\aut(\C^n)$ is not normal.\\ Furthermore, there exists $\eps>0$ such that $F_g \in S^*(\Do)\subset S^0(\Do)$ for all $g$ with $|g(z)|\leq \eps$ for all $z\in \overline{\Do},$ which can be easily checked by using Theorem \ref{Juergen}. Thus, $S^0(\Do)\cap \aut(\C^n)$ contains infinitely many elements.
\end{proof}
\begin{question}
 Is $ S^0(\Do)\cap \operatorname{Aut}(\C^n)$ dense in $S^0(\Do)$?
\end{question}

 \section{Support points of \texorpdfstring{$S^0(\Do)$}{}}\label{Oana}

Let $X$ be a locally convex $\C$--vector space and $E\subset X.$ The set $ex E$ of extreme points  and the set $\supp E$ of support points  of $E$ are defined as follows:
\begin{itemize}
\item[\textbullet] $x\in ex E$ if the representation $x=ta + (1-t)b$ with $t\in[0,1],$ $a,b\in E$,  always implies $x=a=b.$
\item[\textbullet] $x\in supp E$ if there exists a continuous linear functional $L:X\to\C$ such that $\Re L$ is nonconstant on $E$ and $$\Re L(x) = \max_{y\in E}\Re L(y).$$
\end{itemize}

The class $S^0(\Do)$ is a nonempty compact subset of the locally convex vector space $\Ho(\Do, \C^n).$ Thus the Krein--Milman theorem implies that $ex S^0(\Do)$ is nonempty. Of course, $\supp S^0(\Do)$ is nonempty too: Let $f=(f_1,...,f_n)\in \Ho(\Do,\C^n)$ and $f_1(z_1,...,z_n)=a_f z_1+b_fz_1^2+...$ Then the map $f\mapsto b_f$ is an example for a continuous linear functional on $\Ho(\Do,\C^n)$, which is nonconstant on $S^0(\Do).$\\

Extreme points as well as support points of the class $S^0(\mathcal{D}_1)=S$ map $\D$ onto $\C$ minus a slit (which has increasing modulus when one runs through the slit from its starting point to $\infty$), see Corollary 1 and 2 in \cite{MR708494}, \S9.5. In particular, they are unbounded mappings. It would be quite interesting to find similar geometric properties of extreme and support points of $S^0(\Do)$ when $n\geq 2.$ However, it is not even known whether support points of $S^0(\Do)$ are always unbounded in this case.
In the following, we touch this question and our aim is to prove the following theorem which is already known to be true for the case $n=1,$ i.e. for the class $S.$ 

\begin{theorem}\label{sup} 
Let $f\in\supp S^0(\Do)$ and let $\{f_t\}_{t\geq0}$ be a normalized Loewner chain with $f_0=f$ such that $\{e^{-t}f_t\}_{t\geq 0}$ is a normal family on $\Do,$ then $e^{-t}f_t\in\supp S^0(\Do)$ for all $t\geq 0.$ 
\end{theorem}

The proof shows that a special class of bounded mappings cannot be support points -- a class which is actually conjectured to be the set of all bounded mappings in $S^0(\Do)$. 
Theorem \ref{sup} was conjectured in \cite{MR2943779} for $\Do=\B_n$, as the analogous statement for extreme points of $S^0(\B_n)$ could be shown there. It is also true for $S^0(\D^n)$: 

\begin{theorem}[See \cite{MR2943779}]\label{extreme} 
Let $f\in ex S^0(\Do)$ and let $\{f_t\}_{t\geq0}$ be a normalized Loewner chain with $f_0=f$ such that $\{e^{-t}f_t\}_{t\geq 0}$ is a normal family on $\Do,$ then $e^{-t}f_t\in ex S^0(\Do)$ for all $t\geq 0.$ 
\end{theorem}
 Though the proof from \cite{MR2943779} readily generalizes to $S^0(\Do)$, we will also give a proof for the sake of completeness. Our proof for Theorem \ref{sup} generalizes ideas from a proof for the case $n=1$, which is described in \cite{hallenbeck1984linear}.  \\

First, we note the important fact  that, given an evolution family $\varphi_{s,t}$ associated to a $\M(\Do)$--Herglotz vector field and a $G\in S^0(\Do)$, then $e^{t-s} G(\varphi_{s,t})$ is also in $S^0(\Do),$ which is mentioned in the proof of Theorem 2.1 in \cite{MR2943779}.

\begin{lemma}\label{s0}
Let $G\in S^0(\Do)$ and $t\geq 0.$ Furthermore, let  $\{f_u\}_{u\geq0}$ be a normalized Loewner chain  such that $\{e^{-u}f_u\}_{u\geq0}$ is a normal family and let $\varphi_{s,t}$ be the associated evolution family.  Then $e^{t-s} G(\varphi_{s,t}) \in S^0(\Do)$ for every $0\leq s \leq t.$
\end{lemma}
\begin{proof}
Let $\{G(\cdot,u)\}_{u\geq0}$ be a normalized Loewner chain with $G(\cdot,0)=G$ such that $\{e^{-u}G(\cdot,u)\}_{u\geq0}$ is a normal family and let  $F(z,u):\Do\times [0,\infty)\to \C^n$ be the mapping $$F(z,u)=\begin{cases}
e^{t-s} G(\varphi_{s+u,t}(z)), &\quad 0\leq u \leq t-s,\\
e^{t-s} G(z, u+s-t), &\quad u>t-s.                                                                   \end{cases}
$$ Then $\{F(\cdot,u)\}_{u\geq0}$ is a normalized Loewner chain, $F(\cdot,0)=e^{t-s}G(\varphi_{s,t})$ and $\{e^{-u}F(\cdot,u)\}_{u\geq 0}$ is a normal family. Thus $e^{t-s}G(\varphi_{s,t})\in S^0(\Do).$ 
\end{proof}

\begin{proof}[Proof of Theorem \ref{extreme}.]
 Suppose that $e^{-t}f_t\not\in exS^0(\Do)$ for some $t>0.$ Then $e^{-t}f_t=sa+(1-s)b$ for some $a,b\in S^0(\Do)$ with $a\not=b$ and $s\in(0,1).$ As $f=f_t\circ \varphi_{0,t},$ we have
$$ f = s \cdot (e^ta\circ \varphi_{0,t}) +  (1-s)\cdot (e^tb\circ \varphi_{0,t}).$$
The functions $e^ta\circ \varphi_{0,t}$ and $e^tb\circ \varphi_{0,t}$ belong to $S^0(\Do)$ according to Lemma \ref{s0}. Thus, as $f\in exS^0(\Do)$, they are identical and the identity theorem implies $a=b$, a contradiction.
\end{proof}

Choosing $G(z)=z$ in Lemma \ref{s0} shows that $e^{t-s}\varphi_{t-s}\in S^0(\Do).$

\begin{lemma}\label{pol}
Let $\varphi_{s,t}$ be defined as in Lemma \ref{s0} and let $h= e^{t-s}\varphi_{s,t}\in S^0(\Do).$ Furthermore, let  $P:\C^n\to\C^n$ be a polynomial with $P(0)=0,$ $DP(0)=0,$ then there exists $\delta>0$ such that $$h+\eps e^{t-s} P(e^{s-t}h) \in S^0(\Do)\quad \text{for all} \quad \eps\in\C \quad \text{with} \quad |\eps|<\delta.$$
\end{lemma}
\begin{proof} Let $g_\eps(z)=z+\eps P(z).$
Obviously we have $g_{\eps}(0)=0, \; Dg_{\eps}(0)=I_n.$\\
Now $\det(Dg_{\eps}(z))\to 1$ for $\eps\to 0$ uniformly on $\overline{\Do},$ so $g_{\eps}$ is locally biholomorphic for $\eps$ small enough. In this case, for every $z\in \overline{\Do}$, we have:\\ $$[Dg_{\eps}(z)]^{-1}=[I_n+\eps DP(z)]^{-1}=I_n-\eps DP(z) + \eps^2DP(z)^2+... = I_n-\eps\underbrace{(DP(z)+...)}_{:=U(z)\in \C^{n\times n}}.$$ 
Write $[Dg_{\eps}(z)]^{-1}g_{\eps}(z) = z+\eps P(z)-\eps U(z)z-\eps^2 U(z)P(z)=$
$(I_n+\eps M(z))z,$ with a matrix-valued function $M(z).$\\
Now we show that $g_\eps\in S^*(\Do)$ for $|\eps|$ small enough and we distinguish between the unit ball and the polydisc.\\
Case 1: $\Do=\B_n:$\\ Here, $\left< [Dg_{\eps}(z)]^{-1}g_{\eps}(z),z\right>=\left< (I_n+\eps M(z))z, z\right>$ and there is an $\delta>0$ such that $I_n+\eps M(z)$ has only eigenvalues with positive real part for all $\eps\in\C$ with $|\eps|<\delta$. In this case we have $$ \Re\left< [Dg_{\eps}(z)]^{-1}g_{\eps}(z),z\right>>0 \quad \forall z\in \B_n\setminus\{0\}$$ and thus $g_{\eps}\in S^*(\B_n)\subset S^0(\B_n)$ by Theorem \ref{Juergen}. \\
Case 2: $\Do=\D^n:$\\
Let $g_j(z)$ be the $j$-th component of $[Dg_{\eps}(z)]^{-1}g_{\eps}(z)$. For $\eps\to 0,$ the function $g_j(z)/z_j$ converges uniformly to $1$ on the set $K:=\overline{\{z\in\D^n \with  \|z\|_{\infty}=|z_j|>0\}}.$ Thus there exists $\delta >0$ such that $$\Re \left(\frac{g_j(z)}{z_j}\right)>0\quad \text{for all} \quad z\in K, j=1,...,n\quad \text{and all} \quad \eps\in\C \quad \text{with} \; |\eps|<\delta.$$ Hence, $g_\eps \in S^*(\D^n)\subset S^0(\D^n)$ for all $\eps$ small enough by Theorem \ref{Juergen}.\\

From Lemma \ref{s0} it follows that $e^{t-s}g_\eps(\varphi_{s,t})=e^{t-s}g_\eps(e^{s-t}h)=h+\eps e^{t-s}P(e^{s-t}h)\in S^0(\Do).$\end{proof}

The next Lemma shows that a special class of bounded mappings are not support points of $S^0(\Do).$
 
\begin{proposition}\label{marina}
 Let $\varphi_{s,t}$ be defined as in Lemma \ref{s0} and let $h=e^{t-s} \varphi_{s,t} \in S^0(\Do).$ Then $h$ is not a support point of $S^0(\Do).$
\end{proposition}
\begin{proof}
Assume that $h$ is a support point of $S^0(\Do),$ i.e. there is a continuous linear functional $L:\Ho(\Do,\C^n)\to \C$ such that $\Re L$ is nonconstant on $S^0(\Do)$ and $$\Re L(h)=\max_{g\in S^0(\Do)}\Re L(g).$$ Let $P$ be a polynomial with $P(0)=0$ and $DP(0)=0.$ Then $h+\eps e^{t-s} P(e^{s-t}h)\in S^0(\Do)$ for all $\eps\in\C$ small enough by Lemma \ref{pol}.\\
We conclude $$\Re L(P(e^{s-t}h))=\Re L(P(\varphi_{s,t}))=0,$$ otherwise we could choose $\eps$ such that $\Re L(h+\eps e^{t-s}P(e^{s-t}h))>\Re L(h).$\\ Now $\varphi_{s,t}(\Do)$ is a Runge domain by Theorem \ref{runge}. Hence we can write any analytic function $g$ defined in $\Do$ with $g(0)=0$ and $Dg(0)=0$ as $g=\lim_{k\to \infty}P_k(\varphi_{s,t}),$ where every $P_k$ is a polynomial with $P_k(0)=0$ and $DP_k(0)=0$, according to Lemma \ref{ru} c). The continuity of $L$ implies $\Re L(g)=0.$ Hence $\Re L$ is constant on $S(\Do),$ a contradiction.
\end{proof}

\begin{proof}[Proof of Theorem \ref{sup}.]
 Let $L$ be a continuous linear functional on $\Ho(\Do,\C^n)$ such that $\Re L$  is nonconstant on $S^0(\Do)$ with $$\Re L(f)=\max_{g\in S^0(\Do)} \Re L(g).$$
Fix $t\geq 0,$ then $f(z)=f_t(\varphi_{0,t}(z))$ for all $z\in \Do$. Define the continuous linear functional  $$J(g):=L(e^t\cdot g\circ \varphi_{0,t})  \quad \text{for}\quad g\in \Ho(\Do,\C^n).$$
Now we have $$J(e^{-t} f_t)=L(f)\quad \text{and}\quad \Re J(g)\leq \Re J(e^{-t}f_t) \quad \text{for all}\quad g\in \Ho(\Do,\C^n).$$
Furthermore, $\Re J$ is not constant on $S^0(\Do)$: as $e^t\varphi_{0,t}$ is not a support point of $S^0(\Do)$ by Proposition \ref{marina}, we have $\Re J(\text{id})=\Re L(e^t \varphi_{0,t})< \Re L(f)=\Re J(e^{-t}f_t).$
\end{proof}
In view of Proposition \ref{marina}, it is natural to state the following conjecture.
\begin{conj}\label{Conny1}
 If $f\in S^0(\Do)$ is bounded, then $f$ is not a support point of $S^0(\Do).$
\end{conj}
We note two further closely related problems. 
\begin{conj}\label{Conny2}
 If $f\in S^0(\Do)$ is bounded by $M>0$ and $G\in S^0(\Do),$ then $MG(\frac1{M}f)\in S^0(\Do).$
\end{conj}
\begin{conj}\label{Conny3}(Conjecture 2 in \cite{MR2943779})
 If $G\in S^0(\Do)$ is bounded by $M>0,$ then $\frac1{M}G=\varphi_{0,t}$ for some $t\geq 0$ and some evolution family $\varphi_{s,t}$ that corresponds to a $\M(\Do)$--Herglotz vector field.
\end{conj}
The last problem should be the most difficult as we have seen that Conjecture \ref{Conny3} implies Conjecture \ref{Conny2} (because of Lemma \ref{s0}) and that Conjecture \ref{Conny2} implies Conjecture \ref{Conny1}, for we could replace $\varphi_{s,t}$ by $\frac1{M}f$ in Lemma \ref{pol} and Proposition \ref{marina} provided that Conjecture \ref{Conny2} is true.

\begin{remark}
 The situation for extreme points is quite the same, see Proposition 1 in \cite{MR2943779} and its proof: It is known that $e^{t-s}\varphi_{s,t}$ cannot be an extreme point, but there could still be bounded extreme points of $S^0(\Do).$ Also, already Conjecture \ref{Conny2} would imply that all extreme points are unbounded mappings. 
\end{remark}

We finish this section by proving the following statement which touches Conjecture \ref{Isrunge}. (Note, however, that the statement is weaker than the one mentioned in Remark \ref{Fil}.)

\begin{proposition}\label{Fredi}
 Let $f\in F(\Do)$ map a neighborhood $U$ of $\overline{\Do}$ biholomorphically onto a Runge domain. Then $f\in E(\Do).$
\end{proposition}
\begin{proof} Obviously, the statement is true for $f$ if and only if it is true for $Df(0)^{-1}(f-f(0)).$ So assume that $f(0)=0, Df(0)=I_n.$\\
 Let $\phi_n$ be a sequence of automorphisms that approximate $f$ locally uniformly on $U$ with $\phi_n(0)=0, D\phi_n(0)=I_n$. Then, $\phi_n$ converges to $f$ uniformly on $\overline{\Do}.$ Choose $N\in\N$ large enough so that $\phi_N^{-1}\circ f$ is starlike on $\Do.$ This can be done analogously to the proof of Lemma \ref{pol}: Just write $(\phi_n^{-1}\circ f)(z)=z+ G_n(z).$ Then $G_n(z)\to0$ uniformly on $\overline{\Do}$.\\
Consequently, $f$ is the first element of the Loewner chain $\{\phi_N\circ (e^t \cdot \phi_N^{-1}\circ f)\}_{t\geq0},$ hence $f\in E(\Do).$
\end{proof}

\section{A remark on the Roper--Suffridge extension operator}\label{Ikkei}

In this section we only consider the case $\Do=\B_n.$\\
In order to get (nontrivial) examples for biholomorphic mappings for $n\geq 2,$ one can use extension operators, which are mappings of the form $\Theta: \mathcal{E}\to U(\B_{k+m}),$ $\mathcal{E}\subset U(\B_k)$ and $k,m\in \N.$  One example for such an operator is the classical Roper--Suffridge extension operator\\ $\Psi_{n}: S\to S(\B_{n})$ with $n\geq 2$ and  
\begin{equation}\label{Roper}
 \Psi_n(f)(z_1,...,z_n)=(f(z_1),z_2\sqrt{f'(z_1)},...,z_n\sqrt{f'(z_1)})
\end{equation}
(see Section 11.3 in \cite{graham2003geometric}). Here, the branch of the square root is chosen such that $\sqrt{1}=1.$
 $\Psi_n$ has many nice properties:
\begin{itemize}
 \item[(1)] $\Psi_n$  maps starlike / convex / Bloch mappings again onto starlike / convex / Bloch mappings (see section 11.1 in \cite{graham2003geometric}). 
 \item[(2)] It can easily be seen that $\Psi(f)(\B_n)$ is always a Runge domain by using Lemma \ref{ru} b), i.e. by calculating $\Psi(f)^{-1}$ and using the polynomial Runge theorem for the case of one variable. In fact, even more is true:  $\Psi_n(S)\subseteq S^0(\B_n)$ for all $n>1$ and $\Psi_n$ \textit{preserves normalized Loewner chains}, i.e., for any   normalized Loewner chain $\{f_t\}_{t\geq0}$ on $\D$, the family $\{e^t\Psi_n(e^{-t}f_t)\}_{t\geq0}$ is a normalized Loewner chain on $\B^{n}$ (see Theorem 11.3.1 in \cite{graham2003geometric} with $\alpha=\frac1{2}$).
 \item[(3)] If $f\in S$ is unbounded / a support point of $S$, then $\Psi_n(f)$ is also unbounded / a support point of $\Psi_n(S)$, because the first coordinate of $\Psi_n(f)(z_1,...,z_n)$ is just $f(z_1).$
\end{itemize}

Since its introduction,  many extension operators similar to $\Psi_n$ have been found. Some of them have been constructed in order to obtain geometric properties in the sense of (1), some others can be used to extend Loewner chains, comparable to property (2).\\

Another way to extend Loewner chains is to extend Herglotz vector fields, and thus the Loewner equations, directly. One can construct mappings $\xi: \mathcal{E}\to \mathcal{I}(\B_{k+m}),$ $\mathcal{E}\subset \mathcal{I}(\B_{k}),$ and use them to extend Herglotz vector fields on $\B_k$ to Herglotz vector fields on $\B_{k+m}.$

\begin{example} If $g(z,t)$ and $h(z,t)$ are Herglotz vector fields on $\D$, then $(g(z_1,t), h(z_2,t))$ is a Herglotz vector field on $\B_2.$ This follows immediately from the fact that $(\psi_t(z_1), \phi_t(z_2))$ is a semigroup on $\B_2$ when both $\psi_t$ and $\phi_t$ are semigroups on $\D.$ \hfill $\bigstar$
\end{example}

\begin{example} The mapping $\xi: \mathcal{M}_1\to \mathcal{M}_2,$ $-zp(z) \mapsto$ \\$\left(-z_1 p(z_1),-\frac{z_2}{2}(p(z_1)+z_1p'(z_1)+1) \right),$ with $(z_1,z_2)\in \B_2,$ produces infinitesimal generators on $\B_2$ (see, e.g., Theorem 7.1 in \cite{MR3043148} and its proof). It corresponds to the Roper-Suffridge operator $\Psi_2$ in the following sense: If we take a semigroup $\Theta_t$ with infinitesimal generator $G(z)=-zp(z)$ in $\D$ and differentiate $e^{t}\Psi_n(e^{-t}\Theta_t)$ with respect to $t$, we get:
\begin{eqnarray*} &&\lim_{t\to 0}\frac{e^{t}\Psi_n(e^{-t}\Theta_t)(z)-z}{t}=\lim_{t\to 0}\left(\frac{\Theta_t(z_1)-z_1}{t},z_2\frac{\sqrt{e^t\Theta_t'(z_1)}-1}{t}\right)\\
&&=\left(G(z_1),\frac{z_2}{2} (G'(z_1)+1)\right)=\left(-z_1p(z_1),-\frac{z_2}{2} (p(z_1)+z_1p'(z_1)+1)\right).
\end{eqnarray*} \hfill $\bigstar$
\end{example}

\begin{example}
 A similar, more abstract approach to find further examples for $\xi$-mappings by extending semigroups in a unit ball to semigroups in a higher dimensional unit ball can be found in  \cite{elin}. \hfill $\bigstar$
\end{example}

Property $(3)$ of $\Psi_n$ implies that there exist (infinitely many) unbounded support points of $\Psi_n$. Now we will prove Conjecture 3.1 in \cite{MR2341607}, which says that all support points of $\Psi_n(S)$ are in fact unbounded mappings. This, in turn, implies the following result.
\begin{theorem}\label{supp}
 Let $f\in S$ and $F=\Psi_n(f)$. Also, let $\{f_t\}_{t\geq0}$ be a normalized Loewner chain on $\D$ with $f_0=f$ and let $F_t(z)=e^t\Psi_n(e^{-t}f_t)(z).$ If $F\in \supp\Psi_n(S),$ then $e^{-t}F_t\in \supp\Psi_n(S)$ for all $t\geq 0.$ 
\end{theorem}

In order to simplify notation, we only look at the Roper-Suffridge-Operator $\Psi:=\Psi_2:$
 $$\Psi(f)(z_1,z_2)=\left(f(z_1),z_2\sqrt{f'(z_1)}\right).$$

\begin{proposition}
Let $f\in S$ be bounded. Then $\Psi(f)$ is not a support point of $\Psi(S).$  \\
In particular, all support points of $\Psi(S)$ are unbounded.
\end{proposition}
\begin{proof}
 Let $f\in S$ be bounded and let $L:\Ho(\B_n,\C^n)\to\C$ be a continuous linear functional such that $\Re L$ is nonconstant on $\Psi(S)$ and 
$$ \Re L(\Psi(f)) = \max_{G\in\Psi(S)}\Re L(G). $$ 
For every $n\geq 1$ there is an $\eps_0>0$ such that $f+\frac{\eps}{n+1} f^{n+1}\in S,$ for all $\eps\in\C$ with $|\eps|<\eps_0,$ as $f$ is bounded. It follows \begin{align*}\Psi(f+\frac{\eps}{n+1} f^{n+1})(z_1,z_2)=\left(f(z_1)+\frac{\eps}{n+1} f(z_1)^{n+1}, z_2\sqrt{f'(z_1)}\sqrt{1+\eps f(z_1)^n}\right)=\\ \left(f(z_1)+\frac{\eps}{n+1} f(z_1)^{n+1}, z_2\sqrt{f'(z_1)}+\eps z_2\sqrt{f'(z_1)}(1/2\cdot f(z_1)^n+...)\right)=\\
\left(f(z_1), z_2\sqrt{f'(z_1)}\right) +\eps\underbrace{\left(\frac{1}{n+1} f(z_1)^{n+1},  z_2/2\sqrt{f'(z_1)}f(z_1)^n\right)}_{=:P_n}+\left(0, \LandauO(|\eps|^2)\right)\in\Psi(S).\end{align*}
If $\Re L(P_n)\not=0,$ then we can choose $\eps$ such that $\Re L(\Psi(f+\frac{\eps}{n+1} f^{n+1}))> \Re L(\Psi(f)),$ a contradiction. Hence 
\begin{equation}\label{0}
 \Re L\left(\frac{1}{n+1} f(z_1)^{n+1},  z_2/2\cdot\sqrt{f'(z_1)}f(z_1)^n\right)=0 \quad \forall n\geq 1.\end{equation}
Now we can repeat this argument for the $\eps$--terms of higher order, because all coefficients of the expansion $\sqrt{1+x}=1+\frac1{2}x\mp...$ are $\not=0$, and we get 
\begin{equation}\label{1}
 \Re L\left(0,  z_2\sqrt{f'(z_1)}f(z_1)^{jn}\right)=0 \quad \forall j\geq 2, n\geq1.
\end{equation}
Now consider an arbitrary function of the form $(0,z_2g(z_1))$ with $g(0)=g'(0)=0.$ Write $z_2g(z_1)=z_2\sqrt{f'(z_1)}\cdot \frac{g(z_1)}{\sqrt{f'(z_1)}}$ and approximate the second factor by a sequence of polynomials in $f$ (see Lemma \ref{ru} c)): $$z_2g(z_1)=z_2\sqrt{f'(z)}\cdot \sum_{k\geq 2} a_k f(z_1)^k=\sum_{k\geq2}a_kz_2\sqrt{f'(z)}f(z_1)^k.$$
By using (\ref{1}), it follows \begin{equation}\label{2}
 \Re L(0,z_2g(z_1))=0 \quad  \text{for all }\; g\in\Ho(\D,\C) \;\; \text{with} \quad g(0)=g'(0)=0.\end{equation} We can apply this to (\ref{0}) for $n\geq 2$ to get
$ \Re L\left(f(z_1)^{n},  0\right)=0$ for all $n\geq 3,$
and by Runge approximation it follows
\begin{equation}\label{3}
 \Re L\left(g(z_1),  0\right)=0 \quad  \text{for all }\; g\in\Ho(\D,\C) \;\; \text{with} \quad g(0)=g'(0)=g''(0)=0.\end{equation}
Now let $H(z_1,z_2)=(h(z_1),z_2\sqrt{h'(z_1)})\in \Psi(S).$ If $h(z)=z+a_2z^2+...,$ then $$\sqrt{h'(z)}=\sqrt{1+2a_2z+...}=1+a_2z+...$$ Thus $(h(z_1),z_2\sqrt{h'(z_1)})=(z_1+a_2z_1^2+...,z_2+a_2z_2z_1+...)$ and with (\ref{2}) and (\ref{3}):
\begin{equation*}
\Re L(H)=d_1 + d_2\cdot a_2 + d_3 + d_4\cdot a_2
\end{equation*}
 with $d_1=\Re L(z_1,0),$ $d_2=\Re L(z_1^2,0),$ $d_3=\Re L(0,z_2),$ $d_4=\Re L(0,z_2z_1).$
Finally, use (\ref{0}) with $n=1$ to get 
 \begin{equation*}
  \Re L\left(z_1^{2},  z_2z_1\right)=0 \iff d_2+d_4=0, 
  \end{equation*}
which implies $\Re L(H)=d_1+d_3.$ Hence, $\Re L$ is constant on $\Psi(S)$ and $\Psi(f)$ cannot be a support point of $\Psi(S).$
\end{proof}

\begin{proof}[Proof of Theorem \ref{supp}.]
The proof is now quite the same as the proof of Theorem \ref{sup}. We just have to replace $S^0(\B_n)$ by $\Psi(S)$, see Remark 3.1 in \cite{MR2341607}.
\end{proof}

\begin{remark}
The corresponding statement of Theorem \ref{supp} for extreme points instead of support points also holds true, see Theorem 3.1 in \cite{MR2341607}.
\end{remark}

\section{Further problems concerning the class \texorpdfstring{$S^0(\Do)$}{}}\label{Daniela}

\subsection{Projections}
It is possible to define counterparts to the extension operators of infinitesimal generators of the last section, e.g. ``projections'' of the form $\Theta: \mathcal{E}\to \mathcal{I}(\D),$ $\mathcal{E}\subset \mathcal{I}(\B_n).$ One possibility is to use the so-called \emph{Lempert projection devices}, see \cite{MR2578602}, Section 1.3: 

Let $\phi: \D\to \B_n$ be a complex geodesic, i.e. $\phi$ is holomorphic and preserves the Kobayashi metrics. For the unit ball, $\phi$ just parameterizes the slice $\B_n\cap \{z+h(w-z)\with h\in\C\},$ for two given points $z,w\in \overline{\B_n},$ $z\not=w.$ \\
Furthermore, let $\rho_{\phi}$ be the \emph{Lempert projection}, which is in this case just the orthogonal projection onto $\phi(\D)$ and finally define $\tilde{\rho}_\phi: \B_n\to \D, \tilde{\rho_\phi}:= \phi^{-1}\circ \rho_\phi.$\\ 
The triple $(\phi, \rho_\phi, \tilde{\rho_\phi})$ is called \emph{Lempert projection device}. For a holomorphic vector field $G:\B_n\to\C^n$ we define $\xi_{\phi,G}: \D \to \C$ by $$\xi_{\phi,G}(z):= D\tilde{\rho}_\phi (\phi(z)) \cdot G(\phi(z)).$$ 
The important connection to infinitesimal generators is given by the following result, see Proposition 4.5 in \cite{MR2578602}.
\begin{theorem}
If $G\in \mathcal{I}(\B_n),$ then $\xi_{\phi,G}\in \mathcal{I}(\D).$
\end{theorem}

Together with this result, Lempert projection devices are a useful tool to study properties of infinitesimal generators in the unit ball, see, e.g. \cite{bbig}, where a  Julia-Wolff--Carath\'{e}odory type theorem for  infinitesimal generators in the unit ball is proven.

\begin{example}\label{Pope}
 A geodesic through $0$ has the form $\phi(z)=z\cdot w,$ with $w\in \partial \B_n.$ Here we get $$\rho_\phi(z)=\left<z,w \right>\cdot w, \quad \tilde{\rho}_\phi(z)=\left<z,w \right> \quad \text{and} \quad \xi_{\phi,G}(z)=\left<G(z\cdot w),w \right>.$$ 

If $G\in \M(\B_n),$ then $\xi_{\phi,G}\in \M(\D)$. Furthermore, if $G(z,t)$ is a Herglotz vector field on $\B_n$ with $G(\cdot, t)\in \M(\B_n)$ for almost all $t\geq 0,$ then $g(z,t):=\xi_{\phi,G(\cdot,t)}(z)$ is a Herglotz vector field on $\D$ with  $g(\cdot,t)\in\M(\D)$ for almost all $t\geq0.$ Thus we get a ``projection'' of the Loewner ODE
\begin{equation}\label{G1}\dot{\varPhi}_{s,t}=G(\varPhi_{s,t},t),\; \varPhi_{s,s}(z)=z\in\B_n, \end{equation}
to \begin{equation}\label{G2} \dot{\varphi}_{s,t}=\left<G(\varphi_{s,t}\cdot w,t),w \right>,\; \varphi_{s,s}(z)=z\in\D.
 \end{equation}
\hfill $\bigstar$
\end{example}

\begin{question}
 Define $F\in S^0(\B_n)$ by $F:=\lim_{t\to\infty}e^t \varPhi_{0,t}$ and $f_w\in S$ by $f_w:=\lim_{t\to\infty}e^t \varphi_{0,t}.$ How is $F$ related to $f_w$? For example: If $F$ is unbounded, can we find $w\in \partial \B_n$ such that $f_w$ is unbounded, too? 
\end{question}

From Example \ref{Pope} we immediately get the following result which generalizes the well-known coefficient estimate of the class $\mathcal{P}.$

\begin{proposition}
If $G\in \M(\B_n),$ then $$ \left|\left<D^kG(0)(w,w,...,w),w \right> \right|\leq 2 $$
 for all $k\geq 2$ and $w\in \partial \B_n.$
\end{proposition}
\begin{proof} For $n=1,$ we have the classical coefficient estimate for the Carath\'{e}odory class.\\
 Now let $n\geq 2.$
 In fact, it is enough to prove the inequality only for $w=e_1=(1,0,...,0)$: Let $G\in \M(\B_n)$ and let $w\in \partial \B_n.$ Consider the ``rotation'' $H:\B_n\to\C^n, z\mapsto U^*G(Uz),$ where $U$ is a unitary matrix with $Ue_1=w.$\footnote{For $A\in\C^{n\times n},$ we denote by $A^*$ the adjoint matrix of $A,$ i.e. its conjugate transpose.} Then $H\in \M(\B_n),$ and \begin{eqnarray*}
  \left<D^kH(0)(e_1,e_1...,e_1),e_1\right>=\left<U^*D^kG(0)(Ue_1,Ue_1...,Ue_1),e_1\right>=\\
\left<D^kG(0)(Ue_1,Ue_1...,Ue_1),Ue_1\right>=\left<D^kG(0)(w,w...,w),w\right>.                                                                                               \end{eqnarray*}
Now write $G(z_1,...,z_n)=\begin{pmatrix}
G_1(z_1,...,z_n)\\
\vdots
\\ G_n(z_1,...,z_n)\\
            \end{pmatrix}
$ and $G_1(z_1,0,...,0)=-z_1+a_2z_1^2+a_3z_1^3+...$\\
 We have to show that  $|a_n|\leq 2$ for all $n\geq 2.$ But as $G_1(z_1,0,...,0)=\left<G(z_1 e_1),e_1 \right>,$ we know that $G_1(z_1,0,...,0)=-z_1p(z_1)$ for a Carath\'{e}odory function $p\in\mathcal{P}.$ Thus the statement follows from the case $n=1.$
\end{proof}

There is an unproven version of the Bieberbach conjecture for the class $S^0(\B_n),$ namely:
\begin{conj}[see \cite{graham2003geometric}, p. 343]\label{Bieber}
If $f\in S^0(\B_n),$ then $$ \left|\frac1{k!}\left<D^kf(0)(w,w,...,w),w \right> \right|\leq k $$
 for all $k\geq 2$ and $w\in \partial \B_n.$
\end{conj}
Again, it is enough to prove this inequality only for $w=e_1=(1,0,...,0)$:\\
 Let $f\in S^0(\B_n)$ and let $w\in \partial \B_n.$ Consider the ``rotation'' $g:\B_n\to\C^n, z\mapsto U^*f(Uz),$ where $U$ is a unitary matrix with $Ue_1=w.$ Then $g\in S^0(\B_n),$ because $g=\lim_{t\to \infty}e^t \Omega_{0,t}$, where $\Omega_{s,t}$ is an evolution family that corresponds to the Herglotz vector field $(z,t)\mapsto U^*G(Uz,t).$ Moreover, \begin{eqnarray*}
  \left<D^kg(0)(e_1,e_1...,e_1),e_1\right>=\left<U^*D^kf(0)(Ue_1,Ue_1...,Ue_1),e_1\right>=\\\left<D^kf(0)(Ue_1,Ue_1...,Ue_1),Ue_1\right>=\left<D^kf(0)(w,w...,w),w\right>.
 \end{eqnarray*}

Let $f\in S^0(\B_n)$ and write $f(z_1,...,z_n)=\begin{pmatrix}
f_1(z_1,...,z_n)\\
\vdots\\
f_n(z_1,...,z_n)\\
           \end{pmatrix},$
and $f_1(z,0,...,0)=z+a_2z^2+a_3z^3+....$ For $w=e_1$,  Conjecture \ref{Bieber} states $|a_n|\leq n$ for all $n\geq 2.$
This is already known for $k=2$, see Corollary 8.3.15 in \cite{graham2003geometric}, and we give a simple proof by the projection technique. 
\begin{proposition}
 Conjecture \ref{Bieber} holds for $k=2.$
\end{proposition}
\begin{proof}
 We only have to consider $w=e_1$ and to simplify notation we only look at $n=2.$ Let $f=(f_1,f_2)=\lim_{t\to \infty} e^t \varPhi_{0,t}$, where $\varPhi_{s,t}$ is an evolution family satisfying (\ref{G1}). Write $f_1(z,0)=z+\sum_{k=0}^\infty a_kz^k$ and \\
$$G(z,t)=\binom{-z_1+\sum_{j+k\geq 2}g_{j,k}(t)z_1^jz_2^k)}{-z_2+\sum_{j+k\geq 2}h_{j,k}(t)z_1^jz_2^k}, \qquad \varPhi_{0,t}(z)=\binom{a_{1,0}(t)z_1+\sum_{j+k\geq 2}a_{j,k}(t)z_1^jz_2^k}{b_{0,1}(t)z_2+\sum_{j+k\geq 2}b_{j,k}(t)z_1^jz_2^k}.$$
The projection of $G(z,t)$ onto the geodesic through $0$ and $(1,0)$ gives the Herglotz vector field $$g(z_1,t)=\left<G((z_1,0),t),(1,0)\right>, \quad \text{i.e.} \qquad  g(z_1,t)=-z_1+\sum_{j\geq 2}g_{j,0}(t)z_1^j. $$
Denote by $\varphi_{s,t}$ the corresponding evolution family on $\D$ and write 
$$\varphi_{0,t}(z_1)=\sum_{j\geq 1}c_{j}(t)z_1^j.$$ Let $a_2=\lim_{t\to\infty}e^t a_{2,0}(t).$ We have to show that $|a_2|\leq 2.$
The Loewner ODEs for $\varPhi_{0,t}$ and $\varphi_{0,t}$ immediately imply $$a_{1,0}(t)=e^{-t} \quad \text{and} \quad \dot{a}_{2,0}(t)=-a_{2,0}(t)+g_{2,0}(t)e^{-2t},\quad a_{2,0}(0)=0,$$
as well as 
$$c_1(t)=e^{-t} \quad \text{and} \quad \dot{c}_2(t)=-c_2(t)+g_{2,0}(t)e^{-2t},\quad c_2(0)=0.$$
 Thus $a_2=\lim_{t\to \infty}c_2(t)$ is the second coefficient of the function  $r:=\lim_{t\to\infty}e^t\varphi_{0,t},$ which belongs to $S$ and thus $|a_2|\leq 2.$\end{proof}

\begin{remark}
 If $f\in S$ with $f(z)=z+\sum_{k\geq 2}a_kz^k$, then we have the additional information, that $|a_2|=2$ if and only if $f\in S$ is a rotation of the Koebe function. For $f=(f_1,,...,f_n)\in S^0(\B_n)$ with $f_1(z_1,0,...,0)=z_1+a_2z_1^2+...$, the proof shows that $|a_2|=2$ if and only if the function $r\in S$ is a rotation of the Koebe function. Of course, many elements of $S^0(\B_n)$ satisfy this.
\end{remark}

\begin{remark}
 There is also a similar conjecture for the class $S^0(\D^n),$ see \cite{graham2003geometric}, p. 343:\\
Is it true that $$ \left\| \frac1{k!}D^kf(0)(w,w,...,w) \right\|_\infty \leq k $$
for all $k\geq 2$ and $w\in \partial \D^n?$
\end{remark}

\subsection{Generalization of slit mappings}

Slit mappings play an important role in complex analysis. For example, as the set of those elements of $S$ that map $\D$ onto $\C$ minus a slit is dense in $S$, it is enough to solve extremal problems for those slit mappings only. The set of support points as well as the set of extreme points of $S$ consists of slit mappings and the ubiquitous Koebe function is the simplest form of a slit mapping (and the only starlike slit mapping in $S$). This led many people to the question:
$$\text{Are there analogs to slit mappings in higher dimensions?}$$ 
Of course, the answer depends on which property one would like to generalize.\\
In \cite{MR1845017} and \cite{MR2272135}, J. Muir and T. Suffridge study biholomorphic mappings on the unit ball whose image domains are unbounded convex domains that can be written as the union of lines parallel to some vector. In \cite{MR2254484}, the same authors investigate properties of extreme points of the set of all normalized convex mappings on $\B_n.$ \\
In contrast to such concrete geometric properties, one could also ask for mappings that can be embedded in Loewner chains and have similar properties with respect to Loewner chains, evolution families and infinitesimal generators:
\begin{enumerate}[(1)]
 \item Which $f\in S^0(\Do)$ embed in only one normalized Loewner chain $\{f_t\}_{t\geq0}$ such that $\{e^{-t}f_t\}_{t\geq0}$ is a normal family? Compare with Theorem \ref{Pommes} (c).
\item Which mappings $f\in S^0(\Do)$ embed into a Loewner chain whose Herglotz vector field $G(z,t)$ satisfies $G(\cdot,t)\in ex\M(\Do)$ for all $t\geq 0$? \\
Note that slit mappings satisfy this property according to Theorem \ref{Pommes} and the fact that $$ex(\M(\D))= \left\{z\mapsto -z\frac{u+z}{u-z}\with u\in\partial\D\right\},$$
which can be easily derived from the Riesz-Herglotz representation (\ref{Christian}).
\item Characterize the set $ex\M(\Do)$ for $n\geq 2.$ 
\end{enumerate}
 
It even seems to be hard to find examples of extreme points of $\M(\Do).$\\
In \cite{Voda}, Proposition 2.3.5, M. Voda constructs some explicit mappings that belong to $ex\M(\B_n).$ Furthermore, he shows that mappings of the form $(z_1,...z_n)\mapsto (-z_1p_1(z_1),...,-z_np_n(z_n)),$ with $p_1,...,p_n\in\mathcal{P},$ never belong to $ex\M(\B_n),$ see Proposition 2.3.1 in \cite{Voda}.\\

Finally we would like to point  out that the evolution families of some Herglotz vector fields on $\D^n$ show a similar behavior as evolution families on $\D$ that describe a slit connecting a point $p\in\partial\D$ to $0.$ For every $u\in\partial\D,$ the function $G(z)=-zp(z)$ with $p(z):= \frac{u+z}{u-z}$ is an infinitesimal generator of a semigroup on $\D.$ It has the property that  $$\Re(p(z)) = 0 \qquad \text{for all} \; z\in\partial\D\setminus\{u\}. $$
Slit mappings correspond to Herglotz vector fields of the form  $G(z,t)=-z\frac{u(t)+z}{u(t)-z}$ with a continuous function $u:[0,\infty)\to\partial\D.$ In some sense, the function $u$ corresponds to the tip of the slit. We can solve the Loewner ODE in this case also for initial values on $\partial\D\setminus\{u(0)\}.$ The solution may not exist for all $t\geq0,$ as it can hit the singularity $u(t)$ for some $t\geq0.$

If we pass on to the polydisc $\D^n,$ we can find similar Herglotz vector fields provided that we don't look at the whole boundary $\partial (\D^n),$ but at the so-called distinguished boundary $(\partial\D)^n.$
\begin{example}
 Let $n=2$ and consider the function
 $$F_g(z,w):=\left(-z\frac{(1-z)(1-g(w))}{1-zg(w)},0\right).$$ 
It can be easily checked that, if $g\in\Ho(\D,\D),$ then $F_g \in \mathcal{I}(\D^2).$ \\
For some special choices of $g$, the function $F_g$ will have an analytic continuation to a ``large'' subset $A\subset (\partial \D)^2$ such that 
$$\Re \frac{(1-z)(1-g(w))}{1-zg(w)}=0 \quad \text{for all} \quad (z,w)\in A.$$ Take, e.g., $g(w)=g_a(w)=a\cdot w, a\in\partial\D$ and define $U=\{(z,w)\in(\partial\D)^2 \with azw=1\}.$ $U$ is a closed curve on the torus $(\partial \D)^2.$ Then we have $\Re F_{g_a}(z,w)=0$ for all $(z,w)\in (\partial\D)^2\setminus U$. If $a:[0,\infty)\to\partial\D$ is a continuous function, then the evolution family corresponding to the Herglotz vector field $G((z,w),t)=F_{g_{a(t)}}(z,w)$ somehow describes the growth of a ``slit surface'' emerging from the torus $(\partial \D)^2.$ \hfill $\bigstar$
\end{example}

\newpage
\chapter{The chordal multiple-slit equation}

\section{Hydrodynamic normalization}

The solutions to the ordinary radial Loewner equation are univalent functions that belong to the class $$ S^\geq:=\{f\in \Ho(\D,\D)\with f(0)=0, f'(0)\geq 0, f \;\text{is univalent}\}. $$ 
From a geometric, dynamical point of view, this equation describes families of compact sets that grow from the boundary $\partial \D$ towards $0$ within $\D.$\\
What if we want to replace $0$ by another ``point of attraction'' $z_0$, meaning that we would like to generate univalent functions $f:\D\to\D$ with $f(z_0)=z_0$ and an additional condition for $f'(z_0)$?\\
The case $z_0\in \D\setminus\{0\}$ is not different from $z_0=0$ at all, because we can always transfer the first case into the second by applying an appropriate automorphism of $\D.$ The case $z_0\in\partial\D$, however, is different, as we have to assume that $f(z_0)$ can actually be \textit{defined}, e.g. in the sense of a radial limit. Without loss of generality we may assume that $z_0=1.$ Now, one usually maps the unit disc conformally onto the upper half-plane by the Cayley map $z\mapsto \frac{iz+i}{1-z}$ and asks for univalent functions of the form $f:\Ha\to\Ha$ with $f(\infty)=\infty$ in a certain sense. The reason of using $\Ha$ instead of $\D$ in this case is that such univalent mappings can be handled easier and lead to a simpler differential equation.\\
In the simplest case, $\Ha\setminus f(\Ha)$ is a bounded set. Then, $f$ has a meromorphic continuation to a neighborhood of $\infty$ with Laurent expansion $f(z)=a_{-1}z+a_0+\frac{a_1}{z}+...$ \\
In this case, $f$ is said to have \emph{hydrodynamic normalization} if $a_{-1}=1$ and $a_0=0.$  In this case, $f(\infty)=\infty$ and $f$ is an automorphism of $\Ha$ if and only if $f$ is the identity, as every automorphism of $\Ha$ fixing $\infty$ is of the form $z\mapsto az+b$, $a>0,b\in\R.$ \\
More generally, a holomorphic function $f:\Ha\to\Ha$ is said to have  \textit{hydrodynamic normalization} if there exists $c\geq 0$ such that $f$ has the expansion $f(z)=z-\frac{c}{z}+\Landauo(\frac1{z})$ in an angular sense, i.e. the following limit exists and is finite:
$$\angle\lim_{z\to\infty}z\left(f(z)-z\right)=c\in\R.$$
Let $\mathcal{H}_\infty$ denote the set of all these functions and let $l(f):=c.$ Furthermore we denote by $\mathcal{H}_u$ the subset of all $f\in\mathcal{H}_\infty$ that are univalent.\\
The class $\mathcal{H}_u$ can be seen as an analog of the class $S^{\geq}$ and $l(f)$ is playing the role that $f'(0)$ is playing for the class $S^{\geq}.$\\
 
It is known that every $f\in \mathcal{H}_\infty$ has the following useful, so called Nevanlinna integral representation (see Lemma 1 in \cite{MR1201130}):
\begin{equation}\label{Nevanlinna} f(z)=z+\int_\R \frac{1}{x-z}\, d\mu(x), \end{equation}
where $\mu$ is a finite, nonnegative Borel measure on $\R.$ 
From this representation, it follows that $l(f)=\mu(\R)\geq 0$. Moreover, the following two properties hold:

\begin{itemize}\label{Penny}
 \item Semigroup property: If $f,g\in \mathcal{H}_\infty,$ then also $f\circ g\in\mathcal{H}_\infty$ and $l(f\circ g)=l(f)+g(f)$ (Theorem 1 in \cite{MR1201130}).
\item ``Schwarz Lemma'': If $f\in \mathcal{H}_\infty,$ then $\Im(f(z))\geq \Im z$ for all $z\in\Ha.$ Here, equality holds for one point if and only if $f(z)\equiv z.$
\end{itemize}

We denote by $\mathcal{H}_b$ the set of all $f\in\mathcal{H}_u$ such that $\Ha\setminus f(\Ha)$ is bounded. A function $f\in \mathcal{H}_u$ belongs to $\mathcal{H}_b$ if and only if the Borel measure $\mu$ in the Nevanlinna representation of $f$ has compact support.\\

In the rest of this chapter we will only deal with the class $\mathcal{H}_b.$ We will look at the so-called chordal differential equation for this class and use $l(f)$ - the so-called half-plane capacity -  as a time parameter for this equation.

\section{Hulls and half-plane capacity}

A bounded subset $A\subset\Ha$ with the property $A=\overline{A}\cap\Ha$ such that $\Ha\setminus A$ is simply connected is called a \textit{(compact) hull}. By $g_A$ we denote the unique conformal map $g_A:\Ha\setminus A\to \Ha$ with \textit{hydrodynamic normalization}, which means $g_A^{-1}\in \mathcal{H}_b,$ i.e. $$g_A(z)=z+\frac{b}{z}+ \LandauO(|z|^{-2}),\quad  b>0,\quad \text{for}\; z\to\infty.$$

The quantity $\hcap(A):=b$ is called \textit{half-plane capacity} of $A.$ In some sense, $\hcap(A)$ is the size of $A$ seen from $\infty.$ This can be made precise by a probabilistic formula that involves Brownian motions hitting the hull $A$. It will be used in the proof of Lemma \ref{ThomasFundamentallemma}. The half-plane capacity $\hcap(A)$ is comparable to the simpler geometric quantity ``$\operatorname{hsiz}(A)$'', introduced in \cite{MR2576752}.

\begin{theorem}\label{hsiz}(Theorem 1 in \cite{MR2576752})
Let $A$ be a compact hull and $$\operatorname{hsiz}(A)=\operatorname{area}\left(\bigcup_{x+iy\in A}\mathcal{B}(x+iy,y)\right).$$ Then $$\frac1{66}\operatorname{hsiz}(A)<\hcap(A)\leq \frac7{2\pi}\operatorname{hsiz}(A).$$
\end{theorem}

We summarize four basic properties of $hcap$ in the following lemma.

\begin{lemma}\label{hcap1}
 Let $A, A_1,A_2$ be hulls.
\begin{itemize}
 \item[a)] If $A_1\cup A_2$ and $A_1\cap A_2$ are hulls, then $$\hcap(A_1)+\hcap(A_2)\geq \hcap(A_1\cup A_2)+\hcap(A_1\cap A_2).$$
\item[b)] If  $A_1\subset A_2,$ then $\hcap(A_2)=\hcap(A_1)+\hcap(g_{A_1}(A_2\setminus A_1))\geq\hcap(A_1).$
\item[c)] If  $A_1 \cup A_2$ is a hull and $A_1 \cap A_2=\emptyset$, then
$\hcap(g_{A_1}(A_2)) \leq \hcap(A_2).$
\item[d)] If $\lambda>0$, then $\hcap(\lambda\cdot A)=\lambda^2 \cdot \hcap(A)$ and $\hcap(A\pm\lambda)=\hcap(A)$.
\end{itemize}
\end{lemma}
\begin{proof} For a) and b) see \cite[p.~71]{Lawler:2005}. 
Now let $A_1 \cup A_2$ be a hull such that $A_1 \cap A_2=\emptyset$.
Then b) implies
$\hcap(A_1)+\hcap(g_{A_1}(A_2))=\hcap(A_1\cup A_2)$, while
a) shows $\hcap(A_1\cup A_2) \le \hcap(A_1)+\hcap(A_2)$. This proves c).\\
Finally,  note that $g_{\lambda A}(z)=\lambda g_A(z/\lambda) $ and $g_{A\pm\lambda}(z)=g_A(z\mp\lambda)\pm\lambda$. This proves d).\\
\end{proof}

Note that by Schwarz reflection, $g_A$ has an analytic continuation across 
$\R \backslash \overline{A}$ with $g_A(\R \backslash \overline{A}) \subseteq \R$
and $g_A^{-1}$ is analytic at $g_A(x)$ for every $x \in \R \backslash
\overline{A}$. If we denote by $\mu$ the measure in the Nevanlinna representation of $g_A^{-1}$, then the Stieltjes inversion formula (see \cite[Thm.~5.4]{MR1307384})
shows that $g_A(x) \not\in \supp(\mu_A)$ for every $x \in \R \backslash \overline{A}$.

The following lemma shows that $g_A$ is expanding outside the closed convex hull of
$\overline{A} \cap \R$ and nonexpanding in between points of $\R\setminus \overline{A}$.

\begin{lemma} \label{lem:hcapneu}
Let $A$ be a hull.
\begin{itemize}
\item[a)] If $\overline{A} \cap \R$ is contained in the closed interval
  $[a,b]$, then  $g_A(\alpha) \le \alpha$ for every $\alpha \in \R$ with
  $\alpha<a$ and  $g_A(\beta) \ge \beta$ for every $b \in \R$ with $\beta >b$.
\item[b)] If the open interval $(a,b)$ is contained in $\R \backslash
  \overline{A}$, then 
$|g_A(\beta)-g_A(\alpha)| \le |\beta-\alpha|$ for all $\alpha, \beta \in (a,b)$.
\end{itemize}
\end{lemma}

\begin{proof}
a) Let $\alpha<a$. If we denote by $\mu_A$ the measure in the  Nevanlinna representation  
(\ref{Nevanlinna}) for $f:=g_A^{-1}$, then we obtain for $z=g_A(\alpha)$:
$$ g_A(\alpha)=f(g_A(\alpha))-\int \limits_{\R}
\frac{1}{x-g_A(\alpha)} \, \mu_A(dx).$$
Since the interval $(-\infty,g_A(\alpha)]$ has no point in common with
$\supp(\mu_A)$, we actually integrate over a set for which the integrand is
nonnegative, so $g_A(\alpha) \le \alpha$. The proof of $g_A(\beta) \ge \beta$
for every $\beta>b$ is similar.

\smallskip

b) Let $\alpha,\beta \in (a,b)$ and assume $\alpha \le \beta$, so
$g_A(\alpha) \le g_A(\beta)$. If we subtract (\ref{Nevanlinna}) for
$z=g_A(\alpha)$ from (\ref{Nevanlinna}) for $z=g_A(\beta)$, then a short
computation leads to
$$ g_A(\beta)-g_A(\alpha)=\beta-\alpha+\int \limits_{\R}
\frac{g_A(\alpha)-g_A(\beta)}{(x-g_A(\alpha)) (x-g_A(\beta))} \, \mu_A(dx) \,
.$$
Since the closed interval $[g_A(\alpha),g_A(\beta)]$ is disjoint from 
$\supp(\mu_A)$, we integrate over a set for which the integrand is
nonpositive, so $0 \le g_A(\beta)-g_A(\alpha) \le \beta-\alpha$.
\end{proof}

\section{The chordal Loewner differential equation for \texorpdfstring{$\mathcal{H}_b$}{Hb}}

Next we would like to have an ordinary Loewner equation for evolution families with the property that every element is contained in the class $\mathcal{H}_b.$ \\
For this, assume that $\varphi_{s,t}$ is an evolution family in $\D$ which does not contain automorphisms of $\D$ whenever $s<t.$ Recall that all non-identical elements of an evolution family $\varphi_{s,t}$ in $\D$ have the same Denjoy-Wolff point $1$ if and only if the Herglotz vector field $G(z,t)$ has the form $G(z,t)=(z-1)^2 p(z,t), \Re p(z,t)\geq 0$ for all $z\in\D$ and almost all $t\geq 0,$ see p. \pageref{loe}.\\
 By conjugating the corresponding semigroup of an infinitesimal generator $G$ of the form $G(z)=(z-1)^2p(z)$ with the Cayley transform $C(z)=\frac{z-i}{z+i}$,  we obtain an infinitesimal generator $\tilde{G}$ for $\Ha$ with $\tilde{G}(z)=2ip(C(z)),$ $z\in\Ha$, i.e. we get exactly all functions that map the upper half-plane onto the upper half-plane.\\

The next Theorem presents a Loewner equation of ``type III'' (see p. \pageref{Thomas}), i.e. a reversed Loewner ODE, which generates functions whose inverses belong to $\mathcal{H}_b.$  

\begin{theorem}[Theorem 4.6 in \cite{Lawler:2005}]\label{Gregory}
 Suppose $\{\mu_t\}_{t\geq0}$ is a family of Borel probability measures on $\R$ such that $t\mapsto \mu_t$ is continuous in the weak topology, and for each $t$,  there is an $M_t<\infty$ such that $\supp \mu_s\subset [-M_t,M_t],$ $s\leq t.$ For each $z\in \Ha,$ let $g_t(z)$ denote the solution of the initial value problem \begin{equation}\label{chordal}
  \dot{g}_t(z)=\int_\R \frac{2}{g_t(z)-u}\, \mu_t(du), \quad g_0(z)=z.                                                                                                          \end{equation}
Let $T_z$ be the supremum of all $t$ such that the solution exists up to time $t$ and $g_t(z)\in \Ha.$ Let $H_t:=\{z\in \Ha \with T_z>t\},$ then $g_t$ is the unique conformal mapping of $H_t$ onto $\Ha$ with hydrodynamic normalization and $g_t(z)=z+\frac{2t}{z}+\LandauO(|z|^{-2})$ for $z\to\infty.$  
\end{theorem}
It is easy to see that the domains $H_t$ are strictly decreasing, i.e. $H_s\subsetneq H_t$ for all $t<s$ and that $H_t=\Ha\setminus K_t$ for a (compact) hull $K_t.$ Thus, equation (\ref{chordal}) generates a family $\{K_t\}_{t\geq0}$ of strictly increasing hulls with $\hcap(K_t)=2t$ for all $t\geq 0.$ 
\begin{remark}
In \cite{MR2107849}, section 5, the author treats a more general equation and proves a  one-to-one correspondence between certain evolution families in the class $\mathcal{H}_u$  and measurable families of probability measures on $\R.$
\end{remark}

Let us have a look at some examples for equation (\ref{chordal}).
 \begin{itemize}
\item A very important and simple case is $\mu_t=\delta_{U(t)},$ i.e. $\mu_t$ is the point measure with mass 1 in $U(t)$, with $U(t)$ being a continuous, real-valued function, called \textit{driving function}. We call the corresponding equation the \emph{chordal one-slit equation}. For sufficiently regular driving functions, the hulls will describe a slit, i.e. a simple curve growing from $U(0)$ into $\Ha$, see Section \ref{Sip}.
\item Let $\mu_t$  be a convex combination of $n$ point measures, i.e. $\mu_t=\lambda_1(t)\delta_{U_1(t)}+...+\lambda_n(t)\delta_{U_n(t)}$, where $\lambda_1(t),...,\lambda_n(t)$ map into the interval $[0,1]$ and $\lambda_1(t)+...+\lambda_n(t)=1$ for all $t$.
In this case, (\ref{chordal}) will be called the \emph{chordal multiple-slit equation}. We will discuss several properties of this equation in the following sections.
\item If every $\mu_t$ has a density with respect to the Lebesgue measure on $\R$, then equation  (\ref{chordal}) will generate some bubble- or tube-shaped hulls, see \cite{sola} for examples. 
\end{itemize}

Let $\{K_t\}_{t\geq 0}$ be a family of hulls generated by equation (\ref{chordal}) and the family $\{\mu_t\}_{t\geq 0}$ of probability measures.
There are four simple operations on $\{K_t\}$ which can be translated into changes of the family $\{\mu_t\}$ by calculating the changes of the half-plane capacity (Lemma \ref{hcap1}) and the corresponding conformal mappings, see also \cite{LindMR:2010}, Section 2.1.\\
 Let $d>0$ and $s>0$, then we have the following properties:

\begin{itemize}
\item \textit{Truncation: } For $t\geq s,$ the hulls $g_s(K_{s+t}\setminus K_s)$ correspond to the measures $\mu_{t+s}.$
 \item \textit{Reflection in $i\R$:} The reflected hulls $-\overline{K_t}$ correspond to the measures $A \mapsto \mu_t(-A)$.
\item \textit{Scaling:} The scaled hulls $dK_t$ correspond to $A\mapsto \mu_{t/d^2}(d^{-1}\cdot A)$.
\item \textit{Translation:} The translated hulls $K_t+d$ correspond to $A\mapsto \mu_t(A-d)$. 
\end{itemize}

\section{The backward equation}

A very useful tool to study the behavior of hulls generated by equation (\ref{chordal}) is its corresponding, so called \textit{backward equation}. For a given time $T,$ we can generate the inverse function $g_T^{-1}$ by simply reversing the flow of \ref{chordal}, i.e. we consider the initial value problem  
\begin{equation}\label{2slitsb} \dot{f}_t(z)=\int_\R \frac{-2 }{f_t(z)-u}\;\mu_{T-t}(du),\quad f_0(z)=z\in\Ha.  \end{equation}
For every $z\in \Ha$, $f_t(z)$ is defined for all $t\in[0,T]$ and we have $$f_T(z)=g_T^{-1}(z).$$ (However, $f_t(z)\not=g_t^{-1}(z)$ for $t\in(0,T)$ in general.) \\
Note that equation (\ref{2slitsb}) is of ``type I'', i.e. it is just an ordinary Loewner equation for evolution families.

We will need the following simple estimation for the imaginary part of hulls generated by equation (\ref{chordal}).
\begin{lemma}\label{imestim}
Let $\{K_t\}_{t\in[0,T]}$ be the family of hulls generated by equation (\ref{chordal}) with $0\leq t \leq T$. Then $$\max_{z\in K_T}\Im(z)\leq 2\sqrt{T}.$$
\end{lemma}
\begin{proof}
We consider the backward equation for $f_T=g_T^{-1}.$ For any $x_0\in\R$ and $y_0\in (0,\infty)$ write $f_t(x_0+iy_0)=x_t+iy_t.$ (\ref{2slitsb}) gives 
 \begin{eqnarray*} \dot{y}_t=\int_\R \frac{2 y_t }{(x_t-u)^2+y_t^2}\;\mu_{T-t}(du)
\leq  \int_\R \frac{2 y_t}{y_t^2}\; \mu_{T-t}(du)=\frac{2}{y_t}.\end{eqnarray*}
Thus $y_t\leq \sqrt{4t+y_0^2}.$ Letting $t\uparrow T$ and $y_0\downarrow 0$ gives $y_T\leq 2\sqrt{T}.$
\end{proof}

Another simple application of equation (\ref{2slitsb}) implies the following property for the class $\mathcal{H}_u$.
\begin{theorem} \label{thm:halfplane1}
Let $z_0 \in \Ha$. Then
$$ \left\{f(z_0) \with f \in \mathcal{H}_u\right\}=\left\{z
\in \C \, : \, \Im z>\Im z_0\right\} \cup \{z_0\}\, .$$
\end{theorem}

\begin{remark}  \label{rem:halfplane1}
By using the ``Schwarz lemma'' for the class $\mathcal{H}_\infty,$ see p. \pageref{Penny},
it is immediate that
$$ \{f(z_0) \with f \in \mathcal{H}_{u}\} \subseteq \{f(z_0) \with f \in \mathcal{H}_\infty\} \subseteq \{z
\in \C \, : \, \Im z>\Im z_0\} \cup \{z_0\}\, .$$
Hence, Theorem \ref{thm:halfplane1} tells us that the set  of values
$f(z_0)$ for all \textit{univalent} functions $f \in \mathcal{H}_{\infty}$
is the same as the set of values $f(z_0)$ for all $f \in
\mathcal{H}_{\infty}$. This is a significant difference to the unit disc
case, where the set of values $f(z_0)$ for all univalent
functions $f:\D\to\D$ with $f(0)=0,$ $f'(0)\geq0$ is strictly smaller than the set of values $f(z_0)$ for all holomorphic functions $f:\D\to\D$ with $f(0)=0,$ $f'(0)\geq0$, see \cite{RothSchl}, Section 1.
\end{remark}

\begin{proof}[Proof of Theorem \ref{thm:halfplane1}]
 Fix $z_0 \in \Ha$ and let $z\in \Ha$ with $\Im(z) > \Im(z_0).$

Now we look at equation (\ref{2slitsb}) with initial value $z_0$ for the case $\mu_t=\delta_{U(t)}$, where $U$ is a real valued function, i.e. 
\begin{equation} \label{eq:chor}
\begin{array}{rcl}
\dot{w}(t)&=& \, \displaystyle \frac{-2}{w(t)-U(t)} \, , \qquad t \ge 0 \,
, \\[3mm]
w(0) &=& z_0 \in \Ha.
\end{array}
\end{equation}

We need to find a driving function $U$ such that the
solution $w(t)$ of (\ref{eq:chor}) passes through $z$. We separate into real
and imaginary parts and write $w(t)=x(t)+iy(t)$ and $z_0=x_0+iy_0$. Now, we claim
that $U$ can be chosen such that $w(t)$ connects $z_0$ and $z$ by a straight line segment, i.e. 
\begin{equation*} x(t)=c\cdot y(t)+x_0 - c \cdot y_0\, , \end{equation*}
where
$$c=\frac{\Re z-\Re z_0}{\Im z - \Im z_0}\, .$$

In order to prove this, we separate
 equation (\ref{eq:chor}) into real and imaginary parts and obtain
\begin{equation*}
\begin{array}{rcl}
\dot{x}(t)&=& \, \displaystyle \frac{2(U(t)-x(t))}{(U(t)-x(t))^2+y(t)^2} \, ,
\quad  \dot{y}(t)=\, \displaystyle \frac{2y(t)}{(U(t)-x(t))^2+y(t)^2}  \, , 
\end{array}
\end{equation*}
with initial conditions $x(0)=x_0$ and $y(0)=y_0$.
We now assume that $x(t)$ and $y(t)$ are related by
$$U(t)-x(t)=c\cdot y(t) \, .$$ 
Then we get the following initial value problem:
\begin{equation*}
\begin{array}{rcl}
\dot{x}(t)= \, \displaystyle \frac{2c}{(1+c^2)y(t)} \, , \quad  \dot{y}(t)=\, \displaystyle \frac{2}{(1+c^2)y(t)}  \, , \quad
x(0)=x_0, \, y(0)=y_0 \, ,
\end{array}
\end{equation*}
which can be solved directly:
\begin{equation*}
\begin{array}{rcl}
y(t)=\, \displaystyle \sqrt{\frac{4}{1+c^2}t+y_0^2} \qquad \text{and} \qquad
x(t)= \, \displaystyle c y(t)+x_0-cy_0 \, , \quad t \ge 0\,  .
\end{array}
\end{equation*}
Hence if we now \textit{define}
$$ U(t):=c y(t)+x(t)= 2c \sqrt{\frac{4}{1+c^2}t+y_0^2}+ x_0-cy_0 \, ,$$
then by construction the solution $w(t)=x(t)+i y(t)$ of (\ref{eq:chor})
satisfies $x(t)=cy(t)+x_0-c y_0.$
 In particular, the trajectory $t \mapsto w(t)$ is the halfline starting at
 $z_0$ through the point $z$, so $z \in \{f(z_0)\with f\in \mathcal{H}_u\}$.
This completes the proof of Theorem \ref{thm:halfplane1}.
\end{proof}

\section{The one-slit equation }\label{Pascha}

In the case $\mu_t=\delta_{U(t)},$ with $U(t)$ being a real-valued, continuous function, equation (\ref{chordal}) becomes the one-slit equation
\begin{equation}\label{ivp}
   \dot{g}_t(z)=\frac{2}{g_t(z)-U(t)}, \quad g_0(z)=z.
\end{equation}

If $\gamma:[0,T]\to \overline{\Ha}$ is a simple curve, i.e. a one-to-one continuous function, with $\gamma(0)\in\R$ and $\gamma((0,1])\subset \Ha,$ then we call the compact hull $\Gamma:=\gamma(0,1]$ a \textit{slit}. The important connection between slits and equation (\ref{ivp}) is given by the following Theorem.

\begin{theorem}[Kufarev, Sobolev, Spory{\v{s}}eva]\label{slitex}
 For any slit $\Gamma$ with $\hcap(\Gamma)=2T$ there exists a unique, continuous driving function $U:[0,T]\to\R$ such that the solution $g_t$ of (\ref{ivp}) satisfies $g_T=g_{\Gamma}.$
\end{theorem}

To the best of our knowledge the first proof of Theorem \ref{slitex} is due
to Kufarev, Sobolev, Spory{\v{s}}eva in \cite{MR0257336}. A recent, English reference can be found in \cite{Lawler:2005}, p. 92f. We also refer to the survey paper \cite{GM13} for a complete and rigorous proof of Theorem \ref{slitex} using classical complex analysis.
\\

So every slit can be represented as the solution to (\ref{ivp}) with a unique driving function $U$. If $\gamma$ is the parameterization of $\Gamma$ by half-plane capacity, i.e. $\hcap(\gamma(0,t])=2t$ for all $t\in[0,T],$ then we have $$g_t=g_{\gamma[0,t]} \quad \text{for all} \quad t\in[0,T] \quad \text{and} \quad  U(t)=\lim_{z\to\gamma(t)}g_t(z).$$

\begin{remark}\label{Lola} Theorem \ref{slitex} has been generalized to a characterization of all hulls that correspond to the one-slit equation with continuous driving functions.
For an arbitrary hull $K$ let $\operatorname{rad}(K)=\inf\{r\geq0\with \exists x\in\R, K\subset\mathcal{B}(x,r)\}.$ The following statements can be found in \cite{MR1879850}, Theorem 2.6 or \cite{Lawler:2005}, p. 96. Pommerenke considered the corresponding radial case already in 1966, see \cite{MR0206245}, Theorem 1.\\
 Let $\{K_t\}$ be a family of hulls generated by equation (\ref{ivp}). Let $K_{t,s}=g_t(K_{t+s}\setminus K_t).$ Then $K_{t,s}$ has the \textbf{local growth property}:
$$ \lim_{h\downarrow 0}\sup_{\begin{subarray}{c}
t+s\in [0,T]\\
0<s\leq h
                             \end{subarray}
} \operatorname{rad}(K_{t,s}) = 0.$$

Conversely, let $\{K_t\}$ be a family of increasing hulls with $\hcap(K_t)=2t$ satisfying the local growth property. Then there exists a real valued, continuous functions $U$ such that equation (\ref{ivp}) generates the hulls $K_t.$ For each $t$, the value $U(t)$ is the unique $x\in\R$ with $\{x\}= \bigcap_{s>0}\overline{K_{t,s}}$.

\end{remark}

Let us have a look at a simple example.

\begin{example}\label{Maxime}
Consider a slit $L$ that is a segment of a straight line with initial point $0$. Denote by $\phi$ the angle between $\R$ and $L.$ The scaling property immediately implies that the driving function $U$ of $L$ satisfies $U(t)=c\sqrt{t}$ for some $c\in\R.$ \\
Conversely, let $U(t)=c\sqrt{t}$ for an arbitrary $c\in\R.$ In this case, the one-slit equation (\ref{ivp}) can be solved explicitly and one obtains for the generated hull $K_t$ at time $t:$ $K_t=\gamma[0,t]$ with $\gamma(t)=2\sqrt{t}\left(\frac{\pi}{\phi}-1\right)^{\frac1{2}-\frac{\phi}{\pi}}e^{i\phi},$ i.e. $K_t$ is a line segment with angle $\phi.$ \\
The connection between $c$ and $\phi$ is given by
$$c(\phi)=\frac{2(\pi-2\phi)}{\sqrt{\phi(\pi-\phi)}},\qquad \phi(c)=\frac{\pi}{2}\left(1-\frac{c}{\sqrt{c^2+16}}\right),$$
see Example 4.12 in \cite{Lawler:2005}. \hfill $\bigstar$
\end{example}

We will now look at two problems coming along with Theorem \ref{slitex}. In the next two sections, we will be concerned with these questions and we will derive some partial answers.\\

Firstly, we would like to generalize Theorem \ref{slitex} to several slits. Suppose we are given $n$ disjoint slits. We should expect to be able to assign $n$ unique continuous driving functions to those slits. However, for $n\geq 2$, we also have to determine $n-1$ coefficient functions that appear in the multiple-slit equation; so:
 \begin{equation}\label{Charlie}
   \begin{aligned}
 &\textit{How can Theorem \ref{slitex} be generalized to the multiple-slit equation, so that}\\ &\textit{one can assign $n$ unique continuous driving functions to $n$ disjoint slits?}
   \end{aligned}
\end{equation}
The second problem is to find a converse of Theorem  \ref{slitex}. The latter  induces a mapping from the set $$\{\Gamma\with \Gamma \; \text{is a slit with}\,\hcap(\Gamma)=2T\}$$ into the set $$\{f:[0,T]\to\R \with f \;\text{is continuous} \}.$$ This mapping is one-to-one but \emph{not surjective}, which was already known to Kufarev, see \cite{MR0023907}. The following example shows that equation (\ref{ivp}) does not necessarily produce slits even when the driving function is continuous.
\begin{example}\label{Felix}
 Consider the driving function $U(t) = c\sqrt{1-t}$ with $c\geq 4$ and $t\in [0,1]$: \\ 
The generated hull $K_t$ is a simple curve for $t\in[0,1),$ but at $t=1$ this curve hits the real axis at an angle $\varphi$ which can be calculated directly (see \cite{LindMR:2010}, chapter 3): $$ \varphi = \pi-\frac{2\pi\sqrt{c^2-16}}{\sqrt{c^2-16}+c}. $$ Its complement with respect to $\Ha$ has two connected components. If we denote by $B$ the bounded component, then $K_1=\overline{B}\cap \Ha$, see Figure \ref{Fi31}, and consequently $K_1$ is not a slit.
\hfill $\bigstar$
\end{example}
\begin{figure}[h] 
    \centering
   \includegraphics[width=135mm]{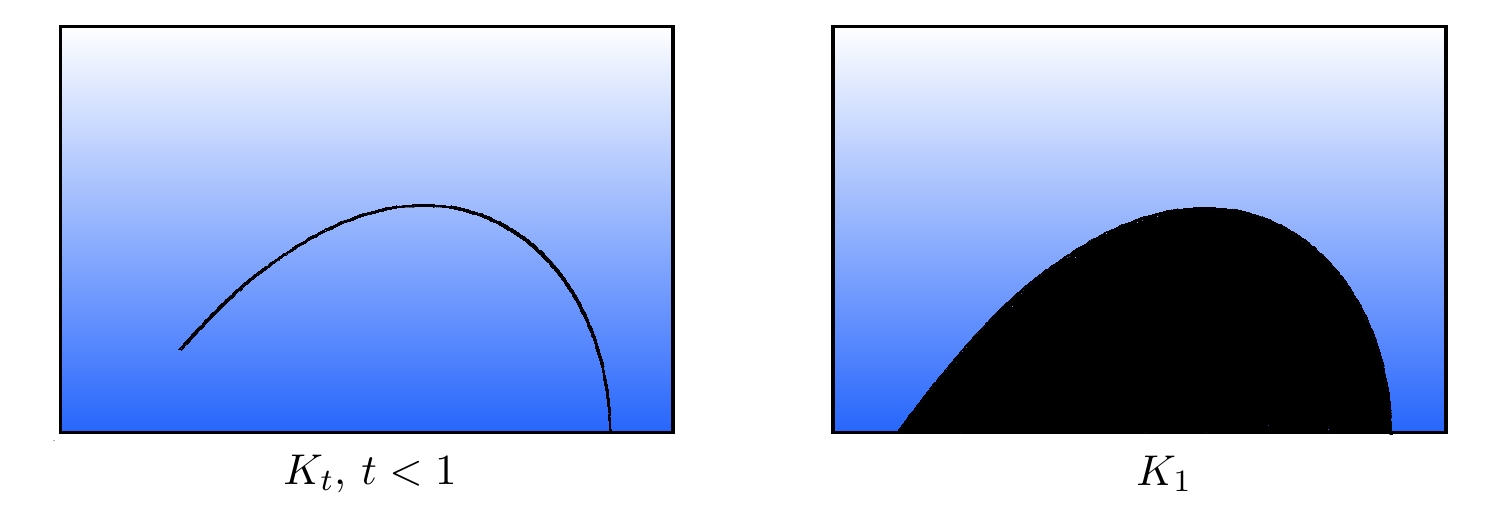}
\caption{Example \ref{Felix} with $c=5.$}\label{Fi31}
 \end{figure}

Now, the following question nearly suggests itself:
\begin{equation}
 \textit{Which driving functions let equation (\ref{ivp}) generate slit mappings?}
\end{equation} 
We will look at necessary and sufficient conditions in section \ref{Sip}, also in the setting of the multiple--slit equation.

\section{Constant coefficients in the multiple-slit equation}\label{Chiara}

Let us consider $n$ slits $\Gamma_1,\ldots,\Gamma_n$ with pairwise disjoint closure and $\hcap(\Gamma_1\cup...\cup\Gamma_n)=2T$. It is not hard to show that there exist continuous functions $\gamma_1,...,\gamma_n:[0,T]\to \C$ with $\gamma_j(0,T]=\Gamma_j$ such that
 \begin{itemize}
\item the functions $t\mapsto \hcap(\gamma_j(0,t])$ are nondecreasing for all $j=1,...,n,$
\item $\hcap(\gamma_1(0,t]\cup...\cup\gamma_n(0,t])=2t$ for every $t\in[0,T].$                    \end{itemize}
We call $(\gamma_1,...,\gamma_n)$ a \emph{Loewner parameterization} for the hull $\Gamma_1\cup...\cup\Gamma_n$.\\
In this case, the conformal mappings $g_t:=g_{\gamma_1(0,t]\cup...\cup\gamma_n(0,t]}$ satisfy the following chordal multiple-slit equation 

\begin{equation}\label{more}\dot{g}_t(z)=\sum_{k=1}^n\frac{2\lambda_k(t)}{g_t(z)-U_{k}(t)} \quad \text{for a.e.}\;\; t\in[0,T],\quad g_0(z)=z,\end{equation} 

where $\sum_{k=1}^n\lambda_k(t) = 1$ for a.e. $t\in[0,T]$ and for a.e. $t\in[0,T],$ the coefficient function $\lambda_k$ is given by 
$$ \lambda_k(t) = \left. \frac1{2}\frac{d}{ds}\right|_{s=0}\hcap(\gamma_k(0,t+s]\cup \bigcup_{j\not=k} \gamma_j(0,t]), $$ see Theorem 2.1 in \cite{RothSchl2}.\\

Again, $U_k(t)$ is the image of $\gamma_k(t)$ under the map $g_t$ and the functions $t\mapsto U_j(t)$ are called \emph{driving functions}. Roughly speaking, the coefficient function $t\mapsto \lambda_k(t)$ corresponds to the speed of the growth of $\Gamma_k$.

\begin{remark}[The multiple--slit equation in Mathematical Physics]
{\rm We note that the multiple--slit equation (\ref{more}) has
recently been used by
 physicists for the study of certain two--dimensional growth phenomena. 
For instance, in \cite{MR1945279}  the authors analyze ``Laplacian path models'', i.e.~Laplacian growth models for multi--slits. By mapping the upper half--plane conformally onto a half-strip one obtains a Loewner equation for the growth of slits in a half--strip, which can be used to describe Laplacian growth in the ``channel geometry'', see \cite{MR2495460} and \cite{PhysRevE.84.051602}. 
Furthermore, equation (\ref{more}) can be used to model so--called
multiple Schramm--Loewner evolutions, see \cite{MR2310306} and
\cite{MR2004294}, \cite{MR2187598}, \cite{MR2358649}, \cite{Graham2007}.} 
\end{remark}

If we compare the case $n\geq 2$ in the multiple-slit equation to the case $n=1,$ then there are two main differences:\\
Firstly, when $n=1$ there exists exactly one Loewner parameterization, but when $n\geq 2$, there exist many as we can choose different ``speeds'' for the growth of the slits. Consequently, there exist many choices for driving functions $U_1,...,U_n$ and nonnegative coefficient functions $\lambda_1,...\lambda_n$ with $\sum_{k=1}^n \lambda_k(t)=1$ for all $t$ such that equation (\ref{more}) generates the hull $\Gamma_1\cup...\cup\Gamma_n,$ $n\geq 2.$ It seems to be natural that the functions $\lambda_j(t), U_j(t)$ can be made unique by requiring constant growth speeds $\lambda_j(t)\equiv\lambda_j$.\\
Secondly, the differential equation (\ref{more}) holds only for almost every $t\in[0,T].$ For $n=1,$ however, we know by Theorem \ref{slitex} that it actually holds for \emph{every} $t\in[0,T].$ \\

In this section we show that there exists one unique Loewner parameterization such that equation (\ref{more}) holds for all $t$ with constant coefficients.

\begin{theorem}\label{Charlie2}
 Let $\Gamma_1,...,\Gamma_n$ be slits with disjoint closure and let $\hcap(\Gamma_1 \cup...\cup \Gamma_n)=2T.$\\
Then there exist unique $\lambda_1,...,\lambda_n\in(0,1)$ with $\sum_{k=1}^n\lambda_k=1$ and unique continuous functions $U_1,...,U_n:[0,T]\to\R,$ such that the solution of the chordal Loewner equation $$ \dot{g}_t(z) = \sum_{k=1}^n \frac{2\lambda_k}{g_t(z)-U_k(t)}, \qquad g_0(z)=z, $$
satisfies $\displaystyle g_T=g_{\Gamma_1\cup ...\cup\Gamma_n}.$ 
\end{theorem}

Hence Theorem \ref{Charlie2} provides a canonical way of describing multiple slits, as Theorem \ref{slitex} does for a single slit. 

\begin{remark}
D. Prokhorov has proven the existence and uniqueness of constant 
	coefficients for the radial Loewner equation under the assumption that all slits are \emph{piecewise analytic}, see Theorem 1 and 2 in \cite{Prokhorov:1993}. This theorem forms the basis for Prokhorov's study of extremal problems for univalent functions in \cite{Prokhorov:1993} by using control-theoretic methods. We will use a completely different method which shows that one can drop any assumption on the regularity of the slits in order to generate them with constant coefficients. This method can also be applied to prove the analog of Theorem \ref{Charlie2} for the radial case, see Section \ref{Christoph}.
\end{remark}

\begin{remark}
 It would be interesting to find an interpretation of the coefficients $\lambda_1,\ldots,\lambda_n$ in terms of geometric or potential theoretic properties of the slits $\Gamma_1,\ldots,\Gamma_n.$  Somehow they describe the size of the slits relative to each other. 
\end{remark}

As in \cite{Prokhorov:1993}, we only give the proofs for the case $n=2,$ because the general cases follow by induction. In Section \ref{auxiliary} we prove some auxiliary lemmas.  The proof of Theorem \ref{Charlie2} is divided into two parts: proof of existence (Section \ref{existence}) and proof of uniqueness (Section \ref{uniqueness}). The latter uses a ``dynamic interpretation'' of the coefficients $\lambda_j$, which is proved in Section \ref{dynamic}.

The proof of existence needs the fact that a certain set of driving functions is precompact, which is proved in the next section.

\subsection{\label{auxiliary}A precompactness statement}

Let $\Theta_1, \Theta_2$ be slits with disjoint closures and 
$\hcap(\Theta_1)=\hcap(\Theta_2)=2$. 
In this section we will prove the following technical result, which states that the set
of driving functions for all Loewner parameterizations of a subhull $A$ of
$\Theta_1 \cup\Theta_2$, $\hcap(A)=2T,$ is  a precompact subset of the Banach space $C([0,T],\R)$ (equipped with the maximum norm).

\begin{theorem}\label{aux} Let $A$ be a subhull of 
$\Theta_1 \cup \Theta_2$ with $\hcap(A)=2T$. For any Loewner parameterization $\gamma=(\gamma_1,\gamma_2)$ of $A$, we let $g^\gamma_t:=g_{\gamma_1(0,t] \cup \gamma_2(0,t]}$  and 
$U^\gamma_1(t)=g^\gamma_{t}(\gamma_1(t))$ and $U^\gamma_2(t)=g^\gamma_{t}(\gamma_2(t))$
be the driving functions for $(\gamma_1,\gamma_2)$. Then the sets
$$ \{U^\gamma_1:[0,T]\to\R \with \gamma \; \text{is a Loewner parameterization of $A$}\}, $$
$$ \{U^\gamma_2:[0,T]\to\R \with \gamma \; \text{is a Loewner parameterization of $A$}\} $$
are precompact subsets of the Banach space $C([0,T],\R)$. \end{theorem}

The proof requires a number of technical estimates for the half--plane capacities of
two--slits and their subhulls.

\medskip

We start with a refinement of Lemma \ref{hcap1} c) for the case when the hulls are slits.

\begin{lemma} \label{hcap2new}
Let $\Theta_1$ and $\Theta_2$ be slits with disjoint closures. Then
there is a constant $c>0$ such that
$$ c \le 
\frac{\hcap(B \cup \Theta_2)-\hcap
    (A \cup \Theta_2)}{\hcap(B)-\hcap(A)} \, $$
for all subslits $A \subsetneq B \subseteq \Theta_1$.
\end{lemma}

We note that a local version of Lemma \ref{hcap2new} in the sense of
$$ \lim \limits_{\hcap(B) \searrow \hcap(A)} \frac{\hcap(B\cup \Theta_2)-\hcap
    (A\cup \Theta_2)}{\hcap(B)-\hcap(A)}>0 \, \quad  \text{ for fixed
  } A \, ,$$
has been proved by Lawler, Schramm and Werner \cite[Lemma 2.8]{MR1879850}.
Our proof shows how to obtain the global statement of Lemma \ref{hcap2new} from
this local version.

\begin{proof} 
Using $g_{A \cup \Theta_2}=g_{\Delta} \circ g_{\Theta_2}$ for $\Delta:=g_{\Theta_2}(A)$
and Lemma
\ref{hcap1}, it is easy to see that
\begin{equation} \label{eq:hcap2new0}
 \hcap(B \cup \Theta_2)-\hcap
    (A \cup \Theta_2)=\hcap(g_{\Theta_2}(B))-\hcap(g_{\Theta_2}(A)) \, .
\end{equation}
Let $T:=\hcap(\Theta_1)/2>0$ and let $\theta : [0,T] \to \C$ be the
  parameterization of $\Theta_1$ by its half--plane capacity.
 For fixed $s \in
  [0,T]$ let $L_s:=\theta (0,s]$ and $\gamma(s):=\hcap(
  g_{\Theta_2}(L_s))$. Hence, in view of (\ref{eq:hcap2new0}) and 
since $\hcap(L_s)=2s$, all we  need to show is that
  there is a constant $c>0$ such that
\begin{equation} \label{eq:hcap2new1}
c \le \frac{\gamma(s)-\gamma(\tau)}{s-\tau} \quad \text{ for all } 0 \le  \tau<s
\le T \, .
\end{equation}
In order to prove (\ref{eq:hcap2new1}), we proceed in several steps.

\smallskip

(i) \, Fix $\tau \in [0,T)$. For $s \in [\tau,T]$ let
$$\begin{array}{rcl}
& K_s:=g_{L_{\tau}}(L_s \backslash L_{\tau})  , \qquad &
K_s^*:=g_{g_{L_{\tau}}(\Theta_2)}(K_s) \\[2mm]
& b(s):=\hcap(K_t) \, , \, & b^*(s):=\hcap( K_s^*)\, . 
\end{array}$$
Then, by \cite[Lemma 2.8]{MR1879850},
the right derivatives $\dot{b}_+(\tau)$ and $\dot{b}^*_+(\tau)$ of $b$ and $b^*$ at $\tau$
exist and
\begin{equation} \label{eq:hcap2new2}
\dot{b}^*_+(\tau)=\left[ g'_{g_{L_{\tau}}(\Theta_2)} \left(
    g_{L_{\tau}} (\theta(\tau)) \right) \right]^2 \dot{b}_+(\tau) \, .
\end{equation}
Here, $$g_{L_{\tau}}(\theta(\tau)):=\lim \limits_{z \to \theta(\tau)}
g_{L_{\tau}}(z)\, , $$ 
where the limit is taken over $z \in \Ha \backslash \Gamma_{\tau}$.
Now note that by Lemma \ref{hcap1} b),
$$ b(s)=\hcap( g_{L_{\tau}}(L_s \backslash
L_{\tau})=\hcap(L_{s})-\hcap(L_{\tau})=s-\tau \, $$
and, in a similar way, $b^*(s)= 
\gamma(s)-\gamma(\tau)$.

Therefore, (\ref{eq:hcap2new2}) shows that the right derivative
$\dot{\gamma}_+(\tau)$ of the function $\gamma :[0,T] \to \R$ exists for every
$\tau \in [0,T)$ and 
\begin{equation} \label{eq:hcap2new3}
 \dot{\gamma}_+(\tau)= \left[ g'_{g_{L_{\tau}}(\Theta_2)} \left(
    g_{L_{\tau}} (\theta(\tau)) \right) \right]^2 \, .
\end{equation}
(ii) \, Next  $\tau \mapsto U(\tau):=g_{\Gamma_{\tau}}(\theta(\tau))$ is
continuous on $[0,T)$ (see \cite[Lemma 4.2]{Lawler:2005}).
Furthermore, since $\Theta_1 \cap \Theta_2=\emptyset$, i.e., 
$U(\tau)\not \in g_{L_{\tau}}(\Theta_2)$, the function
$g_{g_{L_{\tau}}(\Theta_2)}$ has an analytic continuation to a
neighborhood of $U(\tau)$ and $g'_{g_{L_{\tau}(\Theta_2)}}(U(\tau))\not=0$,
see \cite[p.~69]{Lawler:2005}.
Since $\tau \mapsto g_{g_{L_{\tau}}(\Theta_2)}$ is continuous in the topology
of locally uniform convergence, we hence conclude from  (\ref{eq:hcap2new3})  that
$\dot{\gamma}_+$ is a continuous nonvanishing function  on the interval $[0,T)$. 

\smallskip

(iii) \, From (ii) we see that $\gamma : [0,T] \to \R$ is continuous, has
a right derivative $\dot{\gamma}_+(\tau)$ for every point $\tau \in [0,T)$ and
$\dot{\gamma}_+ :[0,T) \to \R$ is continuous. By Lemma 4.3 in
\cite{Lawler:2005}, $\gamma :(0,T] \to \R$ is differentiable with
$\dot{\gamma}(\tau)=\dot{\gamma}_+(\tau)$ for every $\tau \in (0,T)$. Hence
the mean value theorem shows that (\ref{eq:hcap2new1}) holds with
$$c:=\min \limits_{\tau \in [0,T]} \left[ g'_{g_{L_{\tau}}(\Theta_2)} \left(
    g_{L_{\tau}} (\theta(\tau)) \right) \right]^2 >0 \, .$$
\end{proof}

We shall need the following slight extension of Lemma \ref{hcap2new}.

\begin{lemma}\label{hcap2}
 Let  $\Theta_1$ and $\Theta_2$ be slits with disjoint closures. Then there
 exists a constant $c>0$ such that 
$$ c \le \frac{\hcap(B_1 \cup B_2)-\hcap (A_1 \cup
  A_2)} {\hcap(B_j)-\hcap(A_j)}  \, , \qquad j=1,2 \, , $$
for all subslits $A_1 \subsetneq B_1$ of $\Theta_1$ and $A_2 \subsetneq B_2$
of $\Theta_2$.
\end{lemma}

\begin{proof}
It suffices to prove the lemma for $j=1$. If we apply Lemma \ref{hcap1} b) for
the  hulls $B_1 \cup A_2 \subseteq B_1 \cup B_2$ and then Lemma \ref{hcap1}
a) for the hulls $B_1 \cup A_2$ and $A_1 \cup \Theta_2$, we obtain
\begin{eqnarray*}
 \hcap(B_1 \cup B_2)-\hcap (A_1 \cup A_2) &\ge&  \hcap (B_1 \cup A_2)-\hcap(A_1
\cup A_2) \\ &\ge & \hcap (B_1 \cup \Theta_2)-\hcap (A_1 \cup \Theta_2) \, .
\end{eqnarray*}
Therefore, the estimate of Lemma \ref{hcap2new} completes the proof of Lemma \ref{hcap2}.
\end{proof}

Let $\Gamma$ be the union of two slits $\Theta_1$ and
$\Theta_2$ with disjoint closures. Then $g_{\Gamma}$  extends continuously onto each of the sides of
$\overline{\Theta_1}$ and of $\overline{\Theta_2}$
and maps them into $\R$. For every $c\in \overline{\Gamma}$ which is neither the tip of
$\Theta_1$ nor of $\Theta_2$, we write $g^+_\Gamma(c)$ for the image w.r.t the right side and
$g^-_\Gamma(c)$ w.r.t. the left side, so that $g^-_\Gamma(c)<g^+_\Gamma(c).$

\begin{lemma} \label{lem:neu1}
Let $\Theta_1$  and $\Theta_2$ be two slits which start at $p_1 \in
\R$ resp.~$p_2 \in \R$ such that $p_1<p_2$ and $\overline{\Theta}_1 \cap \overline{\Theta}_2=\emptyset$. Then
\begin{itemize}
\item[a)]
$ g^-_{\Theta_1\cup \Theta_2}(p_1) \le g^-_{B_1 \cup B_2}(p_1) \le g^+_{B_1
  \cup B_2}(p_2) \le  g^+_{\Theta_1\cup \Theta_2}(p_2)$, and
\item[b)] $g^-_{B_1 \cup B_2}(p_2)-g^+_{B_1 \cup B_2}(p_1) \ge g^-_{\Theta_1
    \cup \Theta_2}(p_2)-g^+_{\Theta_1\cup\Theta_2}(p_1)$
\end{itemize}
for all subslits $B_1 \subseteq \Theta_1$ and $B_2 \subseteq \Theta_2$.
\end{lemma}

\begin{proof}
a) Let $A_1:=g_{B_1\cup B_2}(\Theta_1\backslash B_1)$ and $A_2:=g_{B_1\cup
  B_2}(\Theta_2\backslash B_2)$.
Then $A_1$ and $A_2$ are two disjoint slits  which start say at $a \in \R$
resp.~$b \in \R$. Let $A:=A_1 \cup A_2$. Then $A$ is a hull such that
$\overline{A} \cap \R \subseteq [a,b]$. Now $\alpha:=g^-_{B_1 \cup B_2}(p_1) \le a$, so Lemma \ref{lem:hcapneu} a) implies $g_A(\alpha) \le \alpha$. Since
$g_{\Theta_1 \cup \Theta_2}=g_A \circ g_{B_1\cup B_2}$, this shows that
$g^-_{\Theta_1\cup \Theta_2}(p_1) \le g^-_{B_1 \cup B_2}(p_1)$ and proves the
left--hand inequality. The proof of the right--hand inequality is similar.

\smallskip

b) Let $A:=g_{B_1 \cup B_2}(\Theta_1\backslash B_1 \cup \Theta_2 \backslash
B_2)$. Then $A$ is a hull with $\overline{A}\cap \R=\{a,b\}$ such that $a < b$
and $a < g^+_{B_1\cup B_2}(p_1) \le g^-_{B_1 \cup B_2}(p_2) < b$. Hence,
$$
g^-_{\Theta_1 \cup \Theta_2}(p_2)-g^+_{\Theta_1\cup\Theta_2}(p_1) =
g_A(g^-_{B_1\cup B_2}(p_2))-g_A(g^+_{B_1\cup B_2}(p_1)) \le
g^-_{B_1\cup B_2}(p_2)-g^+_{B_1\cup B_2}(p_1)  $$
by Lemma \ref{lem:hcapneu} b).
\end{proof}

\begin{lemma} \label{lem:II}
Let $\Theta_1$ and $\Theta_2$ be slits with disjoint closures. Then there is a
constant $L>0$ such that
$$ |g_{B_1 \cup A_2}(b_1)-g_{B_1 \cup B_2}(b_1)| \le L \cdot
|\hcap(B_2)-\hcap(A_2)|$$
for all subslits $A_2,B_2$ of $\Theta_2$ and every subslit $B_1$ of $\Theta_1$
with tip $b_1 \in B_1$.
\end{lemma}

\begin{proof}
We assume $A_2 \subseteq B_2$. Then $A:=g_{B_1 \cup A_2}(B_2 \backslash A_2)$
is a hull, so  Lemma \ref{hcap1} b) shows $\hcap(A)=\hcap(B_1 \cup
B_2)-\hcap(B_1 \cup A_2)$. On the other hand, 
 Lemma \ref{hcap1} a) applied to the two hulls $B_2$ and $B_1 \cup A_2$ 
gives
$\hcap(B_1 \cup B_2)-\hcap(B_1 \cup A_2) \le
\hcap(B_2)-\hcap(A_2)$. Therefore, 
\begin{equation} \label{eq:II1}
 \hcap(A) \le \hcap(B_2)-\hcap(A_2)
\, .
\end{equation}
Note that $g_{B_1 \cup B_2}=g_A \circ g_{B_1 \cup A_2}$, so the Nevanlinna
representation formula (\ref{Nevanlinna}) for $h_A:=g_A^{-1}$ and $w=g_{B_1
  \cup B_2}(b_1)$ shows that
\begin{equation} \label{eq:II2}
g_{B_1 \cup A_2}(b_1)-g_{B_1 \cup B_2}(b_1)=h_A(g_{B_1 \cup B_2}(b_1))-g_{B_1
  \cup B_2}(b_1)=\int \limits_{\R} \frac{d\mu_A(t)}{t-g_{b_1 \cup B_2}(b_1)}
\, .
\end{equation}
Let $\Theta_1$ start at $p_1 \in \R$ and $\Theta_2$ start at $p_2 \in \R$ with
$p_1<p_2$. Then the interval $(-\infty,g^-_{B_1 \cup B_2}(p_2)]$ is disjoint
from the support $\supp(\mu_A)$ of the measure $\mu_A$, so for every $t \in
\supp(\mu_A)$, we have
$$ t-g_{B_1\cup B_2}(b_1) \ge g^-_{B_1 \cup B_2}(p_2)-g^+_{B_1 \cup B_2}(p_1)
\ge  g^-_{\Theta_1 \cup \Theta_2}(p_2)-g^+_{\Theta_1 \cup \Theta_2}(p_1)=:L^{-1}>0$$
by Lemma \ref{lem:neu1} b).
Hence (\ref{eq:II2}) leads to
$$ 0< g_{B_1 \cup A_2}(b_1)-g_{B_1 \cup B_2}(b_1) \le L \int \limits_{\R} d\mu_A(t)
=L \hcap(A).
$$
 In view of (\ref{eq:II1}) the proof of Lemma \ref{lem:II}
is complete.
\end{proof}

\begin{lemma} \label{lem:III}
Let $\Theta_1$ and $\Theta_2$ be slits with disjoint closures. Then there
exists a monotonically increasing function $\omega : [0,\hcap(\Theta_1)] \to
[0,\infty)$ with $\lim \limits_{\delta \searrow 0} \omega(\delta)=\omega(0)=0$ such that
\begin{equation} \label{eq:III1}
 |g_{A_1 \cup A_2}(a_1)-g_{B_1 \cup A_2}(b_1)| \le \omega \left(
  |\hcap(A_1)-\hcap(B_1)| \right)
\end{equation}
for all subslits $A_1$ and $B_1$ of $\Theta_1$ with tips $a_1 \in A_1$ and
$b_1 \in B_1$ and every subslit $A_2 \subseteq \Theta_2$.
\end{lemma}

\begin{proof}
We first define $\omega(\delta)$ for $\delta \in (0,\hcap(\Theta_1)]$ by
$$ \omega(\delta):=\sup \{ g^+_{B_1}(a_1)-g^-_{B_1}(a_1) \} \, . $$
Here the supremum is taken over all subslits $A_1 \subseteq B_1$ of
$\Theta_1$ such that $\hcap(B_1)-\hcap(A_1) \le \delta$ and $a_1$ is the tip
of $A_1$. Clearly,  $\omega : (0,\hcap(\Theta_1)] \to (0,\infty)$ is
monotonically increasing and we need to prove (i) the estimate (\ref{eq:III1})
and (ii) $\lim_{\delta\searrow 0} \omega(\delta)=0$.

\smallskip

(i) Assume $A_1 \subseteq B_1$. Consider the slit $A:=g_{A_1 \cup A_2}(B_1
\backslash A_1)$, which starts at $g_{A_1 \cup A_2}(a_1)$. Then $g_{B_1 \cup B_2}=g_A \circ g_{A_1 \cup A_2}$, so Lemma
\ref{lem:hcapneu} a) implies $g^{-}_{B_1 \cup A_2}(a_1) =g^-_A(g_{A_1 \cup A_2}(a_1))
\le g_{A_1 \cup
  A_2}(a_1) \le g^+_A(g_{A_1 \cup A_2}(a_1))=g^+_{B_1 \cup A_2}(a_1)$. Since
we clearly also have $g^-_{B_1 \cup A_2}(a_1) \le
g_{B_1 \cup A_2}(b_1) \le g^+_{B_1 \cup A_2}(a_1)$, we deduce
$$ |g_{A_1 \cup A_2}(a_1)-g_{B_1 \cup A_2}(b_1)| \le g^+_{B_1 \cup
  A_2}(a_1)-g^-_{B_1 \cup A_2}(a_1) \, .$$
Since $g_{B_1 \cup A_2}=g_{g_{B_1}(A_2)} \circ g_{B_1}$, Lemma \ref{lem:hcapneu}
b) shows that
$$ g^+_{B_1 \cup A_2}(a_1)-g^-_{B_1 \cup A_2}(a_1)=g_{g_{B_1}(A_2)} \left(
  g^+_{B_1}(a_1) \right)-g_{g_{B_1}(A_2)} \left( g^-_{B_1}(a_1) \right) \le
g^+_{B_1}(a_1)-g^-_{B_1}(a_1) \, ,$$
so we get $ |g_{A_1 \cup A_2}(a_1)-g_{B_1 \cup A_2}(b_1)| \le
\omega(\hcap(B_1)-\hcap(A_1))$, i.e.~the estimate (\ref{eq:III1}) holds.

\smallskip

(ii) Let $c_1:=\hcap(\Theta_1)/2$, denote by $\theta_1 : [0,c_1] \to \C$
the parameterization of $\Theta_1$ by its half--plane capacity and let $U :
[0,c_1] \to \R$ be the driving function for the slit $\Theta_1$ according to
Theorem \ref{slitex}.
Let $A_1 \subseteq B_1$ be subslits of $\Theta_1$ and let $a_1$ be the tip of
$\Theta_1$. Then there are $t,s \in [0,c_1]$ with $t \le s$ such that
$\theta_1(t)=a_1$, $\theta_1(0,s]=B_1$ and $s-t=\hcap(B_1)/2-\hcap(A_1)/2$.
Consider the slit $P:=g_{A_1}(B_1 \backslash A_1)$, so $\overline{P} \cap
\R=\{U(t)\}$ and $g^+_{B_1}(a_1)-g^-_{B_1}(a_1)$ is the Euclidean length of
the interval $g_P(P)$.
 By Remark 3.30 in \cite{Lawler:2005} there is an absolute
constant $M>0$ such that
\begin{equation} \label{eq:III2}
g^+_{B_1}(a_1)-g^-_{B_1}(a_1) \le M \cdot \diam(P) \, ,
\end{equation}
where $\diam(P):=\sup \{|p-q| \, : \, p,q \in P\}$. Define 
$\operatorname{rad}(P):=\sup \{|z-U(t)| \, : \, z \in P\}$. Then, by Lemma 4.13 in
\cite{Lawler:2005},
\begin{eqnarray*}
 \operatorname{rad}(P) &\le & 4\max \left\{ \sqrt{s-t}, \sup \limits_{t \le \tau \le s}
  |U(\tau)-U(t)| \right\} \\[3mm]
& \le & 4 \max \left\{ \sqrt{s-t}, \sup \limits_{|\tau-\sigma|\le s-t} 
  |U(\tau)-U(\sigma)| \right\}
 \, .
\end{eqnarray*}
Hence, if we define 
$$ \varrho(\delta):=\max \left\{ \sqrt{\delta/2}, \sup \limits_{|\tau-\sigma|\le \delta/2} 
  |U(\tau)-U(\sigma)| \right\} $$
for $\delta \in [0,\hcap(\Theta_1)]$ then $\operatorname{rad}(P) \le 4 \varrho(\hcap(B_1)-\hcap(A_1))$.
Using the obvious estimate $\diam(P) \le 2 \operatorname{rad}(P)$, we obtain
from (\ref{eq:III2}) that  $g^+_{B_1}(a_1)-g^-_{B_1}(a_1) \le 8 M
\varrho(\hcap(B_1)-\hcap(A_1))$. Recalling the definition of $\omega(\delta)$,
this shows that $$\omega(\delta) \le 8 M \varrho(\delta) \quad \text{for all}\quad \delta \in
(0,\hcap(\Theta_1)].$$ Since the continuous driving function $U : [0,c_1] \to
\R$ is uniformly continuous on $[0,c_1]$, we see that $\varrho(\delta) \to 0$
as $\delta \searrow 0$, so $\lim \limits_{\delta \searrow 0} \omega(\delta)=0$.
\end{proof}

\begin{lemma} \label{lem:IV}
Let $\Theta_1$ and $\Theta_2$ be slits with disjoint closures. Then there
exist constants $c,L>0$ and a monotonically increasing function $\omega : [0,\hcap(\Theta_1)] \to
[0,\infty)$ with $\lim \limits_{\delta \searrow 0} \omega(\delta)=\omega(0)=0$  such that
\begin{eqnarray*}
|g_{A_1 \cup A_2}(a_1)-g_{B_1 \cup B_2}(b_1)| & \le &
\omega \left( \frac{1}{c} \, |\hcap(A_1\cup A_2)-\hcap(B_1 \cup B_2)|
\right)  \\ & & \hspace*{1cm} + \frac{L}{c} \, |\hcap(A_1 \cup A_2)-\hcap(B_1\cup B_2)| 
\end{eqnarray*}
for all subslits $A_1$ and $B_1$ of $\Theta_1$ with
tips $a_1 \in A_1$ and $b_1 \in B_1$ and all subslits
 $A_2, B_2$ of $\Theta_2$.
\end{lemma}

\begin{proof}
We can assume $A_1 \subsetneq B_1$ and $A_2 \subsetneq B_2$.
Then 
\begin{eqnarray*}
|g_{A_1 \cup A_2}(a_1)-g_{B_1 \cup B_2}(b_1)| & \le & |g_{A_1 \cup
  A_2}(a_1)-g_{B_1 \cup A_2}(b_1)|+|g_{B_1 \cup A_2}(b_1)-g_{B_1 \cup
    B_2}(b_1)| \\
&\le &  \omega (\hcap(B_1)-\hcap(A_1))+L \left( \hcap(B_2)-\hcap(A_1) \right)
\, .
\end{eqnarray*}
by Lemma \ref{lem:III} and Lemma \ref{lem:II}. Now the estimate of Lemma
\ref{hcap2}
completes the proof of Lemma \ref{lem:IV}.
\end{proof}

\begin{proof}[Proof of Theorem \ref{aux}] Let $A$ be a subhull of $\Theta_1\cup \Theta_2$ and let $(\gamma_1, \gamma_2)$ be a Loewner parameterization for $A.$ Let  $g_t:=g_{\gamma_1(0,t] \cup \gamma_2(0,t]}.$
 Then
$$ g^-_{t}(\gamma_1(0)) \le U^\gamma_j(t) \le
g^+_{t}(\gamma_2(0)) \, ,$$
so Lemma \ref{lem:neu1} a) implies
$$ g^-_{\Theta_1 \cup \Theta_2}(\theta_1(0)) \le U^\gamma_j(t) \le
g^+_{\Theta_1 \cup \Theta_2}(\theta_2(0)) \, , \qquad j=1,2\, .$$
This gives a uniform bound for $U^\gamma_1(t)$ and $U^\gamma_2(t)$.
Since $\hcap(\gamma_1(0,t] \cup \gamma_2(0,t])=2t$,
Lemma \ref{lem:IV} implies
$$
 |U^\gamma_1(t)-U^\gamma_1(s)|=|g_{t}(\gamma_1(t))-g_{s}(\gamma_1(s))| \le 
\omega \left( \frac{2|t-s|}{c} \right) +\frac{2L}{c} |t-s|$$
for all $t,s \in [0,T]$. This shows that the driving functions $U^\gamma_1$ 
for all Loewner parameterizations $(\gamma_1, \gamma_2)$ are
uniformly equicontinuous on $[0,T]$. By switching the roles of $U^\gamma_1$ and
$U^\gamma_2$, the same result holds for the driving functions $U^\gamma_2$.\\
The statement of Theorem \ref{aux} now follows directly from the Arzel\`a--Ascoli theorem.
\end{proof}

\subsection{\label{existence}Proof of Theorem \ref{Charlie2}, Part I (Existence)}
Now we describe the setting we need for the proof of Theorem \ref{Charlie2} when $n=2.$\\

Let $\Gamma_1$ and $\Gamma_2$ be slits in $\Ha$ with $\overline{\Gamma_1}\cap\overline{\Gamma_2}=\emptyset$.\\
We assume that $\hcap(\Gamma_1\cup \Gamma_2)=2$ and we let $c_1=\hcap(\Gamma_1)/2$ and $c_2=\hcap(\Gamma_2)/2.$\\[0.2cm]
 Furthermore we let $\Theta_1,$ $\Theta_2$ be slits with $\Gamma_j\subset \Theta_j$,  $\overline{\Theta_1}\cap\overline{\Theta_2}=\emptyset,$ and $\hcap(\Theta_1)=\hcap(\Theta_2)=2.$ We denote by $\theta_j(t)$ the corresponding parameterization of $\Theta_j$ by its half-plane capacity, so that we have $\Gamma_j=\theta_j[0,c_j],$ $j=1,2.$ \\
We want to find $\lambda \in [0,1]$ and two driving functions $U_1,U_2,$ such that the two slits $\Gamma_1,\Gamma_2$ are produced by the Loewner equation
\begin{equation}\label{2slits} \dot{g}_t(z)=\frac{2\lambda}{g_t(z)-U_{1}(t)}+\frac{2(1-\lambda)}{g_t(z)-U_{2}(t)},\quad g_0(z)=z. \end{equation}

To begin with, we define a family of Loewner equations, indexed by $(n,\lambda)\in\N_0\times [0,1].$\\ Let  $\alpha_{n,\lambda}:[0,1]\to \{0,1\},$ be a function with \begin{equation}\label{Pete} \alpha_{n,\lambda}(t)=\begin{cases}
1 \quad \text{when} \quad t\in (\frac{k}{2^n}, \frac{k+\lambda}{2^n}),\; k\in\{0,...,2^n-1\},\\
0 \quad \text{when} \quad t\in (\frac{k+\lambda}{2^n}, \frac{k+1}{2^n}),\; k\in\{0,...,2^n-1\}.                                                                  \end{cases}
\end{equation}
 Now consider the (one--slit) differential equation
\begin{equation}\label{e} \dot{g}_{t,n}(z)=\frac{2\alpha_{n,\lambda}(t)}{g_{t,n}(z)-U_{1,n}(t)}+\frac{2(1-\alpha_{n,\lambda}(t))}{g_{t,n}(z)-U_{2,n}(t)}, \quad g_{0,n}(z)=z. \end{equation}
We want to let $\Theta_1$ (resp. $\Theta_2$) grow whenever $\alpha_{n,\lambda}(t)=1$ (resp. $\alpha_{n,\lambda}(t)=0$) and the one-slit case gives us a uniquely determined driving function $U_{1,n}(t)$ (resp. $U_{2,n}(t)$) there (which also depends on $\lambda$). We extend $U_{1,n}(t)$ to all $t\in [0,1]$ by requiring that $U_{1,n}(t)$ is the image of the tip of the  part of $\Theta_1$ produced at time $t$ for all $t\in[0,1].$ Likewise, we extend $U_{2,n}$, so that we get two  continuous functions defined on the whole interval $[0,1].$

For $t\in[0,1]$ we have produced a hull $H_{n,\lambda, t}$ having the form $$H_{n,\lambda, t}=\theta_1[0,x_{n,\lambda, t}]\cup \theta_2[0,y_{n,\lambda, t}],$$ where $x_{n,\lambda, t}\in[0,1]$ depends continuously on $\lambda.$

For every $n\in\N_0$, we have $x_{n,0, 1}=0$ and $x_{n,1, 1}=1$ and consequently there is a  $\lambda_{n}$ with $x_{n,\lambda_n, 1}=c_1$ according to the intermediate value theorem. The monotonicity of $\hcap,$ Lemma \ref{hcap1} b), implies $y_{n,\lambda_n, 1}=c_2$ so that $H_{n,\lambda_{n}, 1}=\Gamma_1\cup \Gamma_2.$\\[0.2cm]
Thus we have a sequence of coefficient functions $\alpha_{n,\lambda_n}$ and continuous driving functions $U_{1,n},U_{2,n},$ such that the solution of the one-slit equation (\ref{e}) always generates the two slits $\Gamma_1\cup\Gamma_2$.\\

According to Theorem \ref{aux}, both sequences $U_{1,n}$ and $U_{2,n}$ possess uniformly converging subsequences. Since $(\lambda_n)_n$ is a sequence of real numbers in the interval $[0,1],$ we can find a convergent subsequence $(\lambda_{n_k})_k$ with limit $\lambda:=\lim_{k\to\infty}\lambda_{n_k}$, such that two sequences $U_{1,n_{k}},$ $U_{2,n_k}$ converge uniformly to $U_1$ and $U_2$ respectively. The functions $U_1$ and $U_2$ are continuous, too. Let $g_t$ be the solution of \begin{equation} \dot{g}_t(z)=\frac{2\lambda}{g_t(z)-U_{1}(t)}+\frac{2(1-\lambda)}{g_t(z)-U_{2}(t)},\quad g_0(z)=z. \end{equation}
 It is easy to see that 
\begin{equation}\label{weak} \int_0^t  \alpha_{n_k}(s) f(s) \, ds \; \underset{k\to\infty}{\longrightarrow}\;  \lambda \cdot \int_0^t  f(s) \, ds  \quad \text{for any} \quad f\in C([0,T], \C) \;\text{and}\; t\in[0,1]. \end{equation}
By combining (\ref{weak}) with the uniform convergence of $U_{1,n_k},$ $U_{2,n_k}$, we conclude that for any  $t\in[0,1]$ we have 
$$ \int_0^t \frac{2\alpha_{n_k,\lambda_{n_k}}(\tau)}{z-U_{1,n_k}(\tau)}+\frac{2(1-\alpha_{n_k,\lambda_{n_k}}(\tau))}{z-U_{2,n_k}(\tau)} \; d\tau \; \underset{k\to\infty}{\longrightarrow}\;
\int_0^t \frac{2\lambda}{z-U_{1}(\tau)}+\frac{2(1-\lambda)}{z-U_{2}(\tau)} \; d\tau $$
locally uniformly in $\Ha$.

From this, it follows that $g_{t,n_k}\to g_t$ locally uniformly for every fixed $t\in [0,1]$ when  $k\to\infty$, see \cite{Roth}, Lemma I.37, or \cite{RothSchl2}, Theorem 2.4. In particular, $$g_1=\lim_{k\to\infty}g_{1,n_k}=\lim_{k\to\infty}g_{\Gamma_1\cup \Gamma_2}=g_{\Gamma_1\cup \Gamma_2}.$$ 
Finally, the last equation implies that $\lambda \not\in\{0,1\}$. This completes the proof of the existence statement of Theorem (\ref{Charlie2}). \hfill \proofsymbol

\subsection{The dynamic interpretation of the coefficients}\label{dynamic}

Let $\Gamma_1$ and $\Gamma_2$ be slits with disjoint closures and $\hcap(\Gamma_1\cup\Gamma_2)=2$. We have proved
in the last section that there exist a constant $\lambda \in (0,1)$ and driving functions
$U_1,U_2 \in C([0,1],\R)$ such that the solution $g_t$ to the chordal Loewner equation
\begin{equation}\label{eq:2slits}
  \dot{g}_t(z)=\frac{2\lambda}{g_t(z)-U_{1}(t)}+\frac{2(1-\lambda)}{g_t(z)-U_{2}(t)},\quad
  g_0(z)=z \, , \end{equation}
satisfies $g_1=g_{\Gamma_1 \cup \Gamma_2}$. Let $\gamma_1(t)$ and
$\gamma_2(t)$ be the tip of the part of $\Gamma_1$ and $\Gamma_2$ respectively
at time $t$, so $(\gamma_1,\gamma_2)$ is a Loewner parameterization of
$(\Gamma_1,\Gamma_2)$ with  constant coefficients $\lambda$ and $1-\lambda$.
 In this section, we will derive some properties of this  Loewner parameterization
 $(\gamma_1,\gamma_2)$.

In the following lemma we let $\mathcal{B}(z,r):=\{w\in\C\with|z-w|<r\},$
where $z\in\C, r>0$ and for $A\subseteq \C$ we define
$\displaystyle\diam(A):=\sup_{z,w\in A}|z-w|.$ 

\begin{lemma}\label{ThomasFundamentallemma} Let $x(t)=\hcap(\gamma_1(0,t])$ and
$y(t)=\hcap(\gamma_2(0,t])$. Then
 $$x(t)+y(t)-2t=\Landauo(t) \quad \text{for} \quad t\to0.$$
\end{lemma}
\begin{proof}
First, we note that $x(t)+y(t)-2t \geq 0$ for all $t$ because of Lemma
\ref{hcap1} a).

\medskip

We will use a formula which translates the half--plane capacity of an arbitrary hull $A$ into an expected value of a random variable derived from a Brownian motion hitting this hull. Let $B_s$ be a Brownian motion started in $z\in \Ha\setminus A.$ We write $\operatorname{\bf P}^{z}$ and $\operatorname{\bf E}^{z}$ for probabilities and expectations derived from $B_s.$ Let $\tau_A$ be the smallest time $s$ with $B_s\in \R\cup A.$ Then formula (3.6) of Proposition 3.41 in \cite{Lawler:2005} tells us
$$\hcap(A)=\lim_{y\to\infty}y\operatorname{\bf E}^{yi}[\operatorname{Im}(B_{\tau_A})].$$
Let $\varrho_t=\tau_{\gamma_1[0,t]}$ and $\sigma_t=\tau_{\gamma_2[0,t]}.$ Then we have (compare with the proof of Proposition 3.42 in \cite{Lawler:2005})
\begin{eqnarray*}x(t)+y(t)-2t=\lim_{y\to \infty}y \left(\operatorname{\bf E}^{yi}[\operatorname{Im}(B_{\varrho_t});\sigma_t<\varrho_t]+ \operatorname{\bf E}^{yi}[\operatorname{Im}(B_{\sigma_t});\sigma_t>\varrho_t]\right).
\end{eqnarray*}

We will estimate the two expected values. First, $\Im(B_{\varrho_t})\leq 2\sqrt{t}$ by Lemma \ref{imestim} and we get
$$\operatorname{\bf E}^{yi}[\operatorname{Im}(B_{\varrho_t});\sigma_t<\varrho_t]\leq 2\sqrt{t} \cdot \operatorname{\bf P}^{yi}\{B_{\varrho_t}\in \gamma_1[0,t];\sigma_t<\varrho_t\}. $$
Now for $t$ small enough there exists $R>0$ such that 
$$ \gamma_1[0,s]\subset \mathcal{B}(\Re(\gamma_1(s)), R), \qquad \gamma_2[0,s]\subset \mathcal{B}(\Re(\gamma_2(s)), R),  $$
$$ \gamma_1[0,t] \cap \mathcal{B}(\Re(\gamma_2(s)), R)=\emptyset,\qquad \gamma_2[0,t] \cap \mathcal{B}(\Re(\gamma_1(s)), R)=\emptyset, $$
for all $s\in[0,t].$

\medskip

 A Brownian motion satisfying $\sigma_t<\varrho_t$ will hit $\gamma_2[0,t]$, say at $\gamma_2(s)$ for some $s\in[0,t]$, and has to leave $\mathcal{B}(\Re(\gamma_2(s)),R)\cap \overline{\Ha}$ without hitting the real axis, see Figure \ref{Fi32}. Call the probability of this event $p_s$. Then we have 
$$ \operatorname{\bf P}^{yi}\{B_{\varrho_t}\in \gamma_1[0,t];\sigma_t<\varrho_t\} \leq  \operatorname{\bf P}^{yi}\{B_{\sigma_t}\in \gamma_2[0,t]\} \cdot \sup_{s\in[0,t]}p_s. $$
Lemma \ref{imestim} implies $\Im(B_{\sigma_t})\leq 2\sqrt{t}$ and Beurling's estimate (Theorem 3.76 in \cite{Lawler:2005}) says that there exists $c_1>0$ (depending on $R$ only) such that $$p_s\leq c_1 \cdot 2\sqrt{t}.$$
(Note that Theorem 3.76 in \cite{Lawler:2005} gives an estimate on the probability that a Brownian motion started in $\D$ will not have hit a fixed curve, say $[0,1]$, when leaving $\D$ the first time. The estimate we use can be simply recovered by mapping the half-circle $\D\cap \Ha$ conformally onto $\D\setminus [0,1]$ by $z\mapsto z^2.$)\\
 We get the same estimates for $\operatorname{\bf
   E}^{yi}[\operatorname{Im}(B_{\sigma_t});\sigma_t>\varrho_t]$ and putting
 all this together gives the following upper bound for $x(t)+y(t)-2t$:
\begin{eqnarray*} 
&\displaystyle x(t)+y(t)-2t \leq   \lim_{y\to \infty}y \left(2\sqrt{t}\cdot c_1\cdot 2\sqrt{t} \cdot \operatorname{\bf P}^{yi}\{B_{\sigma_t}\in \gamma_2[0,t]\}+ 2\sqrt{t}\cdot c_1\cdot 2\sqrt{t} \cdot \operatorname{\bf P}^{yi}\{B_{\varrho_t}\in \gamma_1[0,t]\}\right)\\
 & \displaystyle \hspace*{-3.2cm} =4c_1t\cdot\displaystyle\lim_{y\to \infty}y
 \left( \operatorname{\bf P}^{yi}\{B_{\sigma_t}\in
   \gamma_2[0,t]\}+\operatorname{\bf P}^{yi}\{B_{\varrho_t}\in
   \gamma_1[0,t]\}\right).\end{eqnarray*}
Here the  limit exists and (see \cite[p.~74]{Lawler:2005})
\begin{eqnarray*}
 \lim_{y\to \infty}y \operatorname{\bf P}^{yi}\{B_{\sigma_t}\in
 \gamma_2[0,t]\} &\leq & c_2  \diam(\gamma_2[0,t]) \, , \\
\lim_{y\to \infty}y \operatorname{\bf P}^{yi}\{B_{\varrho_t}\in
\gamma_1[0,t]\} &\leq & c_2 \diam(\gamma_1[0,t]) \end{eqnarray*}
with a universal constant $c_2>0$.
Finally $\diam(\gamma_j[0,t])\to 0$ for $t\to 0$ and $j=1,2;$ see, e.g., Lemma 4.13 in \cite{Lawler:2005}. Hence we have shown $x(t)+y(t)-2t=\Landauo(t).$ 
\end{proof}

\begin{figure}[H] \label{Brownian_motion}
    \centering
   \includegraphics[width=140mm]{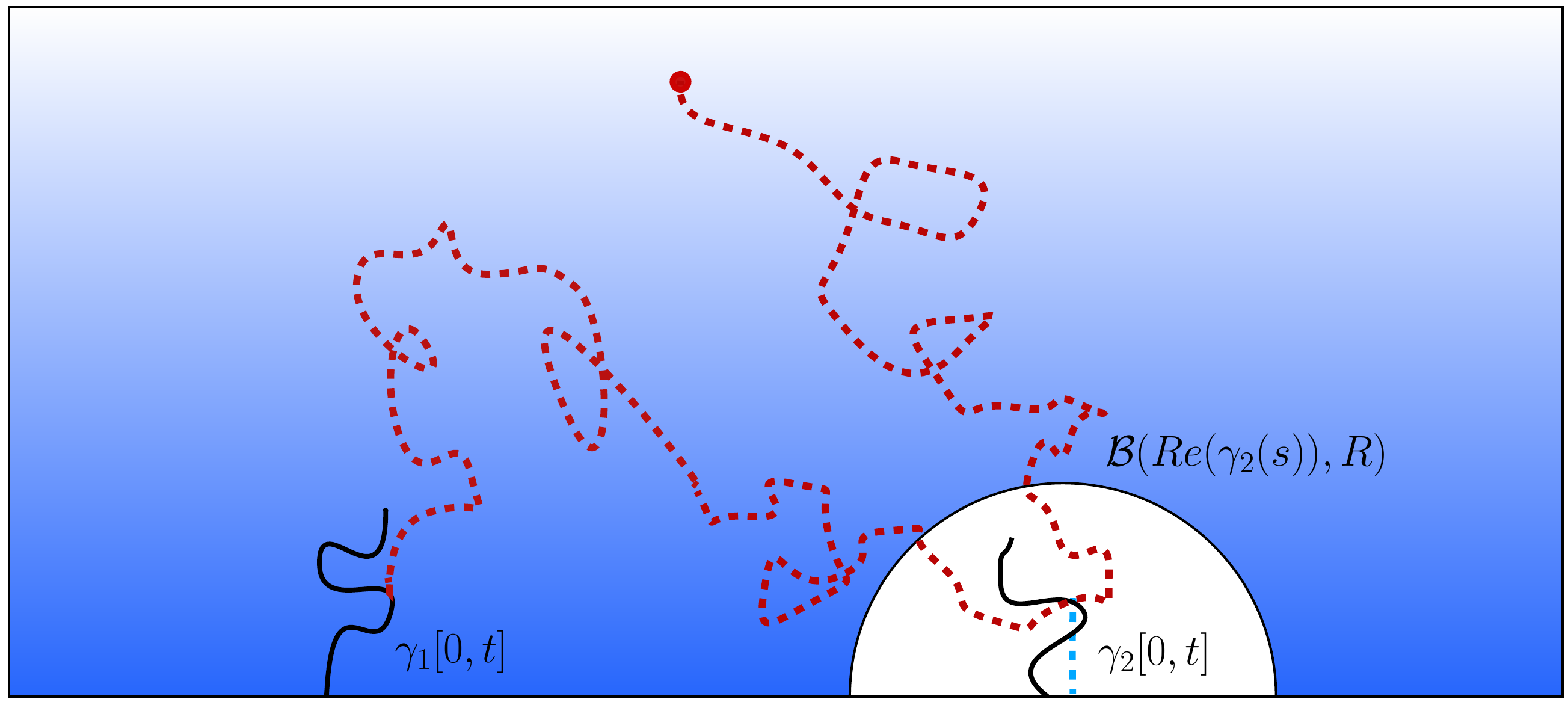}
\caption{A Brownian motion with $\sigma_t<\varrho_t$.}\label{Fi32}
 \end{figure}

{The following lemma gives a dynamical interpretation of the weights $\lambda$ and $1-\lambda$.}

\begin{lemma}\label{nointer} Let $x(t)=\hcap(\gamma_1(0,t])$ and
$y(t)=\hcap(\gamma_2(0,t])$.
Then $x(t)$ and $y(t)$ are differentiable in $t=0$ with $$\dot{x}(0)=2\lambda \quad \text{and} \quad\dot{y}(0)=2(1-\lambda).$$
\end{lemma}
\begin{proof} Let $U_1,U_2$ be the driving functions for the
Loewner parameterization  ($\gamma_1,\gamma_2)$. Without loss of generality we assume that $\Gamma_1$ is the left slit, i.e. $U_1(t)<U_2(t)$ for all $t\in[0,1].$

\medskip

(i) In a first step we prove
 $x(t)\geq 2\lambda t$ for all $t\in[0,\tau]$ with some $\tau>0$. Let $n\in\N$ and consider the Loewner equation
$$\dot{g}_{t,n}(z)=
\frac{2\alpha_{n,\lambda}(t)}{g_{t,n}(z)-U_{1}(t)}+\frac{2(1-\alpha_{n,\lambda}(t))}{g_{t,n}(z)-U_{2}(t)},
\quad g_{0,n}(z)=z,$$ where $\alpha_{n,\lambda}$ is defined as in
(\ref{Pete}). Let $K_{n,t}$, $t\in[0,1]$, be the corresponding family of
hulls. Similarly to Section \ref{existence}, we have $\Ha\setminus K_{n,t}\to \Ha\setminus (\gamma_1(0,t]\cup \gamma_2(0,t])$ for $n\to\infty$ in the sense of  kernel convergence.
Let $z_0\in (U_1(0), U_2(0))$ and denote by $z_n(t)$ the solution to

$$\dot{z}_n(t)= \frac{2\alpha_{n,\lambda}(t)}{z_n(t)-U_{1}(t)}+\frac{2(1-\alpha_{n,\lambda}(t))}{z_n(t)-U_{2}(t)}, \quad z_n(0)=z_0.$$ 
It may not exist until $t=1,$ but during its interval of existence we have $z_n(t)-U_{2}(t)<0<z_n(t)-U_{1}(t)$ and 
$$ \frac{2}{z_n(t)-U_2(t)}\leq \dot{z}_n(t) \leq \frac{2}{z_n(t)-U_1(t)}. $$
From this it follows that there exist $\tau,A,B>0$, independent of $n,$ such that $z_n(t)$ exists until $t=\tau$  and $$\max_{s\in[0,\tau]}U_1(s) < A < z_n(t) < B < \min_{s\in[0,\tau]}U_2(s).$$
Thus, for all $n\in \N$ and $t\in(0,\tau]$, we can write $K_{t,n}=C_{t,n}\cup D_{t,n}$, where $C_{t,n}$ and $D_{t,n}$ are disjoint subhulls of $K_{t,n}$ with $$\Ha\setminus C_{t,n} \to \Ha \setminus \gamma_1(0,t], \quad  \Ha\setminus D_{t,n} \to \Ha \setminus \gamma_2(0,t].$$
The cluster sets\footnote{If $A$ is a hull and $B\subset \overline{A}\cup\R$, then the cluster set of $B$ with respect to $g_A$ is the set $$\{w\in\overline{\Ha}\;|\;\text{There exists a sequence $(z_n)_n,$ $z_n\in \Ha\setminus A,$ with $z_n\to z_0\in \partial B$ and $g_A(z_n)\to w$} \}.$$} of $C_{t,n}$ and $D_{t,n}$ with respect to $g_{t,n}$ are sets
$I_1$ and $I_2$ respectively with $I_1\subset (-\infty, z_n(t))$ and
$I_2\subset (z_n(t),+\infty).$ Hence, $C_{t,n}$ is the hull that is growing if
and only if $\alpha_{n,\lambda}(t)=1.$

\medskip
 Let $x_n(t)=\hcap(C_{n,t})$. Then we get 
\begin{eqnarray*}\label{Alfons}x_n\left(\frac{k}{2^n}\right) = &&\sum_{j=1}^k \left(x_n\left(\frac{j}{2^n}\right)-x_n\left(\frac{j-1}{2^n}\right)\right) = \sum_{j=1}^k \left(x_n\left(\frac{j-1+\lambda}{2^n}\right)-x_n\left(\frac{j-1}{2^n}\right)\right) \\ \underset{\text{Lemma \ref{hcap1} c)}}{\geq} &&\sum_{j=1}^k 2\left(\frac{j-1+\lambda}{2^n}-\frac{j-1}{2^n}\right)  =\sum_{j=1}^k \frac{2\lambda}{2^n} = 2\lambda \cdot \frac{k}{2^n} \end{eqnarray*}
for all $n\in \N$ and $k\in\{1,...,2^n\}$ with $k/2^n\leq \tau$.\\
As $x_n(t)\to x(t)$ for every $t\in[0,\tau],$ we conclude that $x(t)\geq 2\lambda t$ for any $t$ of the form $t=k/2^n\leq \tau.$ The set of all those $t$ is dense in $[0,\tau]$ and as $x(t)$ is a continuous function, we deduce 
$x(t)\geq 2\lambda t$ for every $t\in [0,\tau]$.

\medskip

(ii) In a similar way as in step (i), now 
 utilizing the Loewner equation
$$\dot{h}_{t,n}(z)=
\frac{2(1-\alpha_{n,1-\lambda}(t))}{h_{t,n}(z)-U_{1}(t)}+\frac{2\alpha_{n,1-\lambda}(t)}{h_{t,n}(z)-U_{2}(t)}\,
,
\quad h_{0,n}(z)=z\, ,$$ 
we obtain $y(t)\geq 2(1-\lambda)t$ for all $t\ge 0$ small enough.

\medskip

(iii) Using the estimates in (i) and (ii), Lemma  \ref{ThomasFundamentallemma} gives
$$2\lambda t\leq x(t)\leq 2\lambda t+\Landauo(t) \qquad \text{for} \quad t\to0,$$
i.e., $\dot{x}(0)$ exists and $\dot{x}(0)=2\lambda$. In the same way we obtain
$\dot{y}(0)=2(1-\lambda)$.
\end{proof}

\begin{lemma}\label{LipDiff}  
 Let $x(t)=\hcap(\gamma_1(0,t])$.
Then the function $x:[0,1]\to [0,\infty)$ is continuously differentiable with
$$\dot{x}(0)= 2\lambda \quad\text{and} \quad\dot{x}(t)> 2\lambda \quad \text{for all}
\;\; t \in (0,1]\, .$$
In addition,$$\dot{x}(t)=\frac{2 \lambda}{C(x(t),t)}\, , $$
with a continuous function $C : \{(x_0,t) \, : \, 0 \le x_0 \le t, 0 \le t \le
1\} \to (0,1]$, which is continuously differentiable w.r.t.~$t$.
\end{lemma}

\begin{proof}
For $j=1,2$ denote by $\theta_j(s)$  the parameterization  of $\Gamma_j$
 by its half--plane capacity. Let $t\in[0,1]$ and  let $0\le x_0 \leq t.$ Then
 there exists a unique $y_0 \in [0,1]$ such that
 $\theta_1[0,x_0]\cup\theta_2[0,y_0]$ has half--plane capacity $2t$. Apply the
 mapping $A:=g_{\theta_1[0,x_0]}$ on the two slits. We define $\chi(\Delta):=A(\theta_1(x_0+\Delta))$ for all $\Delta \ge 0$ small
 enough. Then we have by Lemma \ref{hcap1} b),
 $\hcap(\chi[0,\Delta])=\hcap(\theta_1[0,x_0+\Delta])-\hcap(\theta_1[0,x_0])=2\Delta$.
 Next we apply the mapping $B:=g_{A(\theta_2[0,y_0])}.$ Let $\psi(\Delta):=B(\chi(\Delta))$.  Now we have $$\frac{\hcap(\psi[0,\Delta])}{2\Delta}\to B'(\chi(0))^2 \quad \text{for} \quad \Delta\to 0,$$ see \cite{MR1879850}, Lemma 2.8. Note that $B'(\chi(0))^2$ depends on $x_0$ and $t$ only. So let us define the function $C(x_0,t):=B'(\chi(0))^2.$ $C$ has the following properties:
 \begin{itemize} \item $(x_0,t)\mapsto C(x_0,t)$ is continuous: $A$ and $\chi(0)$ depend continuously on $x_0$. Furthermore, $y_0$ depends continuously on the pair $(x_0,t),$ so $B$ depends continuously on $(x_0,t)$ as well as $C(x_0,t)=B'(\chi(0))^2.$  
\item $C$ is continuously differentiable with respect to $t$: For fixed $x_0$, both the value $y_0$ and the mapping $B$ are continuously differentiable with respect to $t,$ see section 4.6.1 in \cite{Lawler:2005}. Hence, as $\chi(0)$ is fixed, also $B'(\chi(0))$ is continuously differentiable w.r.t.~$t$.
\item $C(x_0,t)\in (0,1)$ for all $0\leq x_0 < t \leq 1$: see Proposition 5.15 in \cite{Lawler:2005}.
\end{itemize}
Now we look at the case $x_0=x(t).$ Then
$x(t+h)-x(t)=\hcap(\chi[0,\Delta(h)])=2\Delta(h)$ and we know
that $$\lim_{h\downarrow 0}\frac{\hcap(\psi[0,\Delta(h)])}{h}=2\lambda.$$
This follows by applying Lemma \ref{nointer} to the slits $g_t(\Gamma_1\setminus \gamma_1[0,t])$ and $g_t(\Gamma_2\setminus \gamma_2[0,t])$. Thus
 $$\lim_{h\downarrow 0}\frac{x(t+h)-x(t)}{h}=\lim_{h\downarrow
  0}\frac{2\Delta(h)}{h}=\lim_{h\downarrow 0}\frac{2\Delta(h)\cdot
  \hcap(\psi[0,\Delta(h)])}{\hcap(\psi[0,\Delta(h)])\cdot
  h}=\frac{2\lambda}{C(x(t),t)}.$$  Hence, the right derivative of $x(t)$ exists
and is continuous, so $x(t)$ is continuously differentiable, see Lemma 4.3 in
\cite{Lawler:2005}, and 
$$\dot{x}(t)=\frac{2\lambda}{C(x(t),t)}\, .$$
\end{proof}

\subsection{\label{uniqueness}Proof of Theorem \ref{Charlie2}, Part II (Uniqueness)}

Let $\nu, \mu\in[0,1]$ be constant weights  and  $U_1,U_2,V_1,V_2:[0,1]\to\R$
be continuous driving functions such that the solutions $g_t$ and $h_t$ of
 \begin{eqnarray*}
  \dot{g}_t(z)&=&\frac{2\nu}{g_t(z)-U_1(t)}+\frac{2(1-\nu)}{g_t(z)-U_2(t)},\, g_0(z)=z, \quad\text{and}\\
 \quad\dot{h}_t(z)&=&\frac{2\mu}{h_t(z)-V_1(t)}+\frac{2(1-\mu)}{h_t(z)-V_2(t)}, \,h_0(z)=z,
 \end{eqnarray*}
satisfy $g_1=h_1=g_{\Gamma_1\cup \Gamma_2}.$

\medskip

Assume $\nu >\mu$.
Let $x_1(t)$ and $x_2(t)$ be the half--plane capacities of the generated part
of $\Gamma_1$ at time $t$ with respect to $g_t$ and $h_t$, and let $y_1(t)$
and $y_2(t)$ be the corresponding half--plane capacities of $\Gamma_2.$ Then
$\dot{x}_1(0)=\nu>\mu=\dot{x_2}(0)$ by Lemma \ref{nointer}. Consequently,
$x_1(t)>x_2(t)$ for all $t \ge 0$ small enough. 
Since $x_1(1)=\hcap(\Gamma_1)=x_2(1)$, there is a first time $\tau \in (0,1]$
such that $x_1(\tau)=x_2(\tau)$. Then $x_1(t)>x_2(t)$ for every $t \in
[0,\tau)$, so $\dot{x}_1(\tau)\le \dot{x}_2(\tau)$. On the other hand,
 Lemma \ref{LipDiff} shows that
$$ \dot{x}_1(\tau)=\frac{2 \nu}{C(x_1(\tau),\tau)} > \frac{2
  \mu}{C(x_2(\tau),\tau)}=\dot{x}_2(\tau) \, , $$
a contradiction. Hence we know that $\nu \le \mu$. By switching the roles of
$\nu$ and $\mu$, we deduce $\nu=\mu$.

\medskip

Next, again with the help of  Lemma \ref{LipDiff}, we see that both
functions $x_1$ and $x_2$ are solutions to the same initial value problem,
$$ \dot{x}(t)=\frac{2 \mu}{C(x(t),t)}\, , \quad x(0)=0\, , $$
where $(x_0,t)\mapsto 2\mu/C(x_0,t)$ is continuous, positive and Lipschitz
continuous in $t$. However, the solution to such a problem is
unique according to Theorem 2.7 in \cite{Cid:2003}. Hence $x_1=x_2$ and also
$y_1=y_2$, so
we have $$H(t):=\theta_1[0,x_1(t)]\cup \theta_2[0,y_1(t)] =
\theta_1[0,x_2(t)]\cup \gamma_2[0,y_2(t)]$$ for all $t.$ Using the geometric
meaning of the driving functions, we finally get
$$ U_j(t)=g_{H(t)}(\theta_j(x_1(t))=g_{H(t)}(\theta_j(x_2(t))=V_j(t)$$ for $j=1,2$.
This completes the proof of the uniqueness statement of Theorem \ref{Charlie2}.
\hfill \proofsymbol
\subsection{Some remarks}

We finish Section \ref{Chiara} by noting some remarks about further conclusions revealed by the proof of Theorem \ref{Charlie2}.

We start with a simple estimate for the constant coefficients as a consequence of Lemma \ref{LipDiff}. 
\begin{corollary} Given slits $\Gamma_1,...,\Gamma_n$ with disjoint closures, then the corresponding coefficients $\lambda_1,...,\lambda_n$ from Theorem \ref{Charlie2} satisfy 
 $$2T-\sum_{m\not=k} \hcap(\Gamma_m) < 2\lambda_k T < \hcap(\Gamma_k).$$
\end{corollary}
\begin{proof}
 By Lemma \ref{LipDiff}, $\hcap(\Gamma_k)> \int_0^T 2\lambda_k \, ds= 2\lambda_k T.$\\ Furthermore, $2\lambda_kT = 2T-\sum_{m\not=k} 2\lambda_mT  > 2T-\sum_{m\not=k} \hcap(\Gamma_k).$
\end{proof}

Next we note the - quite intuitive - ``monotonicity'' of the constant coefficients.
\begin{corollary}
Let $\Gamma_1,...,\Gamma_n$ be pairwise disjoint slits with coefficients $\lambda_1,...,\lambda_n$ and let $\Gamma\subsetneq \Gamma_k$ be a subslit of $\Gamma_k.$ Denote by $\mu_1,...,\mu_k$ the coefficients that belong to the configuration $\Gamma_1,...,\Gamma_{k-1},\Gamma,\Gamma_{k+1},...,\Gamma_n.$ Then $$\lambda_k>\mu_k.$$
\end{corollary}
\begin{proof}

Let $\gamma(t)$ be the parameterization of $\Gamma_k$ that we obtain from producing $\Gamma_1,...,\Gamma_n$ by the multiple-slit equation with constant coefficients and let $\alpha(t)$ be the corresponding parameterization of $\Gamma$ obtained by the multiple-slit equation with constant coefficients for the configuration $\Gamma_1,...,\Gamma_{k-1},\Gamma,\Gamma_{k+1},...,\Gamma_n.$ Furthermore, let $x(t)=\hcap(\gamma(0,t])$ and $y(t)=\hcap(\alpha(0,t]).$\\
 By Lemma \ref{LipDiff} we know that $x$ and $y$ satisfy the initial value problems
$$\dot{x}(t)=\frac{2\lambda_k}{C(x(t),t)}, \quad x(0)=0, \qquad \dot{y}(t)=\frac{2\mu_k}{C(y(t),t)}, \quad y(0)=0,$$
where $C$ is continuous and continuously differentiable w.r.t. $t$.\\
Let $2T=\hcap(\Gamma_1\cup...\cup\Gamma_{k-1}\cup\Gamma\cup\Gamma_{k+1}\cup...\cup\Gamma_n)$ and assume that $\lambda_k \leq \mu_k.$ Then we have $x(t)\leq y(t)$ for all $t\in[0,T],$ in particular $x(T)\leq y(T).$ But $$  2y(T)= \hcap(\Gamma) < 2x(T).$$
\end{proof}

A direct consequence of Theorem \ref{Charlie2} is that infinite slits can be generated with arbitrary constant coefficients:
\begin{corollary}
 Let $\gamma_1,..., \gamma_n:[0,1]\to \overline{\Ha}\cup\{\infty\}$ be $n$ simple curves with $\gamma_1(1)=...=\gamma_n(1)=\infty, \gamma_1(0),...,\gamma_n(0)\in\R$, $\gamma_1(0,1)\cup...\cup \gamma_n(0,1)\subset\Ha$ and $\gamma_j[0,1)\cap\gamma_k[0,1)=\emptyset$ whenever $j\not=k.$ Then, for every $(\lambda_1,...,\lambda_n)\in(0,1)^n$ with $\sum_{k=1}^n\lambda_k=1,$ there exist unique driving functions $U_1,...,U_n:[0,\infty)\to\R$ such that equation (\ref{more}) with $T=\infty$ generates the curves $\gamma_1,...,\gamma_n$, i.e. the generated hulls $K_t$ satisfy $$\bigcup_{0\leq t< \infty}K_t=\bigcup_{k=1}^n \gamma_k(0,1).$$
\end{corollary}

\begin{remark}
 In Theorem \ref{Charlie2} we assumed that $\overline{\Gamma_j} \cap \overline{\Gamma_k}=\emptyset,$ $j\not=k.$ In particular, all slits have different starting points on the real axis. If we relax the latter condition and two (or more) slits have a common starting point, then we should still find unique constant coefficients. Also, our proof should still work in this case. However, this certainly requires some non-trivial modifications, e.g. a generalization of Lemma \ref{hcap2new}.

Now suppose two slits intersect in finitely many points. Then we can reduce this case to the case of slits intersecting only in their starting point, \textit{provided} that one slit does not enclose the tip of the other slit. Conversely, when the tip of one slit is enclosed by another slit, then these curves can never be generated by a Loewner equation with constant coefficients, see \cite{Graham2007}, Appendix.
\end{remark}

\begin{remark}
 Given two (or more) disjoint slits $\Gamma_1$ and $\Gamma_2$, how can we calculate the coefficient $\lambda$ corresponding to $\Gamma_1$ and the driving functions $U_1$ and $U_2$? It seems that there is no way to calculate $\lambda$ by computing simple geometric quantities of the two slits. However, in principle, one can use the proof of Theorem \ref{Charlie} to find a numerical approximation of $\lambda$ as well as of $U_1$ and $U_2.$\\
Conversely, given $\lambda,$ $U_1$ and $U_2,$ one can use the approximation  $$\int_0^t \frac{2\lambda}{z-U_{1}(\tau)}+\frac{2(1-\lambda)}{z-U_{2}(\tau)} \; d\tau   \quad \approx \quad  \int_0^t \frac{2\alpha_{n,\lambda}(\tau)}{z-U_{1}(\tau)}+\frac{2(1-\alpha_{n,\lambda}(\tau))}{z-U_{2}(\tau)} \; d\tau, $$
to solve the multiple-slit equation numerically by solving one-slit equations only. Some ways how to solve the one-slit equation are explained in \cite{MR2570752} and \cite{MR2348786}, see also \cite{HuyTran} for a convergence result concerning the simulation of SLE.\\
 In numerical mathematics, such a method - dividing a vector field into a sum of ``simple'' vector fields and integrating each of them separately - is usually called ``splitting method''. 
\end{remark}

\section{On constant coefficients in the radial multiple-slit equation}\label{Christoph}

Already in 1936, E. Peschl considered the multiple--slit version of Loewner's radial one-slit equation (\ref{Obama}), see \cite{0016.03501}, Section IV. He proved that for every disjoint union of $n$ Jordan arcs $\Gamma_1, \ldots, \Gamma_n$ in $\overline{\D} \backslash \{0\}$ such that 
$\D \backslash \Gamma$ is simply connected, there are continuous parameter functions
$\lambda_1, \ldots, \lambda_n : [0,T] \to \R$ with $\lambda_j(t) \ge 0$ and
$\lambda_1(t)+\ldots+\lambda_n(t)=1$ for every $t \in [0,T]$, and continuous
driving functions $\kappa_j : [0,T] \to \partial \D$ such that the solution
$w_t$ to the radial Loewner equation
\begin{equation} \label{eq:rad}
 \dot{w}_t(z)=-w_t(z)\sum \limits_{j=1}^n \lambda_j(t)
\frac{\kappa_j(t)+w_t(z)}{\kappa_j(t)-w_t(z)} \, , \qquad w_0(z)=z \, ,\end{equation}
has the property that $w_T$ maps $\D$ conformally onto $\D \backslash (\Gamma_1,\ldots,\Gamma_n)$.

Now one can also ask for the existence and uniqueness of constant coefficients for the slits $\Gamma_1,...,\Gamma_n$ in this radial case. As already mentioned, D. Prokhorov has proven the existence and uniqueness of constant coefficients for several slits for the radial equation under the assumption that all slits are \emph{piecewise analytic}.\\

It seems to be natural that one can apply the same idea of the proof of Theorem \ref{Charlie2} here, too, in order to show that one can drop any assumption on the regularity of the slits to 
	 generate them with constant coefficients. On the other hand, we used several properties of the half-plane capacity in the proof which seem not to have counterparts in the radial case at all. \\
The author of this thesis presented Theorem \ref{Charlie2} and its proof during the Doc Course ``Complex Analysis and Related Areas'' (IMUS 2013) in Sevilla. Another participant of this course was Christoph B\"ohm, who was working on a radial multiple-slit equation for multiply connected domains (a special case of a Komatu-Loewner differential equation). He realized how to change the proof of Theorem \ref{Charlie2} such that the technical difficulties can be handled in the radial case (and even for multiply connected domains!) by using results from his own work. The result can be found in the joint paper \cite{BoehmSchl} and we cite the special case of a simply connected domain. 

Let $f:\D\setminus(\Gamma_1\cup\ldots\cup\Gamma_n)\to \D$ be the unique conformal mapping with $f(0)=0$ and $f'(0)>0$, then $f'(0)> 1$ and the 
\emph{logarithmic mapping radius} of $\D\setminus(\Gamma_1\cup\ldots\cup\Gamma_n)$ is 
defined to be the real number $\log f'(0)> 0$.

\begin{theorem}[Corollary 2 in \cite{BoehmSchl}]
 	Let  $\Gamma_1,\ldots,\Gamma_n$ be slits in $\D\setminus\{0\}$ with disjoint closures and let $L$ be the logarithmic 
	mapping radius of $\D\setminus(\Gamma_1\cup\ldots\cup\Gamma_n)$. 
	Then there exist unique $\lambda_1,...,\lambda_n\in(0,1)$ with $\sum_{k=1}^n \lambda_k=1$ and unique continuous 
	driving functions $\kappa_1,\ldots,\kappa_n:[0,L]\to\partial\D$ such that the solution of the 
	Loewner equation 
	\begin{equation*}
		 \dot{w}_t(z)=-w_t(z)\sum \limits_{j=1}^n \lambda_j
\frac{\kappa_j(t)+w_t(z)}{\kappa_j(t)-w_t(z)} , 
		\qquad w_0(z)=z,
	\end{equation*}
	generates the slits $\Gamma_1,\ldots,\Gamma_n,$ i.e. $w_L$ maps 
	$\D$ conformally onto $\D\setminus(\Gamma_1\cup\ldots\cup\Gamma_n)$.
\end{theorem}

\begin{remark} In Corollary 2 from \cite{BoehmSchl}, this statement is formulated for the corresponding reversed Loewner ODE. In the simply connected case, however, the statement can be easily transferred from the reversed to the ordinary ODE and vice versa.  
\end{remark}

\section{The simple-curve problem}\label{Sip}

Recall the \textit{simple-curve problem}: Under which conditions does the multiple-slit equation \ref{more} generate simple curves?\\

A sufficient condition for getting slits in the one-slit equation was found by J. Lind, D. Marshall and S. Rohde. First we need the following two definitions.

\begin{definition}${}$\\[-0.5cm]
 \begin{enumerate}[(1)]
  \item A slit $\Gamma$ is called \emph{quasislit} if $\Gamma$ is the image of the line segment $[0,i]$ under a quasiconformal mapping $Q:\Ha\to\Ha$ with $Q(\Ha)=\Ha$ and $Q(\infty)=\infty.$
\item We let $\Lip$ be the set of all continuous functions $U:[0,T]\to\R$ with $|U(t)-U(s)|\leq c\sqrt{|s-t|},$ for a $c>0$ and all $s,t\in[0,T]$.
 \end{enumerate}

\end{definition}

A quasislit is a slit that is a quasiarc\footnote{A \emph{quasiarc} is the image of a line segment under a quasiconformal mapping.} approaching $\R$ nontangentially, see Lemma 2.3 in \cite{MarshallRohde:2005}. 

Together with the metric characterization of quasiarcs by Ahlfors' three point property (also called \emph{bounded turning property}, see \cite{MR0344463}, Section 8.9 or \cite{MR0210889}, Theorem 1), we obtain the following geometric description of quasislits:

A slit $\Gamma$ is a quasislit if and only if $\Gamma$ approaches $\R$ nontangentially and  
\begin{equation}\label{boundedturning} \sup_{
\substack{x,y \in \Gamma \\ x\not=y}} \frac{\diam(x,y)}{|x-y|} < \infty,\end{equation} 
where we denote by $\diam(x,y)$ the diameter of the subcurve of $\Gamma$ joining $x$ and $y.$

Now the following connection between $\Lip$ and quasislits holds.

\begin{theorem}\label{lmr}(Theorem 1.1 in \cite{MarshallRohde:2005} and Theorem 2 in \cite{Lind:2005})\textcolor{white}{0}\\
 If $\Gamma$ is a quasislit with driving function $U$, then $U \in \Lip$. Conversely, if $U \in \Lip$ and for every $t\in[0,T]$ there exists an $\eps>0$  such that
$$ \sup_{\substack{r,s\in[0,T]\\ r\not=s\\|r-t|,|s-t|<\eps}}\frac{|U(r)-U(s)|}{\sqrt{|r-s|}}<4, $$
then $\Gamma$ is a quasislit.
\end{theorem}

In the following, we take a look at some necessary and sufficient conditions for the multiple-slit equation to generate simple curves.  In Section \ref{Ingwer}, we have a look at the following question:
\begin{eqnarray}\label{4ne}\begin{aligned}
&\textit{Is the constant $4$ in Theorem \ref{lmr} sharp for generating quasislits?} &\end{aligned}
\end{eqnarray} 

Finally, in Section \ref{Approach} we will look for relations between the driving functions and the way how the slits approach $\R.$

\subsection{Continuous driving functions}

 In the following, we will look at the multiple-slit equation
\begin{equation}\label{more2}\dot{g}_t(z)=\sum_{k=1}^n\frac{2\lambda_k(t)}{g_t(z)-U_{k}(t)},\quad g_0(z)=z, \quad t\in[0,T], \end{equation}
with the following assumptions:
\begin{itemize}
\item[(a)] Each $\lambda_k:[0,T]\to[0,1]$ is continuous and $\lambda_k(t)\in(0,1)$ for all $t\in[0,T].$
\item[(b)] $\sum_{k=1}^n \lambda_k(t)=1$ for all $t\in[0,1].$
\item[(c)] The driving functions $U_k:[0,T]\to \R$ are continuous and
\item[(d)] $U_j(t)<U_k(t)$ for all $t\in[0,T]$ and $1\leq j<k\leq n.$
\end{itemize}

Condition (b) simply implies that the generated hulls $K_t$ satisfy $\hcap(K_t)=2t$ for all $t\in[0,T].$ \\ 
We are interested in those cases where equation (\ref{more2}) generates hulls that have exactly $n$ connected components. Assumption (d) is necessary for this case. Discontinuous driving functions could also generate $n$ connected components, but we will not consider this case. Property (a) is not necessary either, but we want to exclude that a coefficient function $\lambda_k(t)$ can be equal to $0$ in a whole interval. Clearly, in this case, we could have less than $n$ connected components. 

\begin{lemma}\label{Miriam}
Let $\{K_t\}_{t\in[0,T]}$ be hulls generated by equation (\ref{more2}). Then there exists $\tau>0$ such that $K_t$ falls into $n$ connected components $C^t_1,...,C^t_n$ for all $t\in(0,\tau],$ with $C^s_j \subset C^t_j$ for all $j=1,...,n$ and $s\leq t.$\\
Furthermore, for every $j=1,...,n,$ the family $\{C^t_j\}_{t\in[0,\tau]}$ satisfies the local growth property and $x(t):=\hcap(C^t_j)$ is continuously differentiable with $$\dot{x}(0)=2\lambda_j(0)\quad \text{and} \quad \dot{x}(t)> 2\lambda_j(t) \quad \text{for all} \quad t\in(0,\tau].$$
\end{lemma}
\begin{proof}
The statements can be shown by using the same ideas we used in the proofs of Lemma \ref{nointer} and Lemma \ref{LipDiff}. So we only sketch the steps and, in order to simplify notation, we let $n=2$.\\
First, consider the one-slit equation $$  \dot{g}_{t,n}(z)= \frac{2\alpha_{n}(\tau)}{g_{t,n}(z)-U_{1}(\tau)}+\frac{2(1-\alpha_{n}(\tau))}{g_{t,n}(z)-U_{2}(\tau)} \; d\tau,\qquad g_{0,n}(z)=z\in\Ha,$$
where $\alpha_{n}(t)$ is a step function with \begin{equation} \alpha_{n}(t)=\begin{cases}
1 \quad \text{when} \quad t\in (\frac{k}{2^n}, \frac{k+\lambda(\frac{k}{2^n})}{2^n}),\; k\in\{0,...,2^n-1\},\\
0 \quad \text{when} \quad t\in (\frac{k+\lambda(\frac{k}{2^n})}{2^n}, \frac{k+1}{2^n}),\; k\in\{0,...,2^n-1\}.                                                                  \end{cases}\end{equation}
Denote by $K_{t,n}$ the generated hulls. Then we know that $\Ha \setminus K_{t,n} \to \Ha \setminus K_t$ for every $t\in[0,T].$ As in the proof of Lemma \ref{nointer}, we conclude that there exists $\tau$ such that $K_{t,n}$ consists of two subhulls $A_{t,n},$ $B_{t,n}$ for every $t\in[0,\tau]$ and $n\in\N$ with $A_{s,n}\subset A_{t,n},$ $B_{s,n}\subset B_{t,n}$ whenever $s\leq t$ and $A_{t,n}$ grows whenever $\alpha_n(t)=1$, $B_{t,n}$ grows whenever $\alpha_n(t)=0$. From this it follows that $K_t$ consists of two disjoint subhulls $C^t_1,$ $C^t_2$ for all $t\in[0,\tau]$ with $C^s_j \subset C^t_j$, $s\leq t$, and $\Ha \setminus A_{t,n} \to \Ha \setminus C_1^t,$ $\Ha \setminus B_{t,n} \to \Ha \setminus C_2^t$ for $n\to\infty$. It is easy to verify that the family $\{C_j^t\}_{t\in[0,\tau]}$ has the local growth property. In particular, this shows that $C_1^t$ and $C_2^t$ are connected, so $K_t$ has exactly two connected components for all $t\in(0,\tau]$.\\
Finally, by repeating the steps of the the proof of Lemma \ref{LipDiff}, we obtain that $x(t):=\hcap(C^t_j)$ is continuously differentiable with $\dot{x}(0)=2\lambda_j(0)$ and $\dot{x}(t)> 2\lambda_j(t)$ for all $t\in(0,\tau].$ \end{proof}

The global statement that also $K_T$ has $n$ connected components is not true in general, as, for example, $C^t_1$ could hit $C^t_2$ for some $t\in[0,T].$ \\

 In Example \ref{Felix}, we have already seen that there are continuous driving functions that generate curves hitting the real axis (or itself) at a certain time. From this time on, the hull is not a simple curve any longer.
There are further, more subtle obstacles preventing the one-slit equation from producing simple curves as the following example shows.

\begin{example}\label{spiral}
There exists a driving function $U\in \Lip$ such that $K_t$ is a slit $\gamma(0,t]$ for all $t\in (0,T)$ and for $t\to T$ the curve $\gamma$ wraps infinitely often around, say, $\mathcal{B}(2i,1).$ Hence $K_T=\gamma(0,T)\cup \overline{\mathcal{B}(2i,1)}$ is not locally connected, see Example 4.28 in \cite{Lawler:2005}. \hfill $\bigstar$
\end{example}

In order to distinguish between these two kinds of obstacles, one has introduced two further notions, which are more general than ``the hull is a slit''.

\begin{definition}
Let $\{K_t\}_{t\in[0,T]}$ be a family of hulls generated by the multiple-slit equation (\ref{more2}). We say that $\{K_t\}_{t\in[0,T]}$ is \emph{generated by curves}, if there exist $n$ continuous functions $\gamma_1,\ldots,\gamma_n:[0,T]\to\overline{\Ha},$ such that $\Ha\setminus K_t$ is the unbounded component of $\Ha\setminus (\gamma_1[0,t]\cup\ldots \cup \gamma_n[0,t])$ for every $t\in[0,T].$
\end{definition}
\begin{remark}
Several properties of hulls that are generated by curves are described in \cite{Lawler:2005}, Section 4.4.
\end{remark}

The second notion of ``welded hulls'' is discussed in the next section. We end this subsection by a last example.

\begin{example}
 It is possible to generate space-filling curves by the one-slit equation. The Hilbert space-filling curve, e.g., can be generated by a driving function which is in $\Lip$, see Corollary 1.4 in \cite{spacefill}. The corresponding hulls are generated by a curve, which is ``self-touching everywhere''. ${}$\hfill $\bigstar$
\end{example}

\subsection{Welded hulls}

First, we consider the backward equation to (\ref{more2}) with real initial values, i.e.,
\begin{equation}\label{back}
 \dot{x}(t) = \sum_{j=1}^n\frac{-2\lambda_j(T-t)}{x(t)-U_j(T-t)}, \quad x(0)=x_0\in\R\setminus\{U_1(T),...,U_n(T)\}.
\end{equation}
Here, the solution may not exist for all $t\in[0,T].$ If a solution ceases to exist, say at $t=s$, it will hit a singularity, i.e. $\lim_{t\uparrow s}x(t)=U_j(s),$ with at most two possibilities for $j$, depending on the position of $x_0$ with respect to $U_1(T),...,U_n(T).$\\

Now assume that there are two different solutions $x(t),y(t)$ with $x(0)=x_0<y_0=y(0).$ If $x(t)$ and $y(t)$ meet a singularity after some time, i.e. $\lim_{t\uparrow s}x(t)=\lim_{t\uparrow s}y(t)=U_j(T-s)$ for some $s\in(0,T]$, then $x_0$ and $y_0$ lie on different sides  with respect to $U_j(T)$, i.e. $x_0<U_j(T)<y_0.$ Otherwise, the difference $y(t)-x(t)$ would satisfy 
$$ \dot{y}(t)-\dot{x}(t)= \sum_{j=1}^n\frac{2\lambda_j(T-t)(y(t)-x(t))}{(y(t)-U_j(T-t))(x(t)-U_j(T-t))}>0$$
for all $0\leq t<s$ and thus, $\lim_{t\uparrow s}(x(t)-y(t))=0$ would be impossible.\\
  Consequently, for any $1\leq j\leq n$ and any $s\in(0,T]$, there are at most two initial values so that the corresponding solutions will meet $U_j(T-s)$.

\begin{definition}
Let $\{K_t\}_{t\in[0,T]}$ be a family of hulls generated by the multiple-slit equation (\ref{more2}). We say that $\{K_t\}_{t\in[0,T]}$ \emph{is welded} if for every $j\in\{1,...,n\}$ and every $s\in(0,T]$ there exist exactly two real values $x_0^j,y_0^j$ with $x_0^j<U_j(T)<y_0^j$ such that the corresponding solutions $x(t)$ and $y(t)$ of (\ref{back}) with $x(0)=x_0^j,$ $y(0)=y_0^j$ satisfy $$\lim_{t\uparrow s}x(t)=\lim_{t\uparrow s}y(t)=U_j(T-s).$$

\end{definition}

Before we give some explanations for this definition, we need the following characterization.

\begin{proposition}\label{Fritz} The following statements are equivalent:
\begin{itemize}
 \item[a)] $\{K_t\}_{t\in[0,T]}$ is welded.
\item[b)] For every $\tau\in[0,T)$ there exists $\eps>0$ such that for all $x_0\in\R\setminus\{U_1(\tau),...,U_n(\tau)\}$ the solution $x(t)$ of 
\begin{equation}\label{more3}\dot{x}(t)=\sum_{j=1}^n\frac{2\lambda_j(t)}{x(t)-U_{j}(t)},\quad x(\tau)=x_0, \end{equation}
does not hit $U_1(t)$, ..., $U_n(t)$ for $t<T$ and satisfies $|x(T)-U_j(T)|>\eps$ for all $j\in\{1,...,n\}.$ 
\end{itemize}
\end{proposition}

\begin{proof}
$a)\Longrightarrow b):$ Firstly, the solution $x(t)$ to (\ref{more3}) exists locally, say in the interval $[\tau, T^*),$ $T^*\leq T.$ 
Now we know that there are $x_0^j,y_0^j,$ $j=1,...,n,$ such that the solutions $x_j(t)$ and $y_j(t)$ to equation (\ref{back}) with initial values $x_0^j$ and $y_0^j$ respectively hit the singularity $U_j(\tau)$ (at $s=T-\tau$) and $x_0^j<U_j(T)<y_0^j$. But this implies that $x(t)$ can be extended to the interval $[0,T^*]$ with $|U_j(T^*)-x(T^*)|>\eps,$ where $\eps:=\min\{U_j(T^*)-x_j(T-T^*),y_j(T-T^*)-U_j(T^*)\}.$\\
$b)\Longrightarrow a):$ Let $s\in(0,T]$ and $\tau=T-s.$ We set $a_n:=U_j(\tau)-\frac{1}{n}$ for all $n\in\N.$ The solution $x_n(t)$ of (\ref{more3}) with initial value $a_n$ exists up to time $T$. Hence we can define $\xi_n:=x_n(T)$ for all $n\geq N$ and we have $U_j(T)-\xi_n>\eps.$ The sequence $\xi_n$ is increasing and bounded above, and so it has a limit $x_0<U_j(T).$ Then the solution $x(t)$ of (\ref{back}) with $x_0$ as initial value satisfies $$\displaystyle \lim_{t\uparrow s}x(t)=\lim_{n\to\infty}a_n=U_j(\tau)=U_j(T-s).$$
The second value $y_0$ can be obtained in the same way by considering the sequence $U_j(\tau)+\frac{1}{n}$ instead of $a_n.$
\end{proof}

Figure \ref{2slitsmeet} shows an example of a situation for $n=2$, where all solutions to (\ref{back}) with initial value in $(U_1(T),U_2(T))$ will meet $U_2(t)$. So the whole interval $(U_1(T),U_2(T))$ ``belongs'' to the second slit.

\begin{figure}[h]
  \centering
       \includegraphics[width=155mm]{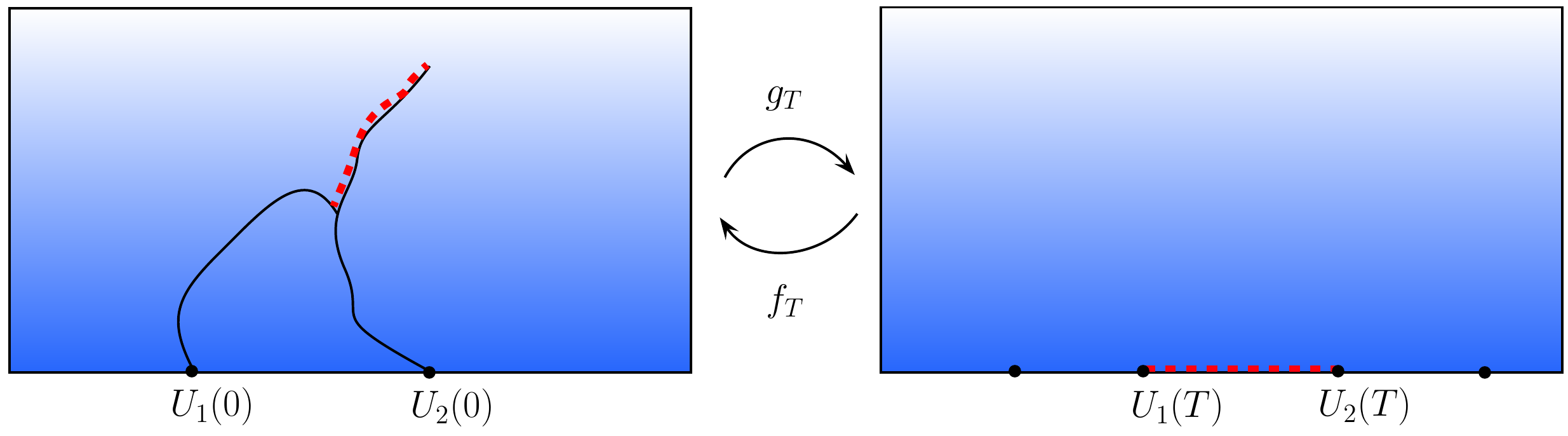}
\caption{The left slit hits the right one at $t=T.$}\label{2slitsmeet}
\end{figure}  

\begin{corollary}
Let $\{K_t\}_{t\in[0,T]}$ be generated by equation (\ref{more2}). If $\{K_t\}_{t\in[0,T]}$ is welded, then $K_T$ has $n$ connected components.
\end{corollary}
\begin{proof}
Clearly, $K_T$ cannot have more than $n$ connected components by Lemma \ref{Miriam}.\\
Let  $N\in\N$ and for $j=1,...,n$ define $p_j^-=U_j(0)-\frac1{N},$ $p_j^+=U_j(0)+\frac1{N}.$
We can choose $N$ large enough such that $p_{j+1}^->U_j(0)$ and $p_{j}^+<U_{j+1}(0),$ for all $j=1,...,n-1$. Furthermore, choose $\eps>0$ so small that
$$p_1^-+\eps < U_1(0) < p_1^+ -\eps < p_1^++\eps < U_2(0) <... < U_n(0) < p_n^+-\eps.$$
Now solve equation (\ref{more3}) with $\tau=0$ and initial value $x_0\in(p_1^+ -\eps,  p_1^++\eps).$ The solution $x(t)$ will always exist until $t=T$ and the set $E:=\{x(T)\;|\; x_0\in(p_1^+ -\eps,  p_1^++\eps)\}$ forms an open set. Thus, the cluster set of $K_T$ with respect to $g_{K_T}$ has at least two connected components, one lying on the left side of $E$ and one on the right side. If we pass on to the case $x_0\in(p_2^+ -\eps,  p_2^++\eps),$ we get one further connected component and so on.\end{proof}

\begin{remark}\label{Schorsch}
 The notion of welded hulls was introduced in \cite{MarshallRohde:2005} for the radial Loewner equation. The chordal case is considered in \cite{Lind:2005}. \\
As we have seen, welded hulls always have $n$ connected components and, informally speaking, each component has a left and a right side.

Assume that $\{K_t\}_{t\in[0,T]}$ is welded. Let $C_T$ be the component of $K_T$ corresponding to $U_j(0)$ and let the interval $I=[a,b]$ be the cluster set of $\overline{C_T}$ with respect to $g_T$, then $a<U_j(T)<b$ and for every $a\leq x_0<U_j(T)$ there exists $U_j(T)<y_0\leq b$ such that the solutions to (\ref{back}) with initial values $x_0$ and $y_0$ hit the singularity at the same time. This gives a welding homeomorphism $h:[a,b]\to[a,b]$ by defining $$h(x_0):=y_0,\quad h(y_0):=x_0, \quad h(U_j(T)):=U_j(T).$$

If $C_T$ is a slit, then $h$ is directly connected to ``classical'' welding (or sewing) homeomorphisms of a Jordan domain and one can relate properties of $h$ to properties of the slit. However, in general, $C_T$ is not a slit and therefore $h$ is only a ``generalized welding'' homeomorphism in the sense of \cite{MR1139801}, see also \cite{MR2373370}.
\end{remark}

The hull of Example \ref{spiral}, which is sketched in picture a) of Figure \ref{Fi34}, is not generated by a curve. Picture c) shows an  example of a hull that is generated by a curve. Here, the curve hits itself and the real axis and consequently, this hull is not welded. The hull in picture b), a ``topologist's sine curve'' approaching a compact interval on $\R$, is neither welded nor generated by a curve. 
\begin{figure}[h]
 \centering \includegraphics[width=135mm]{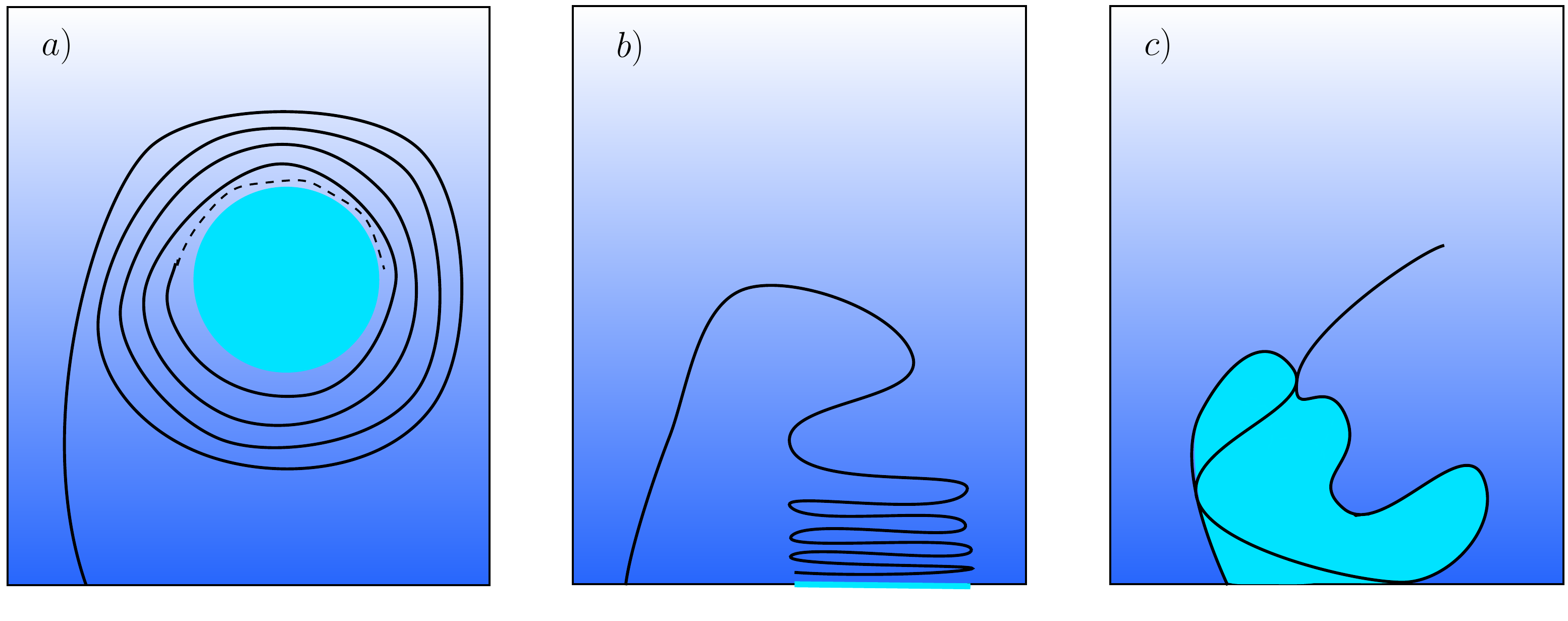}
\caption{Three cases of hulls that are not slits.}\label{Fi34}
\end{figure}

\begin{proposition} 
Let $\{K_t\}_{t\in[0,T]}$ be a family of hulls generated by equation (\ref{more2}). Then $K_T$ is a union of slits with disjoint closures if and only if $\{K_t\}_{t\in[0,T]}$ is generated by curves and welded.
\end{proposition}
\begin{proof}
For the non-trivial direction of the statement, see Lemma 4.34 in \cite{Lawler:2005}.
\end{proof}

\subsubsection{A necessary condition}

Next we derive a non-trivial, necessary condition for getting welded hulls from equation (\ref{more2}), basically by using Example \ref{Felix}.

\begin{proposition}\label{nec} Let $\{K_t\}_{t\in[0,T]}$ be a family of hulls generated by the multiple-slit equation (\ref{more2}). Assume there exists an $s\in(0,T]$ and $j\in\{1,...,n\}$  such that $$\liminf_{h\downarrow0}\frac{|U_j(s)-U_j(s-h)|}{\sqrt{h}}> 4\sqrt{\lambda_j(s)}.$$
Then $\{K_t\}_{t\in[0,s]}$ is not welded.
\end{proposition}
\begin{proof}
Without loss of generality we assume that $s=1$ and $U_j(1)=0.$\\
We know that there exists $c>4$ and $\eps>0$ such that $|U_j(1)-U_j(1-h)|=|U_j(1-h)|\geq c\sqrt{\lambda_j(1)\cdot h}$ for all $h\in[0,\eps],$ i.e. $|U_j(t)|\geq c\sqrt{\lambda_j(1)}\cdot \sqrt{1-t}$ for all $t\in[1-\eps,1].$ We can assume that 
$$U_j(t)\geq c\sqrt{\lambda_j(1)}\cdot \sqrt{1-t}$$ for all $t\in[1-\eps,1].$\\

First, let $n=1.$ Then $U_1(t)\geq c \cdot \sqrt{1-t}$ for all $t\in[1-\eps,1].$\\
 As the driving function $c\sqrt{1-t}$ generates a simple curve that finally hits $\R$ when $t=1,$ see Example \ref{Felix}, we conclude that for any $t_0 \in [0,1)$ there exists $x_0=x_0(t_0)< c\cdot \sqrt{1-t_0}\leq U_1(t_0)$ such that the solution to \begin{equation}\label{String}\dot{y}(t)=\frac{2}{y(t)-c\sqrt{1-t}},\quad y(t_0)=x_0, \end{equation}
exists until $t=1$ and satisfies $y(1)=0=U(1).$\\
Let $\tau \in[1-\eps,1)$ and consider the real initial value problem

\begin{equation*}\label{Valju}\dot{x}(t)=\frac{2}{x(t)-U_{1}(t)},\quad x(\tau)=x_0(\tau) < U_1(\tau). \end{equation*}
Let $y(t)$ be the solution to (\ref{String}) with $t_0=\tau$. As $\dot{x}(t)\geq \dot{y}(t)$ for all $t\geq \tau$ where $x(t)$ exists, we conclude that, if $x(t)$ exists until $t=1,$ then $U_1(1)\geq x(1)\geq y(1)= U_1(1),$ hence $x(1)=U_1(1)$ and the generated hulls $\{K_t\}_{t\in[0,1]}$ are not welded according to Proposition \ref{Fritz}. \\

Now let $n\geq 2.$ In fact we can apply the same trick here. However, it remains to show that the hulls of a multiple-slit equation are not welded provided that one driving function $U_j$ has the form \begin{equation}\label{Drago}U_j(t)=c\sqrt{1-t}, \qquad c>4\sqrt{\lambda_j(1)}.\end{equation}
Fix $s\in[0,1).$ For $t\in[0,1-s],$ let $A_t=g_s(K_{s+t}\setminus K_s).$ Furthermore, let $d_t=g_{A_t}.$ Assume that $s$ is so close to $1$ that $A_t$ has $n$ connected components for all $t\in[0,1-s].$ Denote by $C_t$ the component that belongs to $U_j$ and let $x(t)=\hcap(C_t).$ Furthermore, let $l_t=g_{C_t}$ and define $H_t$ by $d_t=H_t\circ l_t.$
The mappings $l_t$ satisfy a one-slit Loewner equation
$$ \dot{l}_t=\frac{2\dot{x}(t)}{l_t-V_s(t)}, $$
where $x(t)$ is continuously differentiable with $\dot{x}(0)=2\lambda_j(s),$ see Lemma \ref{Miriam}, and $V_s$ is a continuous driving function. This function $V_s$ is related to $U_j$ via $$V_s(t)=H_t^{-1}(U_j(s+t)).$$

For $s$ big enough, the function $H_t^{-1}$ can be extended to the lower half-plane by reflection and to a fixed neighborhood $\mathcal{N}$ of $\{U_j(t+s)\with t\in[0,1-s]\}$ for all  $t\leq 1-s.$
As the function $[0,1-s]\times \mathcal{N} \ni(t,z)\mapsto F_s(t,z):=H_t^{-1}(z)$ is continuously differentiable, we have
\begin{eqnarray*} V_s(1-s-h)-V_s(1-s) = \frac{\partial F_s}{\partial t}(1-s,U_j(1)) \cdot (-h) + \frac{\partial F_s}{\partial z}(1-s,U_j(1)) \cdot  (U_j(1-h)-U_j(1)) \\ +\Landauo(|h|+|U_j(1)-U_j(1-h)|)
= \frac{\partial F_s}{\partial z}(1-s,U_j(1)) \cdot  (c\sqrt{h}) + \Landauo(\sqrt{|h|})
\end{eqnarray*}
for $h\downarrow 0.$ Thus
$$\liminf_{h\downarrow0}\frac{|V_s(1-s)-V_s(1-s-h)|}{\sqrt{h}} = \left|\frac{\partial F_s}{\partial z}(1-s,U_j(1))\right|\cdot c>\left|\frac{\partial F_s}{\partial z}(1-s,U_j(1))\right|\cdot 4\sqrt{\lambda_j(1)}.$$ 

Moreover, $F_s(1-s,\cdot)$ converges uniformly to the identity for $s\to1.$ Thus, $\frac{\partial F_s}{\partial z}(1-s,U_j(1))\to 1$ for $s\to 1.$ Hence, we can choose $s\in[0,1)$ so close to $1$ that the driving function $V_s$ satisfies
$$\liminf_{h\downarrow0}\frac{|V_s(1)-V_s(1-h)|}{\sqrt{h}}> 4\sqrt{\lambda_j(1)}.$$ 
By noting that $\dot{x}(0)=2\lambda_j(s)\to 2\lambda_j(1)$ for $s\to 1,$ a simple time change for $V_s$ shows that we can apply the statement for $n=1$ to $V_s$ when $s$ is close enough to $1$, which shows that $\{g_s(K_{s+h}\setminus K_s)\}_{h\in[0,1-s]}$ is not welded. Consequently, $\{K_t\}_{t\in[0,1]}$ is not welded.

\end{proof}

If we want the driving functions $U_j$ to generate welded hulls, then the condition of the previous Proposition has to be violated. However, this can be done in two different ways.\\
 First, $U_j$ can be too smooth, as in Theorem \ref{lmr}. \\
A second, totally different possibility is represented by the behavior of a Brownian motion $B(t)$. In this case, we always find a sequence of points $t_n$ with $t_n\to1$ and $\frac{|B(t_n)-B(1)|}{\sqrt{|t_n-1|}}>4$ (with probability one), but $B(t)$ fluctuates so much that the inequality of Proposition \ref{nec} never holds. Let's state the contrapositive again:\\ 

If $K_T$ is welded (e.g., it consists of $n$ disjoint simple curves) then, for every $s\in(0,T],$ and $j\in\{1,...,n\},$ we have\\
\begin{eqnarray*}
 \limsup_{h\downarrow0}\frac{|U_j(s)-U_j(s-h)|}{\sqrt{h}}&\leq& 4 \quad \;\;(\text{\bf ``regular case''}), \;\; \text{or}    \\     
 \liminf_{h\downarrow0}\frac{|U_j(s)-U_j(s-h)|}{\sqrt{h}} &\leq& 4 < \limsup_{h\downarrow0}\frac{|U_j(s)-U_j(s-h)|}{\sqrt{h}} \;\;({\text{\bf ``irregular  case''}}). \end{eqnarray*} 

Next we will see that the regular case is close to being a sufficient condition for getting welded hulls. 

\subsubsection{A sufficient condition for the regular case}

In the following, we denote by $\lip$ the set of all ``pointwise left $\frac{1}{2}$-H\"older continuous'' functions $U:[0,T]\to\R$, so that for every $t\in(0,T]$ there is a $c>0$ and an $\eps>0$ such that 
\begin{equation*}\label{left}|U(t)-U(s)|< c\sqrt{t-s} \quad \text{for all} \quad s\in[t-\eps,t].
\end{equation*}

An equivalent formulation is: For every $t\in (0,T],$ there exists $c>0$ such that 
\begin{equation*}
\limsup_{h\downarrow0}\frac{|U(t)-U(t-h)|}{\sqrt{h}}< c. 
\end{equation*}

Analyzing the proof of Theorem \ref{lmr} immediately yields that in the case $n=1$, $U_1$ generates welded hulls if we assume that  
\begin{equation}\label{limsup}
\limsup_{h\downarrow0}\frac{|U_1(t)-U_1(t-h)|}{\sqrt{h}}< 4 
\end{equation}
for every $t\in(0,T]$. By using techniques of \cite{Lind:2005}, we will prove the following statement which generalizes the sufficient condition (\ref{limsup}).

\begin{theorem}\label{sevslits}
Let $\{K_t\}_{t\in[0,T]}$ be a family of hulls generated by equation (\ref{more2}). Suppose that  \begin{equation}\label{hold}\limsup_{h\downarrow0}\frac{|U_j(t)-U_j(t-h)|}{\sqrt{h}}< 4\sqrt{\lambda_j(t)},\end{equation} for every $j=1,...,n$ and $t\in(0,T]$, then $\{K_t\}_{t\in[0,T]}$ is welded.
\end{theorem}
\begin{remark}
In particular, the ``pointwise left $\frac{1}{2}$-H\"older continuous'' condition (\ref{hold}) forces $K_T$ to fall into $n$ disjoint connected components.
\end{remark}

Again, we will give the proof only for $n=2$ in order to simplify notation. We begin with the following lemma.
\begin{lemma}\label{help}
 Let $\lambda>0,$ $0\leq\tau<2\sqrt{\lambda}$ and $h_n$ be the following sequence of functions
\begin{eqnarray*}
&&h_1:\R\to\R, \quad h_1(x)=x,\\
&&h_{n+1}:\{x\in\R\;|\; h_n(x)\not=0\}\to\R, \quad h_{n+1}(x)=x+\tau-\frac{4\lambda}{h_{n}(x)} \quad \text{for}\quad n\geq1. 
\end{eqnarray*}
Let $x_n$ denote the largest zero of $h_n.$ Then $(x_n)_n$ is an increasing sequence that converges to $4\sqrt{\lambda}-\tau.$ Furthermore, if  $h_n(c)\geq0$ for every $n\in\N,$ then $c\geq4\sqrt{\lambda}-\tau.$
\end{lemma} 
\begin{proof} By induction, it can be shown that $h_{n+1}$ maps $(x_n,+\infty)$ strictly monotonically onto $\R.$ Consequently, $(x_n)_n$ is an increasing sequence. We prove that it is bounded above by $4\sqrt{\lambda}-\tau$ by showing $h_n(4\sqrt{\lambda}-\tau)>2\sqrt{\lambda}$
for all $n\geq 1$ inductively:
First, $$h_1(4\sqrt{\lambda}-\tau)=4\sqrt{\lambda}-\tau>4\sqrt{\lambda}-2\sqrt{\lambda}=2\sqrt{\lambda},$$
and for $n\geq 1$ we have \begin{eqnarray*} h_{n+1}(4\sqrt{\lambda}-\tau)&=&4\sqrt{\lambda}-\tau+\tau-\frac{4\lambda}{h_{n}(4\sqrt{\lambda}-\tau)}\\&>&4\sqrt{\lambda}-\frac{4\lambda}{2\sqrt{\lambda}} = 2\sqrt{\lambda}.
\end{eqnarray*}

Hence $x_n$ converges to \begin{equation}\label{*}\tilde{x}\leq4\sqrt{\lambda}-\tau.\end{equation}
Obviously, we have $h_n(\tilde{x})>0$ for all $n.$ Now suppose $h_n(\tilde{x})\leq\sqrt{\lambda},$ then $h_{n+1}(\tilde{x})=\tilde{x}+\tau-\frac{4\lambda}{h_n(\tilde{x})}\leq\tilde{x}+\tau-4\sqrt{\lambda}\leq0,$ a contradiction.
 Hence $ h_n(\tilde{x})>\sqrt{\lambda} .$\\ 
Furthermore, we get from (\ref{*}) 
\begin{eqnarray*}h_n(\tilde{x})-h_{n+1}(\tilde{x}) &=&  h_n(\tilde{x})-\tilde{x}-\tau+\frac{4\lambda}{h_n(\tilde{x})} \geq h_n(\tilde{x})-4\sqrt{\lambda}+\frac{4\lambda}{h_n(\tilde{x})} \\ &=& \frac{h_n(\tilde{x})^2-4\sqrt{\lambda}h_n(\tilde{x})+4\lambda}{h_n(\tilde{x})}=
\frac{(h_n(\tilde{x})-2\sqrt{\lambda})^2}{h_n(\tilde{x})} \geq0.
\end{eqnarray*} It follows that the sequence $h_n(\tilde{x})$ is decreasing and bounded below by $\sqrt{\lambda}.$ It converges to $\tilde{h}$ with
$$\tilde{h}=\tilde{x}+\tau-\frac{4\lambda}{\tilde{h}}.$$
So $\tilde{h}=\frac{\tilde{x}+\tau\pm\sqrt{(\tilde{x}+\tau)^2-16\lambda}}{2}$ and hence $(\tilde{x}+\tau)^2\geq16\lambda.$ As $\tilde{x}$ must be positive, we conclude $\tilde{x}\geq4\sqrt{\lambda}-\tau$ and together with (\ref{*}) this implies $\tilde{x}=4\sqrt{\lambda}-\tau$.
\end{proof}
 A great advantage of J. Lind's proof in \cite{Lind:2005} is the fact that we don't have to work with equation (\ref{more2}) for arbitrary initial values in the upper half-plane, but we can concentrate on the real initial value problem 
\begin{equation}\label{2slit}
 x(t_0) = x_0\in\R\setminus\{U_1(t_0),U_2(t_0)\},  \qquad \dot{x}(t) = \frac{2\lambda_1(t)}{x(t)-U_1(t)} + \frac{2\lambda_2(t)}{x(t)-U_2(t)}, \qquad t\in[t_0,T].
\end{equation}

   \begin{lemma}\label{hitting}
     Let $j=1$ or $j=2$ and $U_j \in \lip$. Suppose that the solution $x(t)$ of (\ref{2slit}) with initial value $x_0 \in\R\setminus\{U_1(t_0),U_2(t_0)\}$ exists until $t=T$ and $x(T)=U_j(T).$ Then $$\limsup_{h\downarrow0}\frac{|U_j(T)-U_j(T-h)|}{\sqrt{h}}\geq 4\sqrt{\lambda_j(T)}.$$
   \end{lemma}
\begin{proof} We will assume that $j=2$ and $T=1$, and we start with the case $U_1(t_0)<x_0<U_2(t_0)$, so that for all $t<1$ the solution satisfies $U_1(t)<x(t)<U_2(t)$.\\
First we define $$\displaystyle \lambda_2^s:=\min_{t \in [s,1]} \lambda_2(t).$$
 Let $S\in[t_0,1)$ be so close to $1,$ that 
\begin{itemize}
 \item $x(t)-U_1(t) > \delta > \frac{1}{\lambda_2^S}(U_2(t)-x(t))$ for a $\delta>0$ and all $t \in [S, 1]$
(which can be achieved because $\frac{1}{\lambda_2^S}$ is bounded and
 $U_2(t)-x(t)$ goes to zero when $t\to1$) and
\item  $|U_2(1)-U_2(t)|\leq c\sqrt{1-t}$ for all $t \in [S, 1]$.
\end{itemize}
Next we define $\eps := \frac{2}{\delta}\sqrt{1-S}$ and we can assume (by passing on to a larger $S$ if necessary)
\begin{equation}\label{cond}
 \eps<2\sqrt{\lambda_2^S}.
\end{equation}

Now $x(t)$ is decreasing in $[S,1]$, since 
$$\dot{x}(t)= \frac{2\lambda_1(t)}{x(t)-U_1(t)} + \frac{2\lambda_2(t)}{x(t)-U_2(t)} < \frac{2}{\frac{1}{\lambda_2^S}(U_2(t)-x(t))} + \frac{2\lambda_2^S}{x(t)-U_2(t)} = 0.$$ 
 
We will now show by induction that 
\begin{equation*}
 U_2(t)-x(t) \leq h_n(c)\sqrt{1-t} \quad \text{for every} \quad n\in\N, \; t\in[S,1],
\end{equation*}
where $h_n$ is the function from Lemma \ref{help} with $\lambda=\lambda_2^S$ and $\tau=\eps.$\\
First we have $$U_2(t)-x(t)\leq U_2(t)-x(1)=U_2(t)-U_2(1) \leq c\sqrt{1-t}=h_1(c)\sqrt{1-t}.$$
Now assume the inequality holds for one $n\in\N.$ Then we have
\begin{equation*}
\dot{x}(t) \leq \frac{2}{\delta} + \frac{2\lambda_2^S}{x(t)-U_2(t)}
 \leq  \frac{2}{\delta} - \frac{2\lambda_2^S}{h_n(c)\sqrt{1-t}}.
  \end{equation*}  
Integrating yields  \begin{equation*}\label{betterest}
x(1)-x(t) \leq \frac{2}{\delta}(1-t)  - \frac{4\lambda_2^S}{h_n(c)}\sqrt{1-t}
\leq (\eps - \frac{4\lambda_2^S}{h_n(c)})\sqrt{1-t}.
  \end{equation*}                                
This implies $$U_2(t)-x(t)\leq U_2(t)-U_2(1)+(\eps - \frac{4\lambda_2^S}{h_n(c)})\sqrt{1-t} 
\leq (c+\eps - \frac{4\lambda_2^S}{h_n(c)})\sqrt{1-t}=h_{n+1}(c)\sqrt{1-t}.$$ 
As $U_2(t)-x(t)$ is always positive, we conclude that $h_n(c)\geq0$ for every $n\in\N$ and Lemma \ref{help} tells us that 
\begin{equation*}
 c \geq4 \sqrt{\lambda_2^S}- \eps.
\end{equation*}
Finally, sending $S$ to 1 yields
\begin{equation*}
 c \geq 4 \sqrt{\lambda_2^1} -0  = 4 \sqrt{\lambda_2(1)}.
\end{equation*}
The case $x_0>U_2(t_0)$ can be treated in the same way and in the case $x_0<U_1(t_0),$ the solution $x(t)$ cannot fulfill $x(1)=U_2(1).$
\end{proof}

Consequently, if we have $U_1,U_2\in\lip$ and $$\limsup_{h\downarrow0}\frac{|U_j(t)-U_j(t-h)|}{\sqrt{h}}< 4\sqrt{\lambda_j(t)}\;\; \text{for all}\;\; t\in(t_0,T], \quad j=1,2,$$ 
then the solution for any $x_0 \in\R\setminus\{U_1(t_0),U_2(t_0)\}$ will exist up to time $t=T$ and $x(T)$ cannot equal $U_1(T)$ or $U_2(T)$. In fact, there are even fixed intervals around $U_1(T)$ and $U_2(T)$ which cannot be reached by $x(T)$ for any initial value $x_0.$

\begin{lemma}\label{nodepend}
Let $U_1,U_2\in\lip$ with $$\limsup_{h\downarrow0}\frac{|U_j(t)-U_j(t-h)|}{\sqrt{h}}< 4\sqrt{\lambda_j(t)}$$ for every $t\in(t_0,T]$ and $j=1,2$. Suppose that $x(t)$ is a solution of (\ref{2slit}) with $x_0\in\R\setminus\{U_1(t_0),U_2(t_0)\}.$ Then there exists $\eps>0$ such that $$|x(T)-U_1(T)|>\eps \quad \text{and} \quad |x(T)-U_2(T)|>\eps$$ for every $x_0 \in\R\setminus\{U_1(t_0),U_2(t_0)\}.$
\end{lemma}
\begin{proof} In order to simplify notation, we let $T=1$ in the proof.\\
We prove the statement by contradiction, hence we assume that for every\\ $\eps>0$ there is an $x_0^\eps \in\R\setminus\{U_1(t_0),U_2(t_0)\}$ such that\\ $$|x(1)-U_1(1)|\leq\eps \quad \text{or} \quad |x(1)-U_2(1)|\leq\eps.$$
Without loss of generality we may assume that $U_1(t_0)<x_0^\eps<U_2(t_0)$ and\\ $U_2(1)-x(1)\leq\eps.$\\
Now there is an $\epsilon>0$ such that for every $\eps\in(0,\epsilon)$, the solution to the corresponding initial value $x_0^\eps$ is decreasing in an interval $[S_0,1]$ and $x(t)-U_1(t)>\delta$ for all $t\in[S_0,1]$ and a $\delta>0$.\\
From now on we require $\eps < \epsilon$ and furthermore we assume that $S\in[t_0,1)$ is so close to $1$ that
 \begin{itemize}
\item $S\geq S_0$,
\item $\tau:=\frac{2}{\delta}\sqrt{1-S}<2\sqrt{\lambda_2^S}$ and
\item $|U_2(1)-U_2(t)|\leq c\sqrt{1-t}\quad \text{with}\;  c<4\sqrt{\lambda_2^S}-\tau$ for all $t \in [S, 1]$.
\end{itemize}

Again, we denote by $h_n$ the sequence
\begin{eqnarray*}
h_1(x)&=& x,\\
 h_{n+1}(x)&=&x+\tau-\frac{4\lambda_2^S}{h_{n}(x)} \quad \text{for}\quad n\geq1.
\end{eqnarray*}
Lemma (\ref{help}) implies that there is an $N \in \N$ such that \begin{equation}\label{theK}h_N(c)<0.\end{equation}
We take the \textbf{smallest} such $N$ and assume that $h_n(c)>0$ for all $n<N.$\\ (If there is a $n$ with $h_n(c)=0$, one can pass on to a slightly greater $c$.)\\
Next, define $e_n$ by
\begin{eqnarray*}
e_1&=& \eps,\\
e_{n+1} &=& \eps + \frac{4\lambda_2^Se_{n}}{(h_n(c))^2}\log\left(1+\frac{h_n(c)}{e_n}\right) \quad \text{for}\quad 1\leq n < N. 
\end{eqnarray*} 
Inductively one can easily show that for every $n\leq N$ we have $e_n>0$ and
 \begin{equation}\label{limit}
 \lim_{\eps\to0}e_n = 0.
 \end{equation}
Now we prove by induction that 
\begin{equation*}
x(1)-x(t) \leq e_n -\eps + (h_n(c)-c)\sqrt{1-t} \quad \text{for all} \quad t \in[S,1]
\quad \text{and} \quad n\leq N .
\end{equation*}
The case $n=1$ states $x(1)-x(t)\leq0$ which is true because $x(t)$ is decreasing. Next assume that the statement holds for a $n<N.$ Then
 \begin{equation*}
-x(t) \leq -x(1) - \eps+ e_n + (h_n(c)-c)\sqrt{1-t} \quad \text{for all} \quad t \in[S,1].
\end{equation*}
Consequently
 \begin{eqnarray*}
U_2(t) - x(t) &\leq& U_2(t)-U_2(1)+U_2(1)-x(1)-\eps+e_n-(c-h_n(c))\sqrt{1-t} \\
&\leq& e_n+h_n(c)\sqrt{1-t}\quad \text{for all} \quad t \in[S,1],
\end{eqnarray*}
which implies
 \begin{eqnarray*}
\dot{x}(t) = \frac{2\lambda_1(t)}{x(t)-U_1(t)} + \frac{2\lambda_2(t)}{x(t)-U_2(t)}\leq \frac{2}{\delta} -  \frac{2\lambda_2^S}{e_n+h_n(c)\sqrt{1-t}}
\end{eqnarray*}
and integrating gives
 \begin{eqnarray*}
x(1)-x(t) &\leq& \frac{2}{\delta}(1-t)-\frac{4\lambda_2^S}{h_n(c)}\sqrt{1-t}+\frac{4\lambda_2^Se_n}{(h_n(c))^2}\log(1+\frac{h_n(c)}{e_n}\sqrt{1-t})\\
&=&\left(\frac{2}{\delta}\sqrt{1-t}-\frac{4\lambda_2^S}{h_n(c)}\right)\sqrt{1-t}+\frac{4\lambda_2^Se_n}{(h_n(c))^2}\log(1+\frac{h_n(c)}{e_n}\sqrt{1-t}) \\
&\leq&(c+\tau-\frac{4\lambda_2^S}{h_n(c)}-c)\sqrt{1-t}+\frac{4\lambda_2^Se_n}{(h_n(c))^2}\log(1+\frac{h_n(c)}{e_n}\sqrt{1-t}) \\
&=&(h_{n+1}(c)-c)\sqrt{1-t}+\eps+\frac{4\lambda_2^Se_n}{(h_n(c))^2}\log(1+\frac{h_n(c)}{e_n}\sqrt{1-t})- \eps.
\end{eqnarray*}
As $h_n(c)>0$ and $e_n>0$ we conclude
 \begin{eqnarray*}
x(1)-x(t) &\leq& (h_{n+1}(c)-c)\sqrt{1-t}+\eps+\frac{4\lambda_2^Se_n}{(h_n(c))^2}\log(1+\frac{h_n(c)}{e_n})-\eps \\ &=& (h_{n+1}(c)-c)\sqrt{1-t}+e_{n+1}-\eps.
\end{eqnarray*}
For $n=N$ we get
 \begin{equation*}x(1)-x(t)\leq e_N -\eps + (h_N(c)-c)\sqrt{1-t}. \end{equation*} On the other hand,  \begin{equation*}x(1)-x(t)=x(1)-U_2(1)+U_2(1)-x(t)\geq-\eps+U_2(1)-U_2(t)\geq-\eps-c\sqrt{1-t}. \end{equation*}
Thus $$h_N(c)\sqrt{1-S}+e_N \geq 0.$$ Now we can send $\eps \to 0$ and get $h_N(c)\geq0$ by (\ref{limit}), a contradiction to (\ref{theK}). 
\end{proof}

\begin{proof}[Proof of Theorem \ref{sevslits}] The statement follows directly from combining Proposition \ref{nodepend} with Proposition \ref{Fritz}. 
\end{proof}

\subsubsection{A sufficient condition for the irregular case}

For driving functions that are irregular in some points, it is somehow harder to find out whether the generated hulls are welded or not. Here we derive a sufficient condition for a very special case. This case will appear later in the proof of Theorem \ref{proop}. In the following, we let $n=1,$ though it is not difficult to generalize the statement to the general multiple-slit case.

\begin{lemma}\label{stran}
Let $U:[0,1]\to\R$ be continuous with $U(1)=0$ and let $\{K_t\}_{t\in[0,1]}$ be the hulls generated by the one-slit equation (\ref{ivp}). Suppose that there are two increasing sequences $s_n,t_n$ of positive numbers with $s_n,t_n\to1, 
$ such that for $$\overline{M}_n:=\max_{s_n\leq t\leq1}\{U(t)\}\quad \text{and} \quad \underline{M}_n:=\min_{t_n\leq t\leq1}\{U(t)\}$$ the two inequalities 
$$4(1-s_n)+U(s_n)^2-2U(s_n)\overline{M}_n>0 \quad \text{and} \quad 4(1-t_n)+U(t_n)^2-2U(t_n)\underline{M}_n>0$$
 hold for all $n\in\N.$ If $\{K_s\}_{s\in[0,t]}$ is welded for all $t\in(0,1)$, then so is $\{K_s\}_{s\in[0,1]}$.
\end{lemma}
\begin{proof}
Let $\tau\in[0,1)$ and $x_0\in\R\setminus\{U(\tau)\}.$
By Proposition \ref{Fritz} we know that the solution $x(t)$ of the initial value problem
 $$ x(\tau)=x_0, \qquad \dot{x}(t)=\frac{2}{x(t)-U(t)}$$
 exists until $t=1$ and we have to show that there is a positive lower bound for $|x(1)-U(1)|$ which is independent of $x_0$.

 Assume that $x_0<U(\tau)$. Then $x(t)$ is decreasing and we have $x(s_m)<U(s_m)$ with $m:=\min\{n\in\N \; |\; s_n\geq \tau\}.$ The initial value problem 
$$ \dot{y}(t)=\frac{2}{y(t)-\overline{M}_m}, \qquad  y(s_m)=x(s_m),$$
has the solution $y(t)=\overline{M}_m-\sqrt{(\overline{M_m}-x(s_m))^2+4(t-s_m)}$. Now we have
 $$\dot{x}(t)\leq\frac{2}{x(t)-\overline{M}_m} \quad \text {for all} \; t\in [s_m,1)$$ and $x(s_m)=y(s_m)$. Consequently,
 \begin{eqnarray*}
  &&x(1)\leq y(1)=\overline{M}_m-\sqrt{(\overline{M}_m-x(s_m))^2+4(1-s_m)}\\
&&<\underbrace{\overline{M}_m-\sqrt{(\overline{M}_m-U(s_m))^2+4(1-s_m)}}_{=:L_1}<\overline{M}_m-\sqrt{\overline{M}_m^2}=0.
 \end{eqnarray*}
The case $x_0>U(\tau)$ can be treated in the same way and gives a bound $L_2>0$ for $x(1)=x(1)-U(1).$
Thus, the condition in Proposition \ref{Fritz} b) is satisfied for $\eps=\min\{-L_1,L_2\}$ and it follows that $\{K_s\}_{s\in[0,1]}$ is welded.
\end{proof}
\begin{corollary}\label{cori}
 If $\{K_s\}_{s\in[0,t]}$ is welded for all $t\in(0,1)$ and there are two increasing sequences $s_n,t_n$ of positive numbers with $s_n, t_n\to1,$ and $U(s_n)\leq U(1) \leq U(t_n)$ for all $n,$ then $\{K_s\}_{s\in[0,1]}$ is welded, too.
\end{corollary}
\begin{proof}
Without loss of generality we can assume $U(1)=0.$ We can apply Lemma \ref{stran} as  $$4(1-s_n)+U(s_n)^2-2U(s_n)\overline{M}_n>-2U(s_n)\overline{M}_n\geq0 \quad \text{and}$$ 
$$ 4(1-t_n)+U(t_n)^2-2U(t_n)\underline{M}_n>-2U(t_n)\underline{M}_n\geq0.$$
\end{proof}

\subsection{Quasisymmetric weldings}

Suppose that $\{K_t\}_{t\in[0,T]}$ is a family of hulls generated by the multiple-slit equation and that it is welded. In this case, $K_T$ has $n$ connected components $C_1,...,C_n.$ Denote by $I_j$ the cluster set of $\overline{C_j}$ with respect to $g_{K_T}.$ Then $U_j(T)\in I_j$ and there exist welding homeomorphisms $h_j:I_j\to I_j,$ $j=1,...,n,$ in the sense of Remark \ref{Schorsch}. \\
We call $h_j$ \textit{quasisymmetric} provided that there is a constant $M\geq 1$ such that
$$\frac1{M}\leq \frac{x-U_j(T)}{U_j(T)-h_j(x)} \leq M$$
for all $x>U_j(T), x\in I_j,$ and 
$$\frac1{M}\leq \frac{h_j(x)-h_j(y)}{h_j(y)-h_j(z)} \leq M$$ for all $x,y,z\in I_j$ with $U_j(T)\leq x<y<z$ and $y-x=z-y$.\\
If we think of the ``fundamental Theorem of conformal welding'', then one could hope to find an affirmative answer to the following question: 
\begin{question}
 Is $C_j$ a quasislit provided that $h_j$ is quasisymmetric?
\end{question}

However, we have the problem that we don't know whether $C_j$ is a slit or not, i.e. we would need a uniqueness result for generalized conformal weldings that are quasisymmetric.\\

In \cite{Lind:2005}, Lemma 6, and \cite{MarshallRohde:2005}, Lemma 2.2 respectively, it was shown for $n=1$ that $C_1=K_T$ is a quasislit if and only if $h_1$ is quasisymmetric and $C_1$ is a slit. This result readily generalizes to the multiple-slit case.

\begin{lemma}
 $C_j$ is a quasislit if and only if $C_j$ is a slit and $h_j$ is quasisymmetric. 
\end{lemma}
 \begin{proof}
  Let $A=g_{C_j}\left(\cup_{k\not=j}C_k\right).$ By applying the function $g_{A}$ we can simply reduce the statement to the case $n=1:$ 
The hull $C_j$ can be generated by the one-slit equation. Denote by $h:[a,b]\to[a,b]$ its (one-slit) welding homeomorphism. Now we have the following connection between $h$ and $h_j:$ 
$$h(y) = g_A^{-1}(h_j(g_A(y)))\quad \text{for every} \quad y\in [a,b].$$
As $g_A$ is conformal in a neighborhood of $[a,b]$, we conclude that $h$ is quasisymmetric if and only if $h_j$ is quasisymmetric. The statement follows now from the one-slit case, i.e. from Lemma 6 in \cite{Lind:2005}.
 \end{proof}

\subsection{The regular case and quasislits}

Now we generalize Theorem \ref{lmr} to the multiple-slit case and we start with the following statement.

\begin{lemma}\label{Michael}
Denote by $\{K_t\}_{t\in[0,T]}$ the generated hulls by equation (\ref{more2}). Assume that $K_T$ has $n$ connected components $C^1_T,...,C^n_T$ with $U_j(0)\in\overline{C^j_T}.$ \\
Assume there exists one $j\in\{1,...,n\}$ such that for every $t\in[0,T]$ there exists an $\eps>0$  such that
\begin{equation}\label{Olga} \sup_{\substack{r,s\in[0,T]\\ r\not=s\\|r-t|,|s-t|<\eps}}\frac{|U_j(r)-U_j(s)|}{\sqrt{|r-s|}}<4\sqrt{\lambda_j(t)}, \end{equation}
then $C^j_T$ is a quasislit.\\
Conversely, if $C^j_T$ is a quasislit, then $U_j \in \Lip.$
\end{lemma}
\begin{proof}
First, assume that condition (\ref{Olga}) holds. Also $K_t$ consists of $n$ disjoint connected components $C^1_t,...,C^n_t $ for every $t\in(0,T],$ where we assume that $C^j_t$ corresponds to $U_j.$ Now assume that $C^j_T$ is not a quasislit. Then, $C^j_T$ does not approach $\R$ nontangentially or there exists a point $p \in \overline{C^j_T}$ such that any neighborhood of $p$ in $\overline{C^j_T}$ is not a quasiarc. Without loss of generality we may assume that the quasislit property is violated in any neighborhood of $\overline{C^j_T}\cap \R=U_j(0),$ i.e. we assume that any subhull of $C^j_T$ is not a quasislit. Let $x(t)=\hcap(C^j_t).$ Let $h_t=g_{C^j_t}$ and define $H_t$ by $g_t=H_t\circ h_t.$ The mappings $h_t$ satisfy a one-slit Loewner equation
$$ \dot{h}_t=\frac{2\dot{x}(t)}{h_t-V(t)}, $$
where $x(t)$ is continuously differentiable with $\dot{x}(0)=2\lambda_j(0),$ see Lemma \ref{Miriam}, and $V$ is a continuous driving function. This function $V$ is related to $U_j$ via 
$$V(t)=H_t^{-1}(U_j(t)).$$
For all $t$ small enough, say $t<t',$ the function $H_t^{-1}$ can be extended to the lower half-plane by reflection and to a fixed neighborhood $\mathcal{N}$ of $\{U_j(t)\with t\in[0,t']\}$ for all  $t< t'.$
As the function $[0,t')\times \mathcal{N} \ni(t,z)\mapsto F(t,z):=H_t^{-1}(z)$ is continuously differentiable, we have for all $r,s\in[0,t')$:
\begin{eqnarray*} |V(r)-V(s)|&\leq& |\frac{\partial F}{\partial t}(r,U(r))| \cdot |(r-s)| + |\frac{\partial F}{\partial z}(r,U(r))| \cdot  |(U_j(r)-U_j(s))| \\ &+&\Landauo(|r-s|+|U_j(r)-U_j(s)|)
\leq |\frac{\partial F}{\partial z}(r,U(r))| \cdot  |(U_j(r)-U_j(s))| + \Landauo(\sqrt{|r-s|})
\end{eqnarray*}
for $s\to r.$
 Moreover, $\frac{\partial F}{\partial z}(0,z)=1$ for all $z\in \mathcal{N}.$ Thus, for every $\eps>0$ we can find $\delta>0$ such that for all $r,s\in[0,\delta],$ $r\not=s,$ we have
\begin{itemize}
\item  $|\frac{\partial F}{\partial z}(r,z)|\leq 1+\eps$ for all $z\in \mathcal{N}$,
\item  $\frac{|U_j(r)-U_j(s)|}{\sqrt{|r-s|}}<4\sqrt{\lambda_j(0)}.$
\end{itemize}
Thus, for $\eps$ small enough, $V$ satisfies
$$  \frac{|V(r)-V(s)|}{\sqrt{|r-s|}}<4\sqrt{\lambda_j(0)} $$
for all $r,s\in[0,\delta],$ $r\not=s.$ By Theorem \ref{lmr} (where we have to change the time $t\to x(t)$) we conclude that $C^j_\delta$ is a quasislit for all $\delta$ small enough, a contradiction.\\
Conversely, if $C^j_T$ is a quasislit, then we can show in a similar way that $U_j \in \Lip.$
\end{proof}

\begin{theorem}\label{Michael2}
Let $U_1,\ldots,U_n\in \Lip$ and assume that for every $j\in\{1,\ldots,n\}$ and every $t\in[0,T]$ there exists an $\eps>0$  such that
\begin{equation}\label{Babutza} \sup_{\substack{r,s\in[0,T]\\ r\not=s\\|r-t|,|s-t|<\eps}}\frac{|U_j(r)-U_j(s)|}{\sqrt{|r-s|}}<4\sqrt{\lambda_j(t)}. \end{equation}
Denote by $\{K_t\}_{t\in[0,T]}$ the generated hulls by equation (\ref{more2}). Then $K_T$ consists of $n$ disjoint connected components $C^1,...,C^n,$ and each $C^j$ is a quasislit.
\end{theorem}
\begin{proof} All driving functions satisfy the condition from Theorem \ref{sevslits}. Hence we know that $K_T$ consists of $n$ disjoint connected components $C^1,...,C^n.$ Every $C^j$ is a quasislit by Lemma \ref{Michael}.
\end{proof}

\begin{remark}
The condition in the last result can be weakened in the following way: If (\ref{Babutza}) is satisfied for all $j\in\{1,...,n\}$ and only for all $t\in(0,T],$ then $K_T$ consists of $n$ connected components $C^1,...,C^n$ (by Theorem \ref{sevslits}) and every $C^j$ is a slit. This can be seen as follows: Let $g_t$ be the mapping from equation (\ref{more2}). By Theorem \ref{Michael2} we know that the hull $g_\eps(K_T\setminus K_\eps)$ consists of $n$ quasislits for every $\eps\in (0,T).$ Thus, $C^j \setminus A$ is a quasiarc for any nonempty subhull $A\subsetneq C^j$. This implies that $C^j$ can be parameterized by a continuous, injective function $\gamma:(0,1]\to\Ha,$ where $\gamma(1)$ is the tip of $C^j.$ The local growth property of $C^j$implies that $\gamma(t)$ can be extended continuously to $t=0$ and $\gamma(0)=U_j(0).$ Consequently, $C^j$ is a slit (but not a quasislit in general).
\end{remark}

\subsection{An example for the irregular case}\label{Ingwer}
Recall question \ref{4ne}: Is the local H\"older constant $4$ in Theorem \ref{lmr} also necessary for generating quasislits? The answer is ``no'': For any $s\in[0,T)$, the ``right'' pointwise H\"older norm, i.e. the value $$\limsup_{h\downarrow0}\frac{|U(s+h)-U(s)|}{\sqrt{h}}$$ can get arbitrarily large, as the driving function $U(t)=c\sqrt{t}$ shows; it generates a straight line for every $c\in \R$, see Example \ref{Maxime}.

Next we show that the ``left'' pointwise H\"older norm can also become arbitrarily large within the space of all driving functions generating quasislits. To this end, we will use Corollary \ref{cori} to construct a driving function that is irregular at one point and generates a quasislit.
\begin{theorem}\label{proop}
 For every $C>0,$ there exists a driving function $U:[0,1]\to\R$ that generates a quasislit and satisfies
$$ \limsup_{h\downarrow0}\frac{|U(1)-U(1-h)|}{\sqrt{h}}=C. $$
\end{theorem}

\begin{proof}
Let $C>0.$ First we construct the driving function $U$, which is shown in Figure \ref{Fi35} for the case $C=5$. 

We set $U(r_n):=0$ with $r_n:=1-\frac1{2^n}$ for all $n\geq 0.$ The mean value of $r_n$ and $r_{n+1}$ is equal to $w_n:=1-\frac3{2^{n+2}}$ and here we define $$U(w_n) := C\sqrt{\frac3{2^{n+2}}}\quad \text{for} \quad n\geq0.$$ Now we define $U(t)$ for $t\in[0,1)$ by linear interpolation, so that 
\begin{eqnarray*}
 U(t)&=& C\sqrt{3\cdot2^{n+2}}\cdot (t-r_n) \quad \text{for} \quad t\in[r_n, w_n]\quad \text{and}\\
U(t)&=& C\sqrt{3\cdot2^{n+2}}\cdot (r_{n+1}-t) \quad \text{for} \quad t\in[w_n,r_{n+1}].
\end{eqnarray*}
  By defining $U(1):=0$ we now have a continuous driving function and 
$$ \limsup_{h\downarrow0}\frac{|U(1)-U(1-h)|}{\sqrt{h}}=
\lim_{n\to\infty}\frac{|U(1)-U(w_n)|}{\sqrt{1-w_n}} = 
\lim_{n\to\infty}\frac{C\sqrt{3/2^{n+2}}}{\sqrt{3/2^{n+2}}} = C.$$
At each $0\leq t <1,$ the hull $K_t$ produced by this function will be a quasislit according to Theorem \ref{lmr}. Thus, we have to show that also $K_1$ is a slit and that this slit is a quasiarc.\\

First, we know that $\{K_t\}_{t\in[0,1]}$ is welded: This follows directly from  Corollary \ref{cori} by setting $s_n:=t_n:=r_n$.

If we scale our hull by $\frac1{\sqrt{2}}$, we end up with the new driving function $\tilde{U}:[0,1/2]\to\R,$ $\tilde{U}(t)=\frac1{\sqrt{2}}U(2t)$. However, this is again $U(t)$, confined to the interval $[1/2,1],$ i.e. $\tilde{U}(t)=U(t+1/2).$ Geometrically, this means that $g_{1/2}(\overline{K_1\setminus K_{1/2}})$ is just the same as $\frac1{\sqrt{2}}K_1,$ the original hull scaled by $\frac1{\sqrt{2}}.$\\

If $f:=g_{1/2}^{-1}$, and $S_n:=\overline{K_{1-1/2^n}\setminus K_{1-1/2^{n-1}}}, n\geq1,$ then  we have  $$S_{n+1}=f\left(\frac1{\sqrt{2}}S_n\right).$$
As the function $z\mapsto f(\frac{1}{\sqrt{2}}z)=:I(z)$ is not an automorphism of $\Ha$, the Denjoy--Wolff Theorem implies that the iterates $I^n=(I\circ\ldots\circ I)$ converge uniformly on $S_1$ to a point $S_\infty \in \overline{\Ha}\cup\{\infty\}.$ $S_\infty=\infty$ is not possible as the hull $K_1$ is a compact set and the case $S_\infty\in \R$ would imply that $K_1$ is not welded. Consequently $S_\infty\in \Ha$ and $K_1=\displaystyle \bigcup_{n\geq 1} S_n \cup \{S_\infty\}$ is a simple curve whose tip is $S_\infty$.\\

Now we show that this curve is a quasiarc.

For this, we will use the metric characterization of quasiarcs which says that $K_1$ is a quasiarc if and only if 
$$ \sup_{
\substack{x,y \in K_1 \\ x\not=y}} \frac{\diam(x,y)}{|x-y|} < \infty,$$ 
where we denote by $\diam(x,y)$ the diameter of the subcurve of $K_1$ joining $x$ and $y.$

For $m\in\N\cup\{0\}$ we define $F_m:=\bigcup_{k\geq m}S_{k}\cup\{S_\infty\}$. As $K_t$ is a quasislit for every $t\in(0,1),$ it suffices to show that 
\begin{equation}\label{show}
 \sup_{
\substack{x,y \in F_m \\ x\not=y}} \frac{\diam(x,y)}{|x-y|} < \infty \quad \text{for one} \; m \in\N.
\end{equation}

The set $S_n$ contracts to $S_\infty$ when $n\to\infty,$ in particular $\diam(S_n)\to0.$ 

As $I$ is conformal in $B(S_\infty, \eps)$ for $\eps>0$ small enough, there is an $N=N(\eps)\in\N$, such that 
$S_n\subset B(S_\infty, \eps)$ for all $n\geq N$. 

$S_\infty$ is a fixpoint of $I(z)$ and so $|I'(S_\infty)|<1.$ Otherwise, $I$ would be an automorphism of $\Ha.$\\

Now, for $x\in S_{N+n}$, $n\geq 0,$ we have $I(x)\in S_{N+n+1}$ and
 \begin{eqnarray*}
 &&|I(x)-S_\infty|=|I(x)-I(S_\infty)|=|I'(S_\infty)(x-S_\infty)+\LandauO(|x-S_\infty|^2)|\\&=&
|I'(S_\infty)+\LandauO(|x-S_\infty|)|\cdot|x-S_\infty|=|I'(S_\infty)||1+\LandauO(|\eps|)|\cdot|x-S_\infty|.
\end{eqnarray*} Consequently we can pass on to a smaller $\eps$ (and larger $N$) such that $\dist(S_\infty, S_{N+n+1})\leq c\cdot \dist(S_\infty, S_{N+n})$ with $c<1$ and for all $n\geq0$. Hence 
$S_{N+n}\subset B(S_\infty,c^n\eps).$\\

Furthermore, for $x,y\in F_{N+n}$ we have
\begin{eqnarray*}
 |I(x)-I(y)|= |I'(x)+\LandauO(c^n\eps)|\cdot|x-y|=|I'(S_\infty)+\LandauO(c^n\eps)|\cdot|x-y|.
\end{eqnarray*}

Hence there exist positive constants $a_1,a_2$ with $1-a_2c^n>0$ such that
\begin{equation}\label{diam}(1-a_2c^n)|I'(S_\infty)||x-y|\leq|I(x)-I(y)|\leq(1+a_1c^n)|I'(S_\infty)||x-y|.\end{equation} Thus \begin{equation}\label{diam2}\diam(I(x),I(y))\leq (1+a_1c^n)|I'(S_\infty)|\diam(x,y).\end{equation}
 
Now we will show (\ref{show}) for $m=N.$ Let $x, y \in F_N$ with $x\not=y.$ We assume that $\diam(x,S_\infty)\geq \diam(y,S_\infty).$ In particular, $x\not=S_\infty$ and thus there is a $k\geq 0$ and an $\hat{x}\in S_N$ such that $x=I^k(\hat{x}).$ Let $\hat{y}\in F_N$ be defined by $y=I^k(\hat{y}).$ 
First note that
$$ \sup_{
\substack{a\in S_N, b\in F_N \\ a\not=b}} \frac{\diam(a,b)}{|a-b|}=:E<\infty,
$$
for $K_t$ is a quasislit for every $t\in(0,1).$ Now we get with (\ref{diam}) and (\ref{diam2}):

\begin{eqnarray*}
 \frac{\diam(x,y)}{|x-y|} &=&  \frac{\diam(I^k(\hat{x}),I^k(\hat{y}))}{|I^k(\hat{x})-I^k(\hat{y})|}\leq \frac{(1+a_1c^{k-1})|I'(S_\infty)|\diam(I^{k-1}(\hat{x}),I^{k-1}(\hat{y}))}{(1-a_2c^{k-1})|I'(S_\infty)||I^{k-1}(\hat{x})-I^{k-1}(\hat{y})|} =\\
&=& \frac{(1+a_1c^{k-1})}{(1-a_2c^{k-1})}\cdot \frac{\diam(I^{k-1}(\hat{x}),I^{k-1}(\hat{y}))}{|I^{k-1}(\hat{x})-I^{k-1}(\hat{y})|} \leq ... \leq \\
&\leq& \prod \limits_{j=0}^{k-1} \frac{(1+a_1c^j)}{(1-a_2c^j)}\cdot \frac{\diam(\hat{x},\hat{y})}{|\hat{x}-\hat{y}|}  \leq
 \prod \limits_{j=0}^{k-1} \frac{(1+a_1c^j)}{(1-a_2c^j)} \cdot E \leq   E \prod  \limits_{j=0}^{\infty}\frac{1+a_1c^j}{1-a_2c^j} =\\ &=&
 E  \prod \limits_{j=0}^{\infty}(1+a_1c^j) / \prod \limits_{j=0}^{\infty}(1-a_2c^j)   <\infty.
\end{eqnarray*}

The two Pochhammer products converge because $|c|<1.$ Consequently, $K_1$ is a quasislit.
\end{proof}

  \begin{figure}[h]
    \centering
\includegraphics[width=170mm]{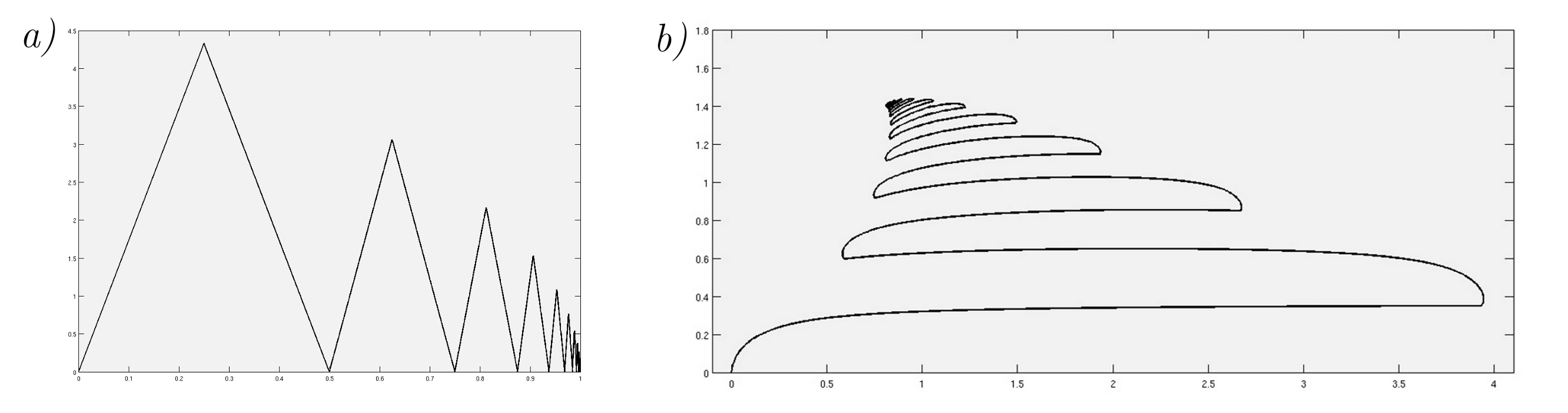}
 \caption{The driving function $U$ from Theorem \ref{proop} with $C=5$ (left) and the generated quasislit (right).}\label{Fi35}
 \end{figure}

\begin{remark}
 The argument that $U$ from the proof of Theorem \ref{proop} generates a slit holds for a more general case:

 Let $U:[0,1]\to\R$ be continuous with $U(1)=0.$ Call such a function $d$--self--similar with $0<d<1$ if $V(t):=U(1-t)$ satisfies $$d\cdot V(t/d^2)=V(t) \quad \text{for all}\quad 0<t\leq d^2.$$ Every $d$--self--similar function can be constructed by defining $V(1)$ arbitrarily, putting $V(d^2)=d\cdot V(1)$ and then defining $V(t)$ for $d^2<t<1$ such that $V$ is continuous in $[d^2,1].$ Then, $V$ is uniquely determined for all $0\leq t \leq 1.$ Now we have:\\ If $U$ is $d$--self--similar such that it produces a slit for all $0\leq t<1$ and $K_1$ is welded, then $K_1$ is a slit, too.
\end{remark}

\begin{example}
 There exists a driving function $V:[0,1]\to\R$ that is ''irregular`` at infinitely many points and generates a quasislit:\\

 Let $U$ be the driving function from the proof of Theorem \ref{proop}. We construct  $V:[0,1]\to\R$ by sticking pieces of $U$ appropriately together. For $n\geq0$ let
$$ V(t):=  U((t-(1-1/2^n))\cdot2^{n+1}) / \sqrt{2^n}
  \quad\text{for}\quad t\in[1-1/2^n, 1-1/2^{n+1}], $$ 
and $V(1):=0.$ Then $V$ is ''irregular`` at $1-1/2^n$ for all $n\geq1$ and it produces a quasislit: The hull generated at $t=1/2$ is a quasislit by Theorem \ref{proop}. Now one can repeat the proof of Theorem \ref{proop} to show that the whole hull is a quasislit, too. \hfill $\bigstar$
\end{example}

Theorem \ref{proop} together with Theorem \ref{lmr} suggests the following question:
 Does $U$ generate a quasislit if $U$ generates a slit and $U\in \Lip$?

The answer is no: There are $\Lip$-driving functions that generate slits with positive area. These slits cannot be quasislits as they are not uniquely determined by their welding homeomorphisms, see Corollary 1.4 in \cite{spacefill}.
\newpage
\section{Approach to \texorpdfstring{$\R$}{R}}\label{Approach}

Recall that a quasislit is a quasiarc in $\Ha$ that approaches $\R$ nontangentially. In this section we take a closer look at the way how a hull generated by the multiple-slit equation approaches the real axis. First we introduce two notions for a family $\{K_t\}_{t\in[0,T]}$ generated by the chordal multiple-slit equation (\ref{more2}).\\

Let $\varphi\in(0,\pi).$ We say that $K_t$ \emph{approaches $\R$ at $x\in\R$ in $\varphi$-direction} if for every $\eps>0$ there is a $t_0>0$ such that the connected component of $K_{t_0}$ having $U_j(0)$ as a boundary point is contained in the set $\{z\in\Ha\with\varphi-\eps<\arg(z-U_j(0))<\varphi+\eps\}.$\\

A further definition extends the ``line approach'' to a ``sector approach'':\\
$K_t$ \emph{approaches $\R$ at $U_j(0)$ in a sector} (or \emph{nontangentially}) if there exist angles $\alpha, \beta\in(0,\pi)$ and a $t_0>0$ such that the connected component of $K_{t_0}$ near $U_j(0)$ is contained in $$\{z\in\Ha\with \alpha<\arg(z-U_j(0))<\beta\}.$$

\subsection{Sector approach}

We start with a necessary condition for the sector approach of a family $\{K_t\}_t$ generated by equation (\ref{more2}).

 \begin{proposition}\label{sec}
  If $K_t$ approaches $\R$ at $U_j(0)$ in a sector, then $$\limsup_{h\downarrow0}\frac{|U_j(h)-U_j(0)|}{\sqrt{h}}<\infty. $$
\end{proposition}
\begin{proof} 
By translation we can assume $U_j(0)=0.$ Let $t_0>0$ be small enough such that the connected component $C_t$ of $K_{t}$ near $U_j(0)$ satisfies $C_t\subset S:= \{z\in\Ha\with \alpha<\arg(z)<\pi-\alpha\}$ for an $\alpha\in (0,\pi/2)$ and all $\in[0,t_0].$\\
 For $t\in(0,t_0]$, let $G_t=t^{-1/2}C_t$ and $H_t=t^{-1/2}K_t$. The monotonicity of $\hcap$ and the scaling property imply \begin{equation*}
\hcap(G_t)<\hcap(H_t)=t^{-1}\hcap(K_t)=2.
\end{equation*}
Then we also have $G_t\subset S$ for all $t\in(0,t_0].$
  From Lemma \ref{imestim} it follows that $$\sup_{t\in(0,t_0)}\max_{z\in G_t}\Im(z)<2.$$ Let $\Delta$ be the following (Euclidean) triangle with height $2$:
 \begin{figure}[h] \centering
\includegraphics[width=65mm]{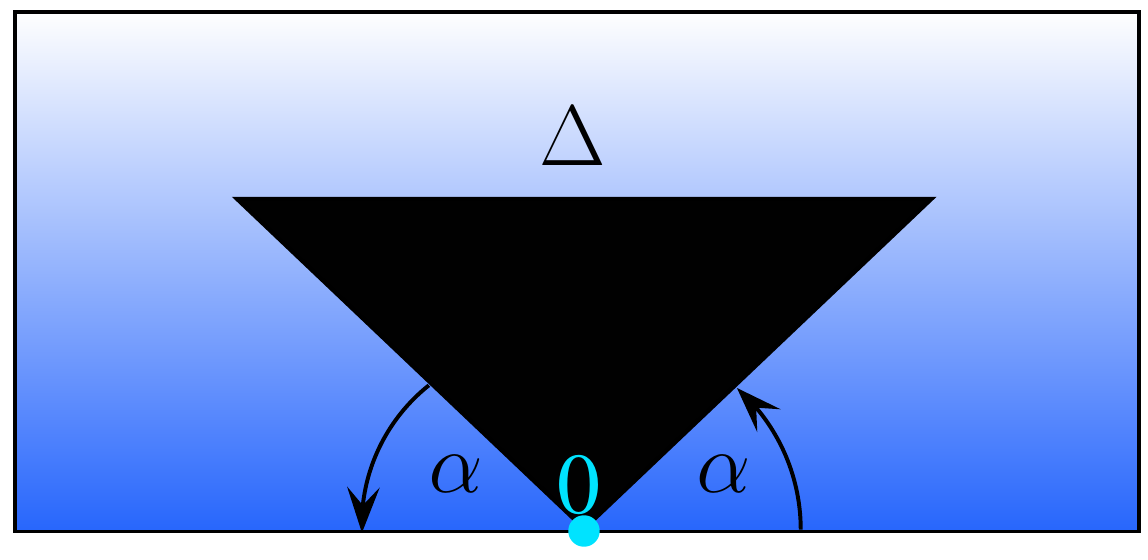}\caption{The triangle $\Delta.$ } \end{figure} 

Then it follows that $G_t\subset \Delta$ for all $t\in(0,t_0].$ \\ Now let $c= \lim_{z\to0^+}g_{\Delta}(z).$ If we denote by $I_t$ the cluster set of $G_t$ with respect to $g_{G_t},$ then $I_t\subset [-c,c]$. Likewise, if we denote by $J_t$ the cluster set of $G_t$ with respect to $g_{H_t}$ and 
\begin{eqnarray*}\gamma_t := \max\{|\lim_{z\to0^-} g_{\Delta \cup H_t}(z)|,|\lim_{z\to0^+}g_{\Delta \cup H_t}(z)|\}, \end{eqnarray*}
then $J_t\subset [-\gamma_t,\gamma_t].$
We have $\gamma_t\to c$ for $t\to 0.$ As $t^{-1/2}U(t)\in J_t$, we can conclude that $$ |t^{-1/2}U(t)|\leq \gamma_t $$
and consequently $\limsup_{t\downarrow0}|t^{-1/2}U(t)|\leq c<\infty.$
\end{proof}

The converse of  Proposition \ref{sec} is wrong, as the following counterexample shows.

\begin{proposition}
 There is a simple curve that does not approach $\R$ within a sector and its driving function $U$ satisfies $$\limsup_{t\downarrow0}\frac{|U(t)-U(0)|}{\sqrt{t}}<\infty.$$
\end{proposition}
\begin{proof}
 We construct the curve by connecting the set $\{i/2^n\with n\in\N_0\}$ of points with $\{1/2^n+i(1/2^n)^2 \with n\in\N_0\},$ namely: Start with $i$ and connect it with $1+i$ by a straight line segment. Then connect $1+i$ with $i/2$, $i/2$ with $1/2+i(1/2)^2,$ etc.\\
We get a simple curve $\gamma$ that starts from the real line in $0$ and does not approach $\R$ in a sector, because the curve $x\to x+ix^2$ approaches $\R$ tangentially. 
Assume that $\gamma(t)$ is parameterized by half-plane capacity with driving function $U:[0,T]\to\R.$ Let $\delta_t=(t/T)^{-1/2}\gamma(0,t],$ then $\hcap(\delta_t)=T/t\cdot \hcap(\gamma[0,t])=T/t\cdot 2t=2T.$ By construction of $\gamma$ we have $$\max_{z\in \delta_t}\Im(z)\asymp\max_{z\in \delta_t}\Re(z)$$ and as $\sup_{t\in(0,1]}\max_{z\in \delta_t}\Im(z)<\infty$ by Lemma \ref{imestim}, we also have $\sup_{t\in(0,1]}\max_{z\in \delta_t}\Re(z)<\infty.$ Hence, there exists $R>0$ such that $\delta_t\subset \mathcal{B}(0,R)$ for all $t\in(0,1]$ and thus there exists $K>0$ with $|(t/T)^{-1/2}U(t)|=|g_{\delta_t}(``\text{tip of}\; \delta_t\text{''})|\leq K.$
\end{proof}

Another example for this statement is the half-Sierpinski gasket described in \cite{spacefill}.

\subsection{Line approach}

Next we derive a sufficient condition for the hulls $K_t$ to approach $\R$ in $\varphi$-direction.

\begin{theorem}\label{line_approach}
Let $j\in\{1,...,n\}$ and let $\{K_t\}_{t\in[0,T]}$ be a family of hulls generated by equation (\ref{more2}). Then $K_t$ approaches $\R$ at $U_j(0)$ in $\varphi$-direction provided that $$ \lim_{h\downarrow0}\frac{U_j(h)-U_j(0)}{\sqrt{h}} =  
\frac{2\sqrt{\lambda_j(0)}(\pi-2\varphi)}{\sqrt{\varphi(\pi-\varphi)}}.
$$
\end{theorem}
\begin{proof}
By translation we can assume that $U_j(0)=0$. \\
We note that the following Loewner equation generates a straight line segment starting in $0$ with angle $\varphi$ (see Example \ref{Maxime} and change the time $t\mapsto \lambda_2(0)t$): $$\dot{g_t}(z)=\frac{2\lambda_j(0)}{g_t(z)-c\sqrt{t}}, \quad \text{with} \;\; c:=\frac{2\sqrt{\lambda_j(0)}(\pi-2\varphi)}{\sqrt{\varphi(\pi-\varphi)}}.$$
Let $L$ be the generated hull at time $t=1.$\\
Now let $d>0$, then the corresponding conformal mappings $g_t(z,d)$ for the scaled hulls $dK_t$ satisfy \begin{equation}\label{scal}\dot{g_t}(z,d)=\sum_{\begin{subarray}{c}
k=1\\
k\not=j
                             \end{subarray}}^n\frac{2\lambda_k(t/d^2)}{g_t(z,d)-dU_k(t/d^2)}+\frac{2\lambda_j(t/d^2)}{g_t(z,d)-dU_j(t/d^2)}, \quad g_0(z,d)=z.\end{equation} If we choose $d$ large enough, the corresponding hull at $t=1$ will always have $n$ connected components. Let $G_d$ be the one near $0.$ We will have to look at the limit case $d\to\infty.$\\
First, let $h_t(z,d)$ be the solution of the Loewner equation 
$$\dot{h_t}(z,d)=\frac{2\lambda_j(0)}{h_t(z,d)-dU_j(t/d^2)}, \quad h_0(z,d)=z,$$
and let $H_d$ be the generated hull at $t=1$.
Choose an $R>0$ and let $D_R:=\Ha\cap\{|z|<R\}.$
If we denote by $g(\cdot,d)\cara g$ the Carath\'{e}odory convergence for Loewner chains\footnote{$g(\cdot, d)\cara g(\cdot)$ if for every $\eps>0$ and every $S\in[0,T],$ $g_t(z,d)$ converges to $g_t(z)$ uniformly on $[0,S]\times\{z\in\Ha\with\operatorname{dist}(z,K_S)\geq\eps\}.$ \\[-2mm]}, defined in \cite{Lawler:2005}, p. 114, then we have $$g(\cdot,d) \cara g \;\; \text{for} \;\; d\to\infty \;\; \text{in} \;\; D_R \quad \text{if and only if} \quad h(\cdot,d) \cara g \;\; \text{for} \;\; d\to\infty \;\; \text{in} \:\; D_R.$$
This can be shown by using the fact that left summand of (\ref{scal}) converges uniformly to $0$ on $D_R\times[0,1]$ and that $\lambda_j(t/d^2)$ converges uniformly to $\lambda_j(0)$, again on $D_R\times[0,1].$ \\

Now, the hulls $K_t$ approach $\R$ at $0$ in $\varphi$-direction if and only if 
 $\Ha\setminus G_d\to \Ha\setminus L$ for $d\to\infty$ in the sense of kernel convergence. Because of the above relation of $g(\cdot,d)$ to $h(\cdot,d)$, this is equivalent to 

\begin{equation}\label{Jenny}\Ha\setminus H_d\to \Ha\setminus L \quad \text{for}\; d\to\infty. \end{equation}

Now suppose $U_j(t)=c\sqrt{t}+\Landauo(\sqrt{t})$, then $dU_j(t/d^2) \to c\sqrt{t}$ uniformly on $[0,1]$ for $d\to\infty$. Uniform convergence of driving functions generally implies kernel convergence of the corresponding domains by Proposition 4.47 in \cite{Lawler:2005}. Consequently, $\Ha \setminus H_d \to \Ha \setminus L$  for $d\to\infty$.\\
\end{proof}

The converse of Theorem \ref{line_approach} is not true in general as the following example shows.\footnote{The author would like to thank Huy Tran for pointing out the existence of such an example.}
\begin{proposition}
 There exists a driving function $U$ with $$\limsup_{t\downarrow0}\frac{U(t)}{\sqrt{t}}\not=0$$ such that the hulls $K_t$ generated by the one-slit equation approach $0$ in $\frac{\pi}{2}$--direction.
\end{proposition}
\begin{proof}
 Consider the region $R$ in $\Ha$ between the two curves $$\{x+iy\in\Ha\with x=y^2\} \quad \text{and} \quad \{x+iy\in\Ha\with x=-y^2\}.$$ Any hull inside this region clearly approaches $0$ in $\frac{\pi}{2}$--direction. Now, for $n\in\N,$ let $a_n$ be the intersection point of $L_n:=\{x+iy\in\Ha\with x=-y^{n+2}\}$ and $C_n:=\{x+\frac{1}{2^{n-1}}i\in\Ha\}$ and let $b_n$ be the intersection point of $R_n:=\{x+iy\in\Ha\with x=y^{n+2}\}$ and $C_n.$ First, we construct a curve $\hat{\gamma}:[0,1]\to\C$ in the following way:\\
Connect $0$ to $a_1$ via $L_1$ (i.e. by the subcurve of $L_1$ connecting 0 to $a_1$), then $a_1$ to $b_1$ via $C_1$ and $b_1$ to $0$ via $R_1$. Assume that $\hat{\gamma}:[0,1/2]\to\C$ describes this ``triangle'' $T_1.$\\
 Next we increase $n$ and construct another triangle $T_2$: connect $0$ to $a_2$ via $L_2,$ $a_2$ to $b_2$ via $C_2$ and then $b_2$ to $0$ via $R_2$. Now we may assume that $\hat{\gamma}:[1/2,3/4]\to\C$ parameterizes this curve.\\
Now we continue inductively and obtain a sequence $(T_n)$ of nested triangles, each parameterized by $\hat{\gamma}:[1-1/2^{n-1},1-1/2^n]\to\C.$ As  $$\diam(T_n)\to 0 \quad \text{for} \quad n\to\infty\quad \text{and} \quad 0\in T_n \quad \text{for all} \quad n\in\N,$$ we can extend $\hat{\gamma}$ continuously to the interval $[0,1]$ by requiring $\hat{\gamma}(1)=0.$\\
Now consider the curve $\hat{\gamma}(1-t)$ and let $\gamma$ be a parameterization of this curve such that $\hcap(K_t)=2t$, where we denote by $K_t$ the smallest hull containing $\gamma[0,t].$\\
The family $\{K_t\}_t$ satisfies the local growth property and thus, by Remark \ref{Lola}, it can be generated by the one-slit equation with a driving function $U.$\\ Finally, let $t_1>t_2>t_3>...$ be the decreasing sequence of zeros of $\gamma(t)$. We have $U(t_n)=\lim_{x\uparrow 0}g_{K_{t_n}}(x)$ or $U(t_n)=\lim_{x\downarrow 0}g_{K_{t_n}}(x).$ However, as $K_{t_n}$ is symmetric with respect to the imaginary axis, we certainly have $$2|U(t_n)|=\lim_{x\downarrow 0}g_{K_{t_n}}(x)-\lim_{x\uparrow 0}g_{K_{t_n}}(x)=:\pi\cdot \operatorname{cap}_\Ha(K_{t_n}).$$
The quantity $\operatorname{cap}_\Ha$ is introduced in \cite{Lawler:2005}, p. 73, see also the first equation on p. 74. There, it is shown that there exists a constant $c_1>0$ such that $$\operatorname{cap}_\Ha(K_{t_n}) \geq c_1 \cdot \diam(K_{t_n}),$$
see (3.14) on p. 74 in \cite{Lawler:2005}. As $\diam(K_{t_n})\geq c_2\cdot \sqrt{\operatorname{hsiz}(K_{t_n})}$ for another constant $c_2>0$ and $\text{hsiz}(K_{t_n}) \geq c_3 \cdot t_n$ for some $c_3>0$ by Theorem \ref{hsiz}, we arrive at 
$$|U(t_n)| = \pi/2 \cdot \operatorname{cap}_\Ha(K_{t_n}) \geq c_1 \cdot \pi/2 \cdot \diam(K_{t_n}) \geq c_1c_2\sqrt{c_3} \cdot \pi/2 \cdot \sqrt{t_n}.$$
As $U(t_n)>0$ for infinitely many $t_n,$ we conclude $$\limsup_{t\downarrow0}\frac{U(t)}{\sqrt{t}}>c_1c_2\sqrt{c_3} \cdot \pi/2.$$
\end{proof}

Under some nice conditions, however, the converse of Theorem \ref{line_approach} is true. As the line approach is equivalent to (\ref{Jenny}), we  need an additional condition that ensures that the kernel convergence of $\Ha\setminus H_d$ implies convergence of the driving functions. The main ingredient for this purpose is the following result by Lind, Marshall and Rohde.

\begin{theorem}[Theorem 4.3 in \cite{LindMR:2010}]\label{Uschi}
Let $\eps, c,R \in(0,\infty)$ and let $A_1,$ $A_2$ be two hulls with $\diam(A_j)\leq R$. Suppose that there exists a hull $B\supset A_1\cup A_2$ such that
$$ \dist(\zeta, A_1)<\eps\;\text{and}\; \dist(\zeta,A_2)<\eps \;\text{for all}\; \zeta \in B.  $$
Suppose further that there are curves $\sigma_j\subset \Ha\setminus A_j$ connecting a point $p\in\Ha\setminus B$  to $p_j\in A_j$ with $\diam \sigma_j \leq c\eps < \diam A_j,$ for $j=1,2.$ Then there exists a constant $c_0$ depending on $R$ only such that
$$ |g_{A_1}(p_1)-g_{A_2}(p_2)| \leq 2 c_0 \sqrt{\eps}(\sqrt{c}+\rho_{\hat{\C}\setminus \tilde{B}}(p,\infty)),$$
where $\tilde{B}=\overline{B\cup B'}\cup_j I_j$, $B'=\{\overline{z}\;|\; z\in B\},$ $I_j$ are the bounded intervals in $\R\setminus \overline{B\cup B'}$ and $\rho_\Omega(z_1,z_2)$ denotes the hyperbolic distance in $\Omega\subset\hat{\C}$ (with curvature $-1$). 
\end{theorem}

Whenever we can apply this Theorem to the hulls $H_d$ and $L,$ the converse of Theorem \ref{line_approach} holds. Next we give a simple geometric condition for this case.

\begin{theorem}\label{line_approach2}
Let $j\in\{1,...,n\}$ and let $\{K_t\}$ be a family of hulls generated by equation (\ref{more2}). Assume that $K_t$ approaches $\R$ at $U_j(0)$ in $\varphi$-direction. Furthermore, assume that there exist $\tau>0$ such that the connected component $C_\tau$ of $K_\tau$ near $U_j(0)$ is a slit having the property that the orthogonal projection $P:C_\tau \to \{U_j(0)+re^{i\varphi}\with r>0\}$ is  one-to-one. Then $$ \lim_{h\downarrow0}\frac{U_j(h)-U_j(0)}{\sqrt{h}} = 
\frac{2\sqrt{\lambda_j(0)}(\pi-2\varphi)}{\sqrt{\varphi(\pi-\varphi)}}                            $$
\end{theorem}
\begin{proof} Note that we only have to show that $\lim_{t\downarrow 0}\frac{U_j(t)}{\sqrt{t}}$ exists or $=\pm\infty$. The relation between the exact value of the limit and $\varphi$ was already proven in Theorem \ref{line_approach}. Assume that $U_j(0)=0$ and let $H_d$ and $L$ be defined as in the proof of Theorem \ref{line_approach}.

First we consider the case $n=1$. Note that $H_d$ is a part of $K_t$ scaled by $d$ in this case.\\

We know that \begin{equation}\label{star2}\Ha\setminus H_d\to \Ha\setminus L \quad \text{for} \quad d\to \infty \tag{$*$}\end{equation} and that the function $P:H_d\to \{re^{i\varphi}\with r>0\}$ is one-to-one for all $d$ large enough. Denote by $T(d)$ the tip of $H_d$ and by $T$ the tip of $L.$ The Carath\'{e}odory convergence (\ref{star2}) implies that $P(T(d))\to T$ and the injectivity of $P|_{H_d}$ that $T(d)\to T.$ For a given $\eps>0$ with \begin{equation}\label{star3}2\eps<|T|\tag{$**$}\end{equation} there exists $D>0$ such that $$H_d\subset B:=\{z\in\Ha\with |P(z)-z|<\eps/3, 0<|z| <|T|+\eps/3\}\;  \text{for all} \; d>D.$$
Furthermore we assume that $D$ is so large that  
\begin{equation}\label{star4}|T(d)-T|<\frac{\eps}{3} \quad \text{for all} \; d>D.\tag{$***$}\end{equation}
Now we would like to apply Theorem \ref{Uschi} with $A_1=L$, $A_2=H_d$ and $p_1=T,$ $p_2=T(d)$ and $B$ as defined above. A simple geometric consideration shows that
$$\dist(\zeta, L)<\eps \;\text{and}\; \dist(\zeta,H_d)<\eps \; \text{for all}\; \zeta \in B.$$ Next we let $p=T+\frac{4}{3}\eps e^{i\varphi}\in \Ha\setminus B.$ We can connect $p$ to $T$ by a straight line segment $\sigma_1.$ Then  $\diam(\sigma_1) = 4/3\eps \underset{\eqref{star3}}{<} |T| = \diam(A_1).$ As $P$ is one-to-one on $H_d,$ we can construct a curve $\sigma_2$ connecting $p$ to $T(d)$ in $\Ha\setminus H_d$ with $\diam(\sigma_2) \underset{\eqref{star4}}{\leq} 5/3\eps =2\eps -\eps/3 \underset{\text{(\ref{star3})}}{<}|T|-\eps/3\leq \diam(A_2).$
Let $c:=g_{L}(T),$ then Theorem \ref{Uschi} implies $$|g_{H_d}(T(d))-c|=\left|dU_1\left(\frac1{d^2}\right)-c\right|\leq 2c_0 \sqrt{\eps}\left(\sqrt{5/3}+\varrho_{\hat{\C}\setminus \tilde{B}}(p,\infty)\right).$$ 
With the explicit formula for the hyperbolic metric of a disc (see Lemma 4.6 in \cite{LindMR:2010}) we obtain 
 \begin{eqnarray*}&&\varrho_{\hat{\C}\setminus \tilde{B}}(p,\infty)\leq \varrho_{\hat{\C}\setminus \mathcal{B}(0,|T|+\frac{\eps}{3})}(p,\infty)
=2\log(1+\frac{2(|T|+\frac{\eps}{3})}{\eps})\\&<&2\log(1+\frac{4|T|}{\eps})=\LandauO(|\log \frac1{\eps}|)=\Landauo(\eps^{-1/2}) \quad \text{for} \; \eps \to 0.
\end{eqnarray*} Thus $|dU_1(\frac1{d^2})-c|\to 0$ for $d\to \infty,$ i.e. $\lim_{t\downarrow 0}\frac{U_1(t)}{\sqrt{t}}=c.$\\

Now let $n\geq 2.$ As in the proof of Theorem \ref{Michael}, for all $t\in[0,\tau],$ we let $C_t$ be the connected component of $K_t$ near $U_j(0),$ $h_t=g_{C_t},$ $x(t)=\hcap(C_t)$, $g_t=g_{K_t}$ and we define $H_t$ by $g_t=H_t\circ h_t.$ Then $x(t)$ is continuously differentiable and $\dot{h_t}=\frac{\dot{x}(t)}{h_t-V(t)}$ with a driving function $V.$ Our proof for the one-slit case implies that $\lim_{t\downarrow 0}\frac{V(t)}{\sqrt{x(t)}}$ exists and from $\frac{V(t)}{\sqrt{x(t)}}=\frac{V(t)}{\sqrt{t}}\cdot \sqrt{\frac{t}{x(t)}}$ and $\dot{x}(0)\not=0,$ we conclude that also $\lim_{t\downarrow 0}\frac{V(t)}{\sqrt{t}}$ exists. We have to show that the same is true for $U_j.$ \\
The function $[0,\tau]\times \mathcal{U}\ni(t,z)\mapsto F(t,z):=H_t(z)$ is continuously differentiable, where $\mathcal{U}$ is a sufficiently small neighborhood of $0$ in $\C.$ For $t$ small enough, we have $V(t)\in \mathcal{U}$ and from $U_j(t)=H_t(V(t))$ for all $t\in[0,\tau],$ we obtain 
\begin{eqnarray*}U_j(t) &=& U_j(t)-U_j(0)= \frac{\partial F}{\partial t}(0,0) \cdot t + \frac{\partial F}{\partial z}(0,0) \cdot V(t) + \Landauo(|t|+|V(t)|)\\ &=&\frac{\partial F}{\partial z}(0,0) \cdot  V(t) + \Landauo(\sqrt{t}) = V(t) + \Landauo(\sqrt{t}) \quad \text{for} \quad t\downarrow 0. \end{eqnarray*}
Thus, the limit $\lim_{t\downarrow 0}\frac{U_j(t)}{\sqrt{t}}$ exists whenever $\lim_{t\downarrow 0}\frac{V(t)}{\sqrt{t}}$ exists. 
\end{proof} 

 \begin{remark}
  The condition of the injectivity of the projection on the line is satisfied, e.g., when $\arg \gamma(t)$ is continuously differentiable with $(\arg\circ \gamma)'(0)=\varphi,$ where $\gamma$ parameterizes the slit $C_\tau$ from Theorem \ref{line_approach2}; see \cite{DongWu}, where the authors also derive further results concerning the line approach.
 \end{remark}

\section{Two further problems}

The fact that we can assign unique driving functions to slits provokes several questions of the form: How can property $X$ of slits be related to property $Y$ of the driving functions?\\

Exemplarily we mention the important property of smoothness, which has been discussed by several authors for the one-slit case: In \cite{MR0480952}, it is shown for the radial case that a slit is in $C^1$ provided that its driving function has a bounded derivative. C. Wong obtains several results concerning the regularity of a generated slit in \cite{Carto} (for the chordal case). Heuristically, his results can be summarized as: If the driving function is in $C^{\beta},$ then the generated slit is in $C^{\beta+1/2},$ where $\beta>1/2,$ see \cite[p. 3]{Carto}.\\
The converse problem, i.e. finding properties of the driving functions of a smooth slit, is discussed in \cite{MR1845014} (for the radial case). The authors prove that a $C^n$ slit has a $C^{n-1}$ driving function and that an analytic slit has a real analytic driving function.\\

It would be interesting to find more relations of this type. Next we mention two further problems, which are motivated by Theorem \ref{Charlie2} and its proof.

\subsection{Jordan hulls}\label{Jordan}
We have seen how to generate given disjoint slits with constant speeds. The corresponding Loewner equation then has a simple form because only the driving functions depend on  time. Furthermore, this equation is unique, so we exploited all degrees of freedom.\\
We call a compact hull $K$ \textit{Jordan hull} if the boundary $\partial K$ is a Jordan curve and $\partial K \cap \R=[a,b]$ is a closed interval with positive length, i.e. $a<b$. Can we find a similar unique Loewner equation for Jordan hulls? Here, ``similar'' means that we think of $K$ as a thickened slit. Let us assume that $\hcap(K)=2$ and recall the general Loewner equation (\ref{chordal}):

\begin{equation}\label{meas}
  \dot{g}_t(z)=\int_\R \frac{2\mu_t(du)}{g_t(z)-u}, \quad g_0(z)=z.                                                                                                          \end{equation}
We would like to find a family of measures $\{\mu_t\}_{t\in[0,1]}$ of the form $$\mu_t(A)=\mu(a(t)+c(t)\cdot A),$$ where $\mu$ is a probability measure, $a:[0,1]\to \R$ is a continuous (``driving'') function with $a(0)=0$ and $c:[0,1]\to [0,\infty)$ is a continuous (``scaling'') function with $c(0)=1.$\\
The initial values $a(0)=0$ and $c(0)=1$ are equivalent to $\mu=\mu_0.$ So, $\mu_t$ is nothing but $\mu$ stretched and translated. We expect that $\supp\mu \subset [a,b].$\\
In fact, it is easy to find those measures: We  can use the one-slit equation to generate the boundary -- either growing  from $a$ to $b$ or from $b$ to $a$ -- or the two--slit equation to generate the boundary by two slits growing from $a$ and $b$ and meeting in the upper half--plane. In the latter case, we have infinitely many possibilities as we can prescribe an arbitrary $\lambda\in (0,1)$ as the speed for one of the slits. \\  The slit equations generate the hull in a quite ``non-uniform'' way, as they let the boundary grow and then ``fill up'' the rest of $K$ at the end. In other words, we always have $\supp\mu \subsetneq [a,b].$ This leads to the following question.

\begin{question}
Let  $K$ be a Jordan hull with $\hcap(K)=2.$ Can we find  continuous functions $a:[0,1]\to\R,$ $c:[0,1]\to [0,\infty)$ with $a(0)=0$ and $c(0)=1$ and a probability measure $\mu$ with $$\supp \mu = \overline{K}\cap \R$$ such that the solution of equation (\ref{meas}) with $\mu_t(A)=\mu(a(t)+c(t)\cdot A)$ satisfies $g_1=g_K$?\\ If they exist, are $\mu, a$ and $c$ determined uniquely? 
\end{question}

\newpage
\subsection{A problem concerning multiple SLE}

We mention one further problem that arises if one thinks of the proof of Theorem \ref{Charlie2} and multiple SLE, roughly speaking: Does bang-bang of single SLE$(\kappa)$ converge to multiple SLE$(\kappa)$?\\
In \cite{MR2310306} the authors define multiple SLE as a generalization of the stochastic Loewner evolution (for one slit). Let us fix the points $0,1$ and $\infty$ and let $\kappa\in[0,4].$ Multiple SLE($\kappa$) from $(0,1)$ to $(\infty,\infty)$ within $\Ha$ can be viewed as a certain probability measure on the space of all pairs of disjoint curves that connect $0$ to $\infty$ and $1$ to $\infty$ in $\Ha,$ see \cite{MR2310306} for more details.\\
Fix $n\in\N$ and $\kappa\in[0,4].$ We construct two random curves $\Gamma_n$ and $\Delta_n$ in the following way:
\begin{enumerate}[(1)]
 \item 
Generate an SLE($\kappa$) curve in $\Ha$ starting from $0$ and going to $\infty$ by the stochastic Loewner differential equation and stop it when it has half-plane capacity $1/n.$ Call this random curve $\gamma_1.$
\item  Next, construct an SLE($\kappa$) curve within the domain $\Ha\setminus \gamma_1$ from $1$ to $\infty$. Stop it when it has half-plane capacity $1/n$ and call this curve $\delta_1.$ Clearly, $\gamma_1\cap \delta_1 = \emptyset$ w.p.1.
\item Continue this process: $\gamma_{k+1}$ is a part of an SLE($\kappa$) curve from the tip of $\gamma_k$ to $\infty$ within $\Ha\setminus (\bigcup_{j=1}^k \gamma_j \cup \bigcup_{j=1}^k \delta_j)$ having half-plane capacity $1/n;$ and $\delta_{k+1}$ is a part of an SLE($\kappa$) curve from the tip of $\delta_k$ to $\infty$ within $\Ha\setminus (\bigcup_{j=1}^{k+1} \gamma_j \cup \bigcup_{j=1}^k \delta_j)$ also having half-plane capacity $1/n$.
\item Finally, let $$ \Gamma_n=\bigcup_{k=1}^\infty\gamma_k, \qquad  \Delta_n=\bigcup_{k=1}^\infty\delta_k.$$
\end{enumerate}

\begin{figure}[h]
 \centering \includegraphics[width=105mm]{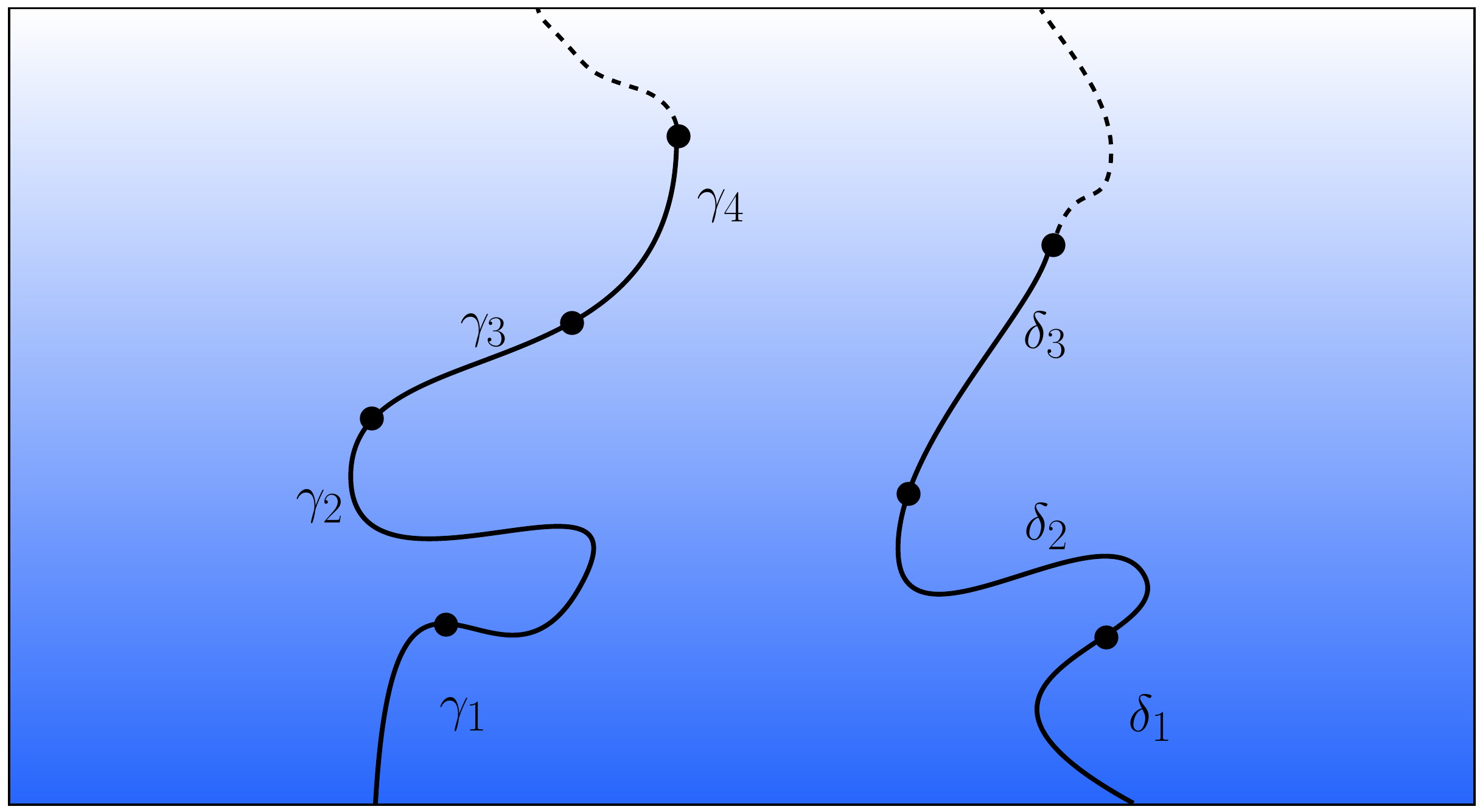}
\caption{$\Gamma_n$ and $\Delta_n$.}
\end{figure}

\begin{question}
 Does $(\Gamma_n,\Delta_n)$ converge to multiple SLE for $n\to\infty$ in some sense?
\end{question}

\textit{Update:} In general, the answer to the question is ``no'', see \cite{MR2358649}.

\newpage

\newcommand{\etalchar}[1]{$^{#1}$}
\providecommand{\bysame}{\leavevmode\hbox to3em{\hrulefill}\thinspace}
\providecommand{\MR}{\relax\ifhmode\unskip\space\fi MR }

\providecommand{\MRhref}[2]{%
  \href{http://www.ams.org/mathscinet-getitem?mr=#1}{#2}
}
\providecommand{\href}[2]{#2}


\begin{thebibliography}{CDMG10b}
\addcontentsline{toc}{chapter}{Bibliography}
\bibitem[AB11]{MR2887104}
Leandro Arosio and Filippo Bracci, \emph{Infinitesimal generators and the
  {L}oewner equation on complete hyperbolic manifolds}, Anal. Math. Phys.
  \textbf{1} (2011), no.~4, 337--350.

\bibitem[Aba92]{MR1174816}
Marco Abate, \emph{The infinitesimal generators of semigroups of holomorphic
  maps}, Ann. Mat. Pura Appl. (4) \textbf{161} (1992), 167--180.

\bibitem[ABFW]{ABWold}
Leandro Arosio, Filippo Bracci, and Erlend Fornaess~Wold, \emph{Solving the
  {L}oewner {P}{D}{E} in complete hyperbolic starlike domains of {$\C^N$}},
  eprint arXiv:1207.2721.

\bibitem[ABHK13]{MR3043148}
Leandro Arosio, Filippo Bracci, Hidetaka Hamada, and Gabriela Kohr, \emph{An
  abstract approach to {L}oewner chains}, J. Anal. Math. \textbf{119} (2013),
  no.~1, 89--114.

\bibitem[ABW]{filtering}
Leandro Arosio, Filippo Bracci, and Erlend~F. Wold, \emph{Embedding univalent
  functions in filtering {L}oewner chains in higher dimension}, eprint
  arXiv:1306.6759.

\bibitem[AL92]{MR1185588}
Erik Anders{\'e}n and L{\'a}szl{\'o} Lempert, \emph{On the group of holomorphic
  automorphisms of {${\bf C}^n$}}, Invent. Math. \textbf{110} (1992), no.~2,
  371--388.

\bibitem[Ale76]{MR0480952}
I.~A. Aleksandrov, \emph{Parametricheskie prodolzheniya v teorii odnolistnykh
  funktsii}, Izdat. ``Nauka'', Moscow, 1976.

\bibitem[Aro12]{Arosio:2012}
L.~Arosio, \emph{Basins of attraction in loewner equations}, Ann. Acad. Sci.
  Fenn. Math. \textbf{37} (2012), 563--570.

\bibitem[Bau05]{MR2107849}
Robert~O. Bauer, \emph{Chordal {L}oewner families and univalent {C}auchy
  transforms}, J. Math. Anal. Appl. \textbf{302} (2005), no.~2, 484--501.

\bibitem[BBH{\etalchar{+}}98]{MR1608852}
S.~R. Bell, J.-L. Brylinski, A.~T. Huckleberry, R.~Narasimhan, C.~Okonek,
  G.~Schumacher, A.~Van~de Ven, and S.~Zucker, \emph{Complex manifolds},
  Springer-Verlag, Berlin, 1998, Corrected reprint of the 1990 translation
  [{Several complex variables. VI}, Encyclopaedia, Math. Sci., 69, Springer,
  Berlin, 1990].

\bibitem[BBK05]{MR2187598}
Michel Bauer, Denis Bernard, and Kalle Kyt{\"o}l{\"a}, \emph{Multiple
  {S}chramm-{L}oewner evolutions and statistical mechanics martingales}, J.
  Stat. Phys. \textbf{120} (2005), no.~5-6, 1125--1163.

\bibitem[BCDM09]{MR2507634}
Filippo Bracci, Manuel~D. Contreras, and Santiago D{\'{\i}}az-Madrigal,
  \emph{Evolution families and the {L}oewner equation. {II}. {C}omplex
  hyperbolic manifolds}, Math. Ann. \textbf{344} (2009), no.~4, 947--962.

\bibitem[BCDM10]{MR2578602}
\bysame, \emph{Pluripotential theory, semigroups and boundary behavior of
  infinitesimal generators in strongly convex domains}, J. Eur. Math. Soc.
  (JEMS) \textbf{12} (2010), no.~1, 23--53.

\bibitem[BCDM12]{MR2995431}
\bysame, \emph{Evolution families and the {L}oewner equation {I}: the unit
  disc}, J. Reine Angew. Math. \textbf{672} (2012), 1--37.

\bibitem[BF08]{MR2357634}
Robert~O. Bauer and Roland~M. Friedrich, \emph{On chordal and bilateral {SLE}
  in multiply connected domains}, Math. Z. \textbf{258} (2008), no.~2,
  241--265.

\bibitem[Bis07]{MR2373370}
Christopher~J. Bishop, \emph{Conformal welding and {K}oebe's theorem}, Ann. of
  Math. (2) \textbf{166} (2007), no.~3, 613--656.

\bibitem[BP78]{MR0480965}
Earl Berkson and Horacio Porta, \emph{Semigroups of analytic functions and
  composition operators}, Michigan Math. J. \textbf{25} (1978), no.~1,
  101--115.

\bibitem[BSa]{BoehmSchl}
Christoph B\"ohm and Sebastian Schlei{\ss}inger, \emph{Constant {C}oefficients
  in the {R}adial {K}omatu-{L}oewner {E}quation for {M}ultiple {S}lits}, eprint
  arXiv:1311.2279.

\bibitem[BSb]{bbig}
Filippo Bracci and David Shoikhet, \emph{Boundary behavior of infinitesimal
  generators in the unit ball}, eprint arXiv:1203.1839.

\bibitem[Car03]{MR2004294}
John Cardy, \emph{Stochastic {L}oewner evolution and {D}yson's circular
  ensembles}, J. Phys. A \textbf{36} (2003), no.~24, L379--L386.

\bibitem[CDMGa]{localdu}
Manuel~D. Contreras, Santiago D{\'{\i}}az-Madrigal, and Pavel Gumenyuk,
  \emph{Local duality in {L}oewner equations}, eprint arXiv:1202.2334.

\bibitem[CDMGb]{annulus1}
\bysame, \emph{Loewner theory in annulus {I}: evolution families and
  differential equations}, eprint arXiv:1011.4253.

\bibitem[CDMGc]{annulus2}
\bysame, \emph{Loewner theory in annulus {II}: Loewner chains}, eprint
  arXiv:1105.3187.

\bibitem[CDMG10a]{MR2719792}
\bysame, \emph{Geometry behind chordal {L}oewner chains}, Complex Anal. Oper.
  Theory \textbf{4} (2010), no.~3, 541--587.

\bibitem[CDMG10b]{MR2789373}
\bysame, \emph{Loewner chains in the unit disk}, Rev. Mat. Iberoam. \textbf{26}
  (2010), no.~3, 975--1012.

\bibitem[CM02]{MR1945279}
L.~Carleson and N.~Makarov, \emph{Laplacian path models}, J. Anal. Math.
  \textbf{87} (2002), 103--150, Dedicated to the memory of Thomas H. Wolff.

\bibitem[CP03]{Cid:2003}
J.~Angel Cid and Rodrigo~L\'{o}pez Pouso, \emph{On first-order ordinary
  differential equations with nonnegative right-hand sides}, Nonlinear Anal.
  \textbf{52} (2003), no.~8, 1961--1977.

\bibitem[DG60]{MR0148939}
Ferdinand Docquier and Hans Grauert, \emph{Levisches {P}roblem und {R}ungescher
  {S}atz f\"ur {T}eilgebiete {S}teinscher {M}annigfaltigkeiten}, Math. Ann.
  \textbf{140} (1960), 94--123.

\bibitem[Dub07]{MR2358649}
Julien Dub{\'e}dat, \emph{Commutation relations for {S}chramm-{L}oewner
  evolutions}, Comm. Pure Appl. Math. \textbf{60} (2007), no.~12, 1792--1847.

\bibitem[Dur83]{MR708494}
Peter~L. Duren, \emph{Univalent functions}, Grundlehren der Mathematischen
  Wissenschaften [Fundamental Principles of Mathematical Sciences], vol. 259,
  Springer-Verlag, New York, 1983.

\bibitem[DV11]{PhysRevE.84.051602}
Miguel~A. Dur\'an and Giovani~L. Vasconcelos, \emph{Fingering in a channel and
  tripolar {L}oewner evolutions}, Phys. Rev. E \textbf{84} (2011), 051602.

\bibitem[DWar]{DongWu}
Han Dong and Hua Wu, \emph{Driving functions and traces of the {L}oewner
  equation}, Sci. China Math. (to appear).

\bibitem[EE01]{MR1845014}
Clifford~J. Earle and Adam~Lawrence Epstein, \emph{Quasiconformal variation of
  slit domains}, Proc. Amer. Math. Soc. \textbf{129} (2001), no.~11,
  3363--3372.

\bibitem[Eli11]{elin}
Mark Elin, \emph{Extension operators via semigroups}, Journal of Mathematical
  Analysis and Applications \textbf{377} (2011), 239--250.

\bibitem[FP85]{MR792819}
Carl~H. FitzGerald and Ch. Pommerenke, \emph{The de {B}ranges theorem on
  univalent functions}, Trans. Amer. Math. Soc. \textbf{290} (1985), no.~2,
  683--690.

\bibitem[FR93]{MR1213106}
Franc Forstneri{\v{c}} and Jean-Pierre Rosay, \emph{Approximation of
  biholomorphic mappings by automorphisms of {${\bf C}^n$}}, Invent. Math.
  \textbf{112} (1993), no.~2, 323--349.

\bibitem[FS77]{MR0435441}
John~Erik Fornaess and Edgar~Lee Stout, \emph{Polydiscs in complex manifolds},
  Math. Ann. \textbf{227} (1977), no.~2, 145--153.

\bibitem[GB92]{MR1201130}
V.~V. Goryainov and I.~Ba, \emph{Semigroup of conformal mappings of the upper
  half-plane into itself with hydrodynamic normalization at infinity}, Ukrain.
  Mat. Zh. \textbf{44} (1992), no.~10, 1320--1329.

\bibitem[GdM]{GM13}
Pavel Gumenyuk and Andrea del Monaco, \emph{Chordal loewner equation}, eprint
  arxiv:1302.0898v2.

\bibitem[GHKK08]{MR2425737}
Ian Graham, Hidetaka Hamada, Gabriela Kohr, and Mirela Kohr, \emph{Parametric
  representation and asymptotic starlikeness in {$\mathbb C^n$}}, Proc. Amer.
  Math. Soc. \textbf{136} (2008), no.~11, 3963--3973.

\bibitem[GHKK12]{MR2943779}
\bysame, \emph{Extreme points, support points and the {L}oewner variation in
  several complex variables}, Sci. China Math. \textbf{55} (2012), no.~7,
  1353--1366.

\bibitem[GK03]{graham2003geometric}
I.~Graham and G.~Kohr, \emph{Geometric function theory in one and higher
  dimensions}, Pure and Applied Mathematics, Taylor \& Francis, 2003.

\bibitem[GKK11]{greene2011geometry}
R.E. Greene, K.T. Kim, and S.G. Krantz, \emph{The geometry of complex domains},
  Progress in Mathematics - Birkh{\"a}user, Birkhauser Verlag GmbH, 2011.

\bibitem[GKP07]{MR2341607}
Ian Graham, Gabriela Kohr, and John~A. Pfaltzgraff, \emph{Parametric
  representation and linear functionals associated with extension operators for
  biholomorphic mappings}, Rev. Roumaine Math. Pures Appl. \textbf{52} (2007),
  no.~1, 47--68.

\bibitem[Goo73]{Goodman}
Gerald~S. Goodman, \emph{Control {T}heory in {T}ransformation {S}emigroups},
  Geometric Methods in System Theory, Proceedings NATO Advanced Study Institute
  (1973), 215--226.

\bibitem[Gra07]{Graham2007}
K.~Graham, \emph{On multiple schramm--loewner evolutions}, Journal of
  Statistical Mechanics: Theory and Experiment \textbf{2007} (2007), no.~03.

\bibitem[GS08]{MR2495460}
T.~Gubiec and P.~Szymczak, \emph{Fingered growth in channel geometry: a
  {L}oewner-equation approach}, Phys. Rev. E (3) \textbf{77} (2008), no.~4,
  041602, 12.

\bibitem[Ham91]{MR1139801}
D.~H. Hamilton, \emph{Generalized conformal welding}, Ann. Acad. Sci. Fenn.
  Ser. A I Math. \textbf{16} (1991), no.~2, 333--343.

\bibitem[HM84]{hallenbeck1984linear}
D.J. Hallenbeck and T.H. MacGregor, \emph{Linear problems and convexity
  techniques in geometric function theory}, Monographs and studies in
  mathematics, Pitman, 1984.

\bibitem[Ken07]{MR2348786}
Tom Kennedy, \emph{A fast algorithm for simulating the chordal
  {S}chramm-{L}oewner evolution}, J. Stat. Phys. \textbf{128} (2007), no.~5,
  1125--1137.

\bibitem[Ken09]{MR2570752}
\bysame, \emph{Numerical computations for the {S}chramm-{L}oewner evolution},
  J. Stat. Phys. \textbf{137} (2009), no.~5-6, 839--856.

\bibitem[KK11]{MR2768636}
Shulim Kaliman and Frank Kutzschebauch, \emph{On the present state of the
  {A}nders\'en-{L}empert theory}, Affine algebraic geometry, CRM Proc. Lecture
  Notes, vol.~54, Amer. Math. Soc., Providence, RI, 2011, pp.~85--122.

\bibitem[KL07]{MR2310306}
Michael~J. Kozdron and Gregory~F. Lawler, \emph{The configurational measure on
  mutually avoiding {SLE} paths}, Universality and renormalization, Fields
  Inst. Commun., vol.~50, Amer. Math. Soc., Providence, RI, 2007, pp.~199--224.

\bibitem[Koh01]{MR1929522}
Gabriela Kohr, \emph{Using the method of {L}\"owner chains to introduce some
  subclasses of biholomorphic mappings in {${\bf C}^n$}}, Rev. Roumaine Math.
  Pures Appl. \textbf{46} (2001), no.~6, 743--760 (2002).

\bibitem[Kom43]{KomatuZweifach}
Yusaku Komatu, \emph{{Untersuchungen \"uber konforme {A}bbildung von zweifach
  zusammenh\"angenden {G}ebieten.}}, Proc. Phys. Math. Soc. Japan, III. Ser.
  \textbf{25} (1943), 1--42 (German).

\bibitem[Kom50]{Komatu}
\bysame, \emph{{On conformal slit mapping of multiply-connected domains.}},
  Proc. Japan Acad. \textbf{26} (1950), no.~7, 26--31 (English).

\bibitem[KSS68]{MR0257336}
P.~P. Kufarev, V.~V. Sobolev, and L.~V. Spory{\v{s}}eva, \emph{A certain method
  of investigation of extremal problems for functions that are univalent in the
  half-plane}, Trudy Tomsk. Gos. Univ. Ser. Meh.-Mat. \textbf{200} (1968),
  142--164.

\bibitem[Kuf43]{MR0013800}
P.~P. Kufarev, \emph{On one-parameter families of analytic functions}, Rec.
  Math. [Mat. Sbornik] N.S. \textbf{13(55).} (1943), 87--118.

\bibitem[Kuf47]{MR0023907}
\bysame, \emph{A remark on integrals of {L}\"owner's equation}, Doklady Akad.
  Nauk SSSR (N.S.) \textbf{57} (1947), 655--656.

\bibitem[Law05]{Lawler:2005}
G.~F. Lawler, \emph{Conformally invariant processes in the plane}, Mathematical
  Surveys and Monographs, vol. 114, American Mathematical Society, Providence,
  RI, 2005.

\bibitem[Lin05]{Lind:2005}
J.~R. Lind, \emph{A sharp condition for the {L}oewner equation to generate
  slits}, Ann. Acad. Sci. Fenn. Math. \textbf{30} (2005), no.~1, 143--158.

\bibitem[LLN09]{MR2576752}
S.~Lalley, G.~F. Lawler, and H.~Narayanan, \emph{Geometric interpretation of
  half-plane capacity}, Electron. Commun. Probab. \textbf{14} (2009), 566--571.

\bibitem[LMR10]{LindMR:2010}
J.~Lind, D.~E. Marshall, and S.~Rohde, \emph{Collisions and spirals of
  {L}oewner traces}, Duke Math. J. \textbf{154} (2010), no.~3, 527--573.

\bibitem[L{\"o}w23]{Loewner:1923}
K.~L{\"o}wner, \emph{Untersuchungen {\"u}ber schlichte konforme abbildungen des
  einheitskreises. i}, Mathematische Annalen \textbf{89} (1923), no.~1,
  103--121.

\bibitem[LR]{spacefill}
Joan Lind and Steffen Rohde, \emph{Spacefilling {C}urves and {P}hases of the
  {L}oewner {E}quation}, eprint arXiv:1103.0071.

\bibitem[LSW01a]{MR1849257}
Gregory~F. Lawler, Oded Schramm, and Wendelin Werner, \emph{The dimension of
  the planar {B}rownian frontier is {$4/3$}}, Math. Res. Lett. \textbf{8}
  (2001), no.~4, 401--411.

\bibitem[LSW01b]{MR1879850}
\bysame, \emph{Values of {B}rownian intersection exponents. {I}. {H}alf-plane
  exponents}, Acta Math. \textbf{187} (2001), no.~2, 237--273.

\bibitem[LV73]{MR0344463}
O.~Lehto and K.~I. Virtanen, \emph{Quasiconformal mappings in the plane},
  second ed., Springer-Verlag, New York, 1973, Translated from the German by K.
  W. Lucas, Die Grundlehren der mathematischen Wissenschaften, Band 126.

\bibitem[Mat55]{MR0080931}
Takeshi Matsuno, \emph{On star-like theorems and convexlike theorems in the
  complex vector space}, Sci. Rep. Tokyo Kyoiku Daigaku Sect. A. \textbf{5}
  (1955), 88--95.

\bibitem[MR05]{MarshallRohde:2005}
D.~E. Marshall and S.~Rohde, \emph{The {L}oewner differential equation and slit
  mappings}, J. Amer. Math. Soc. \textbf{18} (2005), no.~4, 763--778.

\bibitem[MS01]{MR1845017}
Jerry~R. Muir, Jr. and Ted~J. Suffridge, \emph{Unbounded convex mappings of the
  ball in {${\mathbb{C}}^n$}}, Proc. Amer. Math. Soc. \textbf{129} (2001),
  no.~11, 3389--3393.

\bibitem[MS06]{MR2254484}
\bysame, \emph{Extreme points for convex mappings of {$B_n$}}, J. Anal. Math.
  \textbf{98} (2006), 169--182.

\bibitem[MS07]{MR2272135}
\bysame, \emph{A generalization of half-plane mappings to the ball in
  {$\mathbb{C}^n$}}, Trans. Amer. Math. Soc. \textbf{359} (2007), no.~4,
  1485--1498.

\bibitem[Pes36]{0016.03501}
E.~Peschl, \emph{{Zur {T}heorie der schlichten {F}unktionen.}}, J. Reine Angew.
  Math. \textbf{176} (1936), 61--94 (German).

\bibitem[Pet07]{MR2262775}
Han Peters, \emph{Perturbed basins of attraction}, Math. Ann. \textbf{337}
  (2007), no.~1, 1--13.

\bibitem[Poi07]{Poi}
H.~Poincar{\'e}, \emph{Les fonctions analytiques de deux variables et la
  repr{\'e}sentation conforme}, Rend. Circ. Mat. Palermo \textbf{23} (1907),
  185--220.

\bibitem[Pom66]{MR0206245}
Ch. Pommerenke, \emph{On the {L}oewner differential equation}, Michigan Math.
  J. \textbf{13} (1966), 435--443.

\bibitem[Pom75]{Pom:1975}
\bysame, \emph{Univalent functions}, Vandenhoeck \& Ruprecht, G\"ottingen,
  1975.

\bibitem[Por87a]{MR1049182}
T.~Poreda, \emph{On the univalent holomorphic maps of the unit polydisc in
  {${\bf C}^n$} which have the parametric representation. {I}. {T}he
  geometrical properties}, Ann. Univ. Mariae Curie-Sk\l odowska Sect. A
  \textbf{41} (1987), 105--113 (1989).

\bibitem[Por87b]{MR1049183}
\bysame, \emph{On the univalent holomorphic maps of the unit polydisc in {${\bf
  C}^n$} which have the parametric representation. {II}. {T}he necessary
  conditions and the sufficient conditions}, Ann. Univ. Mariae Curie-Sk\l
  odowska Sect. A \textbf{41} (1987), 115--121 (1989).

\bibitem[Pro93]{Prokhorov:1993}
D.V. Prokhorov, \emph{Reachable set methods in extremal problems for univalent
  functions}, Saratov University, 1993.

\bibitem[Ric66]{MR0210889}
Seppo Rickman, \emph{Characterization of quasiconformal arcs}, Ann. Acad. Sci.
  Fenn. Ser. A I No. \textbf{395} (1966).

\bibitem[Rot98]{Roth}
O.~Roth, \emph{Control {T}heory in ${H}(\mathbb{D})$}, Ph.D. thesis, University
  of {W}\"urzburg, 1998.

\bibitem[RR88]{MR929658}
Jean-Pierre Rosay and Walter Rudin, \emph{Holomorphic maps from {${\bf C}^n$}
  to {${\bf C}^n$}}, Trans. Amer. Math. Soc. \textbf{310} (1988), no.~1,
  47--86.

\bibitem[RR94]{MR1307384}
Marvin Rosenblum and James Rovnyak, \emph{Topics in {H}ardy classes and
  univalent functions}, Birkh\"auser Advanced Texts: Basler Lehrb\"ucher.
  [Birkh\"auser Advanced Texts: Basel Textbooks], Birkh\"auser Verlag, Basel,
  1994.

\bibitem[RSa]{RothSchl}
Oliver Roth and Sebastian Schlei{\ss}inger, \emph{Rogosinski's lemma for
  univalent functions, hyperbolic {A}rchimedean spirals and the {L}oewner
  equation}, eprint arXiv:1311.0677.

\bibitem[RSb]{RothSchl2}
\bysame, \emph{The {S}chramm--{L}oewner equation for multiple slits}, eprint
  arXiv:1311.0672.

\bibitem[Sch00]{MR1776084}
Oded Schramm, \emph{Scaling limits of loop-erased random walks and uniform
  spanning trees}, Israel J. Math. \textbf{118} (2000), 221--288.

\bibitem[Sol13]{sola}
Alan Sola, \emph{Elementary examples of loewner chains generated by densities},
  Annales Universitatis Mariae Curie-Sklodowska Sectio A (2013), to appear.

\bibitem[Suf70]{MR0261040}
T.~J. Suffridge, \emph{The principle of subordination applied to functions of
  several variables}, Pacific J. Math. \textbf{33} (1970), 241--248.

\bibitem[Tra]{HuyTran}
Huy Tran, \emph{Convergence of an algorithm simulating {L}oewner curves},
  eprint arXiv:1303.3685.

\bibitem[Vod11]{Voda}
Mircea~Iulian Voda, \emph{Loewner {T}heory in {S}everal {C}omplex {V}ariables
  and {R}elated {P}roblems}, Ph.D. thesis, University of {T}oronto, 2011.

\bibitem[Wer59]{MR0121500}
John Wermer, \emph{An example concerning polynomial convexity}, Math. Ann.
  \textbf{139} (1959), 147--150 (1959).

\bibitem[Wol08]{MR2372737}
Erlend~Forn{\ae}ss Wold, \emph{A {F}atou-{B}ieberbach domain in {$\mathbb C^2$}
  which is not {R}unge}, Math. Ann. \textbf{340} (2008), no.~4, 775--780.

\bibitem[Won]{Carto}
Carto Wong, \emph{Smoothness of {L}oewner {S}lits}, eprint arXiv:1201.5856.

\end{thebibliography}
\end{document}